\documentclass[10pt]{article}

\usepackage[margin=1in]{geometry}
\usepackage{amsmath,amsfonts,amssymb,amsthm,MnSymbol}
\usepackage{bm}
\usepackage{float}
\usepackage{graphicx}
\usepackage[section]{algorithm}
\usepackage{algpseudocode}

\newcommand{\ds}{\displaystyle}
\newcommand{\myspace}{\\[0.3cm]}
\newcommand{\eqand}[1]{\quad \text{#1} \quad}
\newcommand{\eqfor}{\quad \text{for }}
\newcommand\numberthis{\addtocounter{equation}{1}\tag{\theequation}}
\newcommand{\di}[1]{\, d#1}

\newcommand{\myvector}[2]{\begin{bmatrix} #1 \\ #2 \end{bmatrix}}
\newcommand{\myvectorw}[2]{\begin{bmatrix} #1 \\[0.1cm] #2 \end{bmatrix}}

\newcommand{\ha}[1]{\widehat{#1}}
\newcommand{\la}[1]{\bm{\mathrm{#1}}}
\newcommand{\ti}[1]{\widetilde{#1}}

\newcommand{\bg}{\mathbf{g}}
\newcommand{\hbg}{\widehat{\bg}}

\newcommand{\hw}{\widehat{w}}
\newcommand{\hx}{\widehat{x}}

\newcommand{\bG}{\la{G}}
\newcommand{\bM}{\la{M}}
\newcommand{\bMt}{\ti{\la{M}}}
\newcommand{\br}{\la{r}}

\newcommand{\btA}{\ti{\la{A}}}
\newcommand{\bhC}{\ha{\la{C}}}
\newcommand{\btu}{\ti{\la{u}}}
\newcommand{\bhw}{\ha{\la{w}}}

\newcommand{\bDhat}{\widehat{\la{D}}}

\newcommand{\bcO}{\la{\mathcal{O}}}
\newcommand{\bcT}{\la{\mathcal{T}}}

\newcommand{\cF}{{\cal{F}}}
\newcommand{\cL}{\mathcal{L}}
\newcommand{\cQ}{\mathcal{Q}}

\newcommand{\ou}{\overline{u}}
\newcommand{\oU}{\overline{U}}
\newcommand{\ougs}{\ou^{\mathrm{GS}}}
\newcommand{\ow}{\overline{w}}
\newcommand{\oW}{\overline{W}}
\newcommand{\owgs}{\ow^{\mathrm{GS}}}
\newcommand{\obmu}{\overline{\la{\mu}}}
\newcommand{\obomega}{\overline{\la{\omega}}}

\newcommand{\tA}{\ti{A}}
\newcommand{\tC}{\ti{C}}

\newcommand{\tb}{\ti{b}}
\newcommand{\tq}{\widetilde{q}}

\newcommand{\tu}{{\widetilde{u}}}

\newcommand{\tw}{{\widetilde{w}}}

\newcommand{\tv}{\ti{v}}
\newcommand{\tx}{{\widetilde{x}}}
\newcommand{\tmu}{\widetilde{\mu}}

\newcommand{\tomega}{\widetilde{\omega}}

\newcommand{\bfR}{\mathbb{R}}
\newcommand{\cLout}{\begin{bmatrix}
                         0 & -\tv \Lt^T \\
                         \frac{1}{\tv}\Lt & 0
                     \end{bmatrix}}

\newcommand{\dtau}{\partial_{\tau}}
\newcommand{\g}{\gamma}
\newcommand{\G}{\la{\Gamma}}
\newcommand{\ghat}{\widehat{\g}}
\newcommand{\Ghat}{\ha{\G}}

\newcommand{\Lt}{{L_{\tau}}}
\newcommand{\Skn}{\la{S}_k^{(n)}}
\newcommand{\Ts}{T^{(2)}}
\newcommand{\ukn}{\tu_k^{(n)}}
\newcommand{\wkn}{\tw_k^{(n)}}
\newcommand{\xmax}{x_{\max}}
\newcommand{\txmax}{\ti{x}_{\max}}
\newcommand{\ymax}{y_{\max}}
\newcommand{\td}{\text{---}}

\newcommand{\der}[3]{\dfrac{\partial^{#1} #2}{\partial #3^{#1}}}

\newcommand{\oovs}[2]{\left\langle #1, #2 \right\rangle_{1/v^2}}

\newcommand{\ltrn}[2]{\left\langle #1, #2 
    \right\rangle_{l^2(\mathbb{R}^n)}}

\newcommand{\normltrtn}[1]{\left\| #1
    \right\|_{l^2(\mathbb{R}^{2n})}}
\newcommand{\usuh}[2]{\left\langle #1, #2 \right\rangle_{U^*U}}

\newcommand{\ipyt}[2]{\left\langle #1, #2 \right\rangle_{y,\theta}}
\newcommand{\ipxit}[2]{\left\langle #1, #2 \right\rangle_{\xi,\theta}}
\newcommand{\ootv}[2]{\left\langle #1, #2 \right\rangle_{1/\tv}}
\newcommand{\normootv}[1]{\left\| #1 \right\|_{1/\tv}}
\newcommand{\iptv}[2]{\left\langle #1, #2 \right\rangle_{\tv}}
\newcommand{\normiptv}[1]{\left\| #1 \right\|_{\tv}}
\newcommand{\ootviptv}[2]{\left\langle #1, #2
    \right\rangle_{1/\tv, \tv}}

\newcommand{\Lttx}[2]{\left\langle #1, #2 \right\rangle_{L^2[0,\txmax]}}

\newcommand{\ipgh}[2]{\left\langle #1, #2 \right\rangle_{\ghat}}

\newcommand{\llrr}[2]{\left\llangle #1, #2 \right\rrangle}

\newcommand{\rank}{\mathrm{rank \ }}
\newcommand{\Range}{\mathrm{Range \ }}
\newcommand{\diag}{\mathrm{diag \ }}
\newcommand{\const}{\mathrm{const. \ }}

\numberwithin{equation}{section}

\newtheorem{theorem}{Theorem}[section]
\newtheorem{definition}[theorem]{Definition}
\newtheorem{lemma}[theorem]{Lemma}
\newtheorem{proposition}[theorem]{Proposition}
\newtheorem{remark}[theorem]{Remark}
\newtheorem{problem}[theorem]{Problem}

\usepackage{url}
\usepackage{hyperref}

\title{Direct, nonlinear inversion algorithm for hyperbolic problems via
projection-based model reduction}

\author{Vladimir Druskin\footnotemark[2]%
\and Alexander Mamonov\footnotemark[3]%
\and Andrew E. Thaler\footnotemark[4]%
\and Mikhail Zaslavsky\footnotemark[2]}

\begin{document}
\maketitle

\renewcommand{\thefootnote}{\fnsymbol{footnote}}

\footnotetext[2]{Schlumberger--Doll Research Center, Cambridge, MA, USA
02139}
\footnotetext[3]{University of Houston, Houston, TX, USA 77004}
\footnotetext[4]{Institute for Mathematics and its Applications,
    University of Minnesota, College of Science and Engineering,
    Minneapolis, MN, USA 55455.  The research of AET was supported in
    part by Schlumberger and the Institute for Mathematics and its
Applications with funds provided by the National Science Foundation
through NSF Award 0931945.} 

\begin{abstract}
    We estimate the wave speed in the acoustic wave equation from
    boundary measurements by constructing a reduced-order model (ROM)
    matching discrete time-domain data. The state-variable
    representation of the ROM can be equivalently viewed as a Galerkin
    projection onto the Krylov subspace spanned by the snapshots of the
    time-domain solution.  The success of our algorithm hinges on the
    data-driven Gram--Schmidt orthogonalization of the snapshots that
    suppresses multiple reflections and can be viewed as a discrete form
    of the Marchenko--Gel'fand--Levitan--Krein algorithm.  In particular,
    the orthogonalized snapshots are localized functions, the (squared)
    norms of which are essentially weighted averages of the wave speed.
    The centers of mass of the squared orthogonalized snapshots provide
    us with the grid on which we reconstruct the velocity. This grid is
    weakly dependent on the wave speed in traveltime coordinates, so the
    grid points may be approximated by the centers of mass of the
    analogous set of squared orthogonalized snapshots generated by a
    known reference velocity. We present results of inversion
    experiments for one- and two-dimensional synthetic models.
\end{abstract}

\\ 


\pagestyle{myheadings}
\thispagestyle{plain}
\markboth{Druskin, Mamonov, Thaler, Zaslavsky}{Inversion via 
projection-based model reduction}


\section{Introduction}\label{sec:intro}

In seismic reflection tomography, one attempts to utilize measurements
of elastic waves to create an (approximate) image of a region in the
earth's subsurface. In this paper, we present a nonlinear tomographic
inversion method that can be placed within the so-called full waveform
inversion (FWI) framework.  Full waveform inversion algorithms employ
the full equations of motion and utilize as much of the information
contained in the recorded waveforms as possible to image the material
properties of the region of interest \cite{Fichtner:2011:FSW}. 

The most common numerical approach to FWI is nonlinear optimization,
i.e.,  minimization of the  misfit between the measured elastic field
and the forward model --- see, e.g.,
\cite{Virieux:2009:AOF,Fichtner:2011:FSW} (and the references within).
The images created via the optimization approach tend to have high
resolution;   however, the conventional FWI optimization procedure
suffers from a few computational and theoretical difficulties.  First,
the equations and models are typically discretized on a fine grid to
ensure the synthetic data sets are accurately computed --- the model
parameters tend to be on the order of billions \cite{Fichtner:2011:FSW}.
Even with the help of adjoint-state methods, the solution to 3D FWI
problems can take days or weeks of processing time.  The second
difficulty with the optimization problem is that the quadratic misfit
functional is nonconvex and has many local minima
\cite{Fichtner:2011:FSW}.  Gradient-based algorithms will tend to get
stuck in one of these local minima, rather than the true minimum, unless
the initial model is extremely close to the true model.  Several
approaches have been developed to mitigate the effects of the
nonconvexity of the misfit functional  --- see \cite{Virieux:2009:AOF,
Fichtner:2011:FSW} and the references therein --- though they come at
a cost.

Another, direct, nonlinear approach originated from several celebrated
works by Marchenko, Krein, Gel'fand, and Levitan  (MKGL)
\cite{Marchenko:1950:SPT, Gelfand:1955:ODD, Krein:1951:SIS,
Krein:1953:OTF, Krein:1954:OME, Levitan:1964:DDE}. The main idea of this
approach is the reduction of the inverse problem to a nonlinear integral
equation with Volterra (triangular) structure that can be solved
explicitly. It yields a very powerful tool for inverse hyperbolic
problems in 1D   \cite{Gopinath:1971:ITE,
Sondhi:1971:DVT,Symes:1979:IBV,Burridge:1980:GLM,
Santosa:1982:NSI,Habashy:1991:GGL} (and the references therein).  The
main difficulty involved in the application of this layer-stripping-type
approach in the multidimensional setting is the fact that the scattering
data is overdetermined.  Recently, progress was made in extending the
Marchenko and Gel'fand--Levitan approaches to 2D and 3D settings, see,
e.g., \cite{Kabanikhin:2004:DMS, Wapenaar:2013:TDS}, though more work
must be done to improve the \textit{lateral} resolution of the images in
each layer.  We also point out the related work by Bube and Burridge
\cite{Bube:1983:ODI}, in which the authors solve the 1D problem by
deriving a finite-difference scheme that corresponds exactly to a
continuum problem with a piece-wise constant coefficient.  

In this paper we apply the discrete MKGL approach (that can be expressed
via the Lanczos algorithm well known in the linear algebra community)
within the reduced-order model (ROM) framework.  The ROM is obtained by
matching discrete time-domain data and its finite-difference
interpretation yields a data-driven discretization scheme. 

Reduced-order models recently became popular tools for the solution
of \textit{frequency-domain, diffusion-dominated} inverse problems, such
as diffusive optical tomography, the quasi-stationary Maxwell equations,
etc.\ \cite{deSturler:2013:NPI,Druskin:2007:OCM}. The system's order was
reduced by projecting the original system onto a pre-computed or
dynamically-updated basis of frequency-domain solutions, and then using
the projected system as a fast proxy in the optimization process. A 
subspace size sufficient for accurate approximation of the forward
solver is critical for the success of the method.

As we shall see, the MKGL approach applied within the ROM not only
allows us to obtain images directly without optimization, but also to
compute sufficiently accurate ROMs with a single Galerkin basis
obtained for a background (e.g., constant coefficient) model.


\subsection{Reduced-order models and optimal grids}\label{subsec:ROM}

Our inversion algorithm employs a projection-based ROM.  In model order
reduction, one replaces a large-scale problem with a smaller, more
computationally efficient model that retains certain features of the
larger model --- see, e.g., the review article by Antoulas and Sorensen
\cite{Antoulas:2001:ALS} and the book by Antoulas
\cite{Antoulas:2005:ALS} (and the references therein).  

We now describe in some detail a particular ROM that is closely related
to the model we construct in this paper.  Consider the following
one-dimensional problem for $x \in [0,1]$:
\begin{equation}\label{eqn:model_problem} u''(x) - \lambda u(x) = 0,
    \qquad u'(0) = -1, \ u(1) = 0, 
\end{equation} 
where $\lambda \in \mathbb{C}\setminus ]-\infty,0[$ is a constant.  The
\emph{impedance function}, also known as the \emph{Neumann-to-Dirichlet
map}, \emph{Poincar\'e--Steklov operator}, or \emph{Weyl function}, is
defined by 
\[
    f(\lambda) \equiv u(0).
\]
We wish to construct a small discrete model (a ROM) that accurately
computes the impedance function $f(\lambda)$ for, say, $\lambda \in
\left[\lambda_1,\lambda_2\right] \subset [0,\infty[$.  

To that end, we consider the staggered grid (see Figure~\ref{fig:grid}
in \S~\ref{subsec:leapfrog} in the appendix):
\[
    0 = x_1 = \ha{x}_0 < \ha{x}_1 < x_2 < \ha{x}_2 < \cdots < 
        \ha{x}_{N-1} < x_N \le 1;
\]
the stepsizes are $h_j \equiv x_{j+1} - x_j$ and $\ha{h}_j \equiv
\ha{x}_j - \ha{x}_{j-1}$ for $j = 1,\ldots,N$.  A three-term
finite-difference approximation of \eqref{eqn:model_problem} on this
grid is \cite{Druskin:1999:GSR}
\[
    \begin{aligned}
        &\frac{1}{\ha{h}_j} \left[\frac{U_{j+1} - U_j}{h_j} - 
            \frac{U_j - U_{j-1}}{h_{j-1}}\right] - \lambda U_j = 0, 
            \qquad j = 2,3,\ldots,N \\
            &\frac{1}{\ha{h}_1}\left(\frac{U_2 - U_1}{h_1}\right) - 
            \lambda U_1 = -\frac{1}{\ha{h}_1}, \\
        &U_{N+1} = 0,
    \end{aligned}
\]
where $U_j \approx u(x_j)$.  This may be written in matrix form as
\[
    \la{A}\la{U} + \lambda\la{U} = -\frac{1}{\ha{h}_1}\la{e}_1,
\]
where $\la{A} \in \mathbb{R}^{N\times N}$, $\la{U} \in \mathbb{R}^N$,
and $\la{e}_1 \in \mathbb{R}^N$ contains a $1$ in its first component
and zeros elsewhere.  The \emph{discrete impedance function} is then
defined by 
\[
    f_N(\lambda) \equiv U_1 \approx u(0) = f(\lambda).
\]
The goal is to choose the stepsizes $h_j$, $\ha{h}_j$ in such a way that
$f_N(\lambda)$ is an excellent approximation $f(\lambda)$ with $N$
small.  

For example, if the grid spacing is uniform and $N \gg 1$, $\la{U}$ will
be a good approximation to $u$ over the entire interval $[0,1]$; in
particular, $f_N(\lambda)$ will be a good approximation to $f(\lambda)$.
However, if we are only interested in obtaining a good approximation to
the solution at $x = 0$ (i.e., the impedance function), taking $N \gg 1$
is inefficient.  A proper reduced-order model should have
$f_N(\lambda)$ very close to $f(\lambda)$ for $N$ small.

As Kac and Krein observed \cite{Kac:1974:OSF}, the discrete impedance
function $f_N$ may be written as a Stieltjes continued fraction
\cite{Stieltjes:1995:RSL} with the grid steps $h_j$, $\ha{h}_j$ as
coefficients; in particular, 
\begin{equation*}
    f_N(\lambda) =  \cfrac{1}{\ha{h}_1\lambda
                  + \cfrac{1}{h_1
                  + \cfrac{1}{\ha{h_2}\lambda
                  + \cdots
                  + \cfrac{1}{h_{N-1}
                  + \cfrac{1}{\ha{h}_N\lambda
                  + \cfrac{1}{h_N} } } } } }.
\end{equation*}
If the grid steps are judiciously chosen, $f_N$ will be
a Pad\'e approximant of $f$ and therefore converge to $f$ exponentially
as $N \rightarrow \infty$ \cite{Druskin:1999:GSR, Ingerman:2000:OFD,
Druskin:2002:TPF}.  In other words, $f_N$ will be an excellent
approximation to $f$ even if $N$ is quite small.  These grids are thus
known in the literature as \emph{optimal grids}, and have been
successfully applied in other related contexts as well
\cite{Druskin:2000:GSR, Asvadurov:2000:ADG}.  There is also an intimate
connection between optimal grids and the Galerkin method.  In
particular, to every $N$-term Galerkin approximation there corresponds a
stable three-term finite-difference scheme of no more than $N$ nodes
that has the same impedance function \cite{Druskin:2002:TPF}; we will
exploit a similar idea when we construct our ROM based on Galerkin
projection.  Finally, optimal grids have been generalized to
variable-coefficient Sturm--Liouville problems as well
\cite{Borcea:2002:OFD}.

Optimal grids have also been applied to inverse Sturm--Liouville
problems \cite{Borcea:2002:OFD}.  Their usefulness in inverse problems
stems from the fact that optimal grids are \emph{weakly dependent} on
the variable coefficients of the problem.  This extraordinary property
allows one to use the optimal grids constructed for the constant
coefficient Sturm--Liouville problem \eqref{eqn:model_problem} as the
grids in an inversion algorithm \cite{Borcea:2002:OFD}, and has also
been used in the context of inverse spectral problems
\cite{Borcea:2005:OCL} and electrical impedance tomography
\cite{Borcea:2008:EIT, Borcea:2010:PRN}.  This idea of the weak
dependence of optimal grids on the PDE coefficients plays a crucial role
in our inversion algorithm as well, although we should emphasize that it
only holds in traveltime coordinates in the context of the wave equation
(whereas it holds in physical coordinates in the case of
Sturm--Liouville problems).


\subsection{Direct inversion algorithm for FWI in 1D}\label{subsec:ours}

To fix the idea, let us consider the one-dimensional acoustic wave
equation on $[0,\xmax] \times [0,T]$:
\begin{equation*}
    -u_{xx}(x,t) + \frac{1}{v^2(x)}u_{tt}(x,t) = 0, 
        \qquad u(x,0) = b(x), \ u_t(x,0) = 0,
\end{equation*}
subject to appropriate boundary conditions at $x = 0$ and $x = \xmax$.
The goal of the forward problem is to determine $u$ for $t \in [0,T]$
given the wave speed $v$ and the source distribution $b$ (which we assume
is a smooth approximation of the delta function).  We study the inverse
problem of estimating $v$ given the source distribution $b$ and $2n$
equally-spaced samples of the time-domain \emph{transfer function} 
\[
    f(t) \equiv \int_{0}^{\xmax} b(x) u(x,t) \frac{1}{v^2(x)} \di{x}
    \approx \frac{1}{v^2(0)}u(0,t).
\]
In other words, we are given $b$ and $f_k \equiv f(k\tau)$ for $k =
0,\ldots,2n-1$ and a timestep $\tau > 0$ and wish to approximate the
wave speed $v$ in the interior of the domain $[0,\xmax]$.  We will see
that the choice of $\tau$ plays a crucial role in the quality of the
inversion results, but we can typically take $\tau$ to be near the
Nyquist--Shannon limit of the cutoff frequency of $b$.  The transfer
function $f$ is called the single-input/single-output (SISO) transfer
function in control theory terminology, implying that it was obtained
via single-source (input) and single-receiver (output) measurements.   

The core of our inversion algorithm is essentially a discrete version of
the Krein--Gel'fand--Levitan--Marchenko method \cite{Marchenko:1950:SPT,
Gelfand:1955:ODD, Krein:1951:SIS, Krein:1953:OTF, Krein:1954:OME,
Levitan:1964:DDE}; also see the works by Gopinath and Sondhi
\cite{Gopinath:1971:ITE, Sondhi:1971:DVT}, Symes \cite{Symes:1979:IBV},
Burridge \cite{Burridge:1980:GLM}, Santosa \cite{Santosa:1982:NSI}, and
Habashy \cite{Habashy:1991:GGL}
for more on the Gel'fand--Levitan method in the continuous case.  A
summary of our application of this method is as follows.  We consider
the $2n$ time-domain \emph{snapshots} 
\begin{equation*}
    u_k(x) \equiv u(x,k\tau) \eqfor k = 0,\ldots,2n-1,
\end{equation*}
and we define a ``matrix'' $U$ of the first $n$ snapshots, i.e., 
\[
    U \equiv \left[u_0(x), \ldots, u_{n-1}(x) \right].
\]
Because $b(x)$ is an approximation of the delta function, it is
localized near $x = 0$.  Then, due to causality, the matrix $U$ will be
an approximation of an upper triangular matrix (reminiscent of the
``upper triangular'' kernel from Gel'fand--Levitan theory
\cite{Gelfand:1955:ODD}).  We may orthogonalize the snapshots via the
Gram-Schmidt process and obtain the $QR$ decomposition $U =
V\bm{\mathcal{R}}$. Since $U$ is already approximately upper triangular,
the ``matrix'' $V$ of the orthogonalized snapshots will be an
approximation of the identity matrix, i.e., the orthogonalized snapshots
are localized.  In physical terms, orthogonalization suppresses multiple
reflections. 

Unfortunately, we do not have access to the true snapshot matrix $U$
because the wave speed $v$ is unknown (so the snapshots are also
unknown).  However, as we discuss in \S~\ref{sec:inversion}, in
traveltime coordinates the centers of mass of the squared
orthogonalized snapshots are \emph{weakly dependent} on the wave speed
$v$.  Thus we compute the snapshots $u^0_k(x)$ corresponding to a
reference velocity $v^0$, which we typically take to be constant.  After
orthogonalization, the centers of mass of the reference squared
orthogonalized snapshots approximate the centers of mass of the true
squared orthogonalized snapshots, and, hence, provide us with a grid
for inversion.  (This is similar to the weak dependence of the grid on
the parameters in \cite{Borcea:2002:OFD}.)

In our approach, we orthogonalize the snapshots via the Lanczos
algorithm without normalization.  In this case, the (squared) norm of
each orthogonalized snapshot contains information about the magnitude of
$v$ near the center of mass of the squared orthogonalized snapshot; thus
the orthogonalized snapshots not only provide us with a grid for
inversion, but they also provide us with knowledge about the wave speed
on that grid.  

The crucial feature of our orthogonalization process is that, depending
the available data, the computation of these norms can be performed in
two isomorphically equivalent ways.  If the velocity, and, hence, the
snapshots, are known, the norms are computed explicitly in the Lanczos
algorithm.  On the other hand, if only the time-domain data is
available, we show that the norms correspond to parameters of a ROM that
interpolates the discretely sampled time-domain data.  In fact, this
data-driven, projection-based ROM corresponds to the Galerkin method on
a (Krylov) subspace spanned by the snapshots and may be constructed
solely from the discrete time-domain data.  The spectral coefficients of
the Galerkin approximation satisfy a three-term finite-difference
recursion that reproduces the data $f_k$ exactly, and the coefficients
of the finite-difference matrix are related to the norms of the
orthogonalized snapshots in a simple way. (For more on the construction
of ROM based on projection onto polynomial and rational Krylov
subspaces, see the book by Antoulas \cite{Antoulas:2005:ALS} and the
paper by de Villemagne and Skelton \cite{deVillemagne:1987:MRU};
Gallivan, Grimme, and Van Dooren \cite{Gallivan:1996:ARL} and Grimme
\cite{Grimme:1997:KPM} discuss the relationship between model order
reduction via Krylov projection and rational interpolation.)  

We should also discuss the important work of Bube and Burridge
\cite{Bube:1983:ODI}, in which the authors solve the 1D inversion
problem using a finite-difference scheme and Cholesky factorization.
Our method also involves a finite-difference scheme and a Cholesky
factorization (see Remark~\ref{rem:QR}), but the fundamental difference
between our finite-difference scheme and that of Bube and Burridge is
that ours is equivalent to Galerkin projection onto the space of
orthogonalized snapshots.  Indeed, the novelty of the ROM approach
discussed in this paper is data-driven Galerkin discretization that
yields localization of the basis functions.

In summary, our algorithm may be outlined as follows:
\begin{enumerate}
    \item Record the data $f_k = f(k\tau)$ for $k = 0,\ldots,2n-1$ and
        $\tau$ near the Nyquist limit.
    \item Compute the snapshots $u^0_k(x) = u^0(x,k\tau)$ corresponding
        to the reference velocity $v^0(x)$ (typically we take $v^0(x)
        \equiv v(0)$ for all $x \in [0,\xmax]$).
    \item Orthogonalize the snapshots $u^0_k$ via the Lanczos process
        (equivalently, the Gram--Schmidt procedure) --- the grid nodes
        $\tx_j$ (in traveltime coordinates) we use for our inversion are
        given by the centers of mass of these squared reference
        orthogonalized snapshots.
    \item From the recorded data $f_k$, construct
        the projection-based ROM that interpolates $f_k$ for $k = 0,
        1, \ldots,2n-1$. Use it to compute the norms of the true
        orthogonalized snapshots.
    \item The estimate of the velocity at the grid point $x_j$ is
        proportional to a ratio of the norms of the
        $j\textsuperscript{th}$ true and reference orthogonalized
        snapshots.
\end{enumerate}

Since our algorithm is direct, it avoids the difficulties associated
with iterative gradient-based algorithms that we described earlier.  In
particular, our algorithm cannot become trapped in a local minimum.
Additionally, we only need to solve a single forward problem (to compute
the reference snapshots in step $2$), and the reference velocity for
this forward problem is typically very simple (e.g., constant).
Finally, one may use our algorithm as a direct imaging algorithm (as we
do in this paper), or as a nonlinear preconditioner (similar to that in
\cite{Borcea:2014:MRA}) which generates a reasonable initial model $m^0$
close to the true model $m$ that can be used in least-squares
optimization.  

The remainder of our paper is organized as follows.  In
\S~\ref{sec:problem}, we define the problem.  We discuss the
orthogonalization of the snapshots in \S~\ref{sec:continuum}.
Construction of our data-driven, interpolatory ROM, based on Galerkin
projection onto the Krylov subspace spanned by the snapshots, is
discussed in \S~\ref{sec:transformation}.  We develop our inversion
algorithm in \S~\ref{sec:inversion} and demonstrate it via several
numerical experiments in \S~\ref{sec:numerics}.  We describe a
two-dimensional extension of our algorithm in \S~\ref{sec:2D}.  Detailed
proofs of many of the lemmas are given in the appendix.


\section{Problem formulation}\label{sec:problem}

We start with the Cauchy problem for the Green's function for the
one-dimensional wave equation on $[0,\xmax]\times [0,\infty[$:
\begin{equation}\label{eqn:w1} 
    Ag+g_{tt}=0,  \qquad  
    g_x|_{x=0}=0, \ g|_{x=\xmax}=0, \quad 
    g|_{t=0}=\delta(x+0), \ g_t|_{t=0}=0,
\end{equation}
where
\[
    A \equiv -v^2 \frac{d^2}{dx^2}
\]
with the Neumann--Dirichlet boundary conditions from \eqref{eqn:w1}, and
the wave speed $v(x)$ is a regular enough, positive function on
$[0,\xmax]$.  Here and throughout the paper, $\delta(x+0)$ denotes the
``right-half'' Dirac delta function and satisfies the normalization
\[
    \int_0^{\infty} \delta(x+0) \di{x} = 1.
\]

We study the inverse problem of determining $v(x)$ from the
boundary data $g|_{x=0}$.  For regular enough boundary data and
for all $x\in [0,\xmax]$ there is a unique mapping associating the data
$g(0,t)$ to the velocity $v(x)$ where $t\in[0,2\tx(x)]$ and the
\emph{slowness (traveltime) coordinate transformation} is
\begin{equation}\label{eqn:slowness} 
    \tx(x) \equiv \int_0^x \frac{1}{v(x')}\di{x'};
\end{equation} 
see, e.g., \cite{Gelfand:1955:ODD, Krein:1951:SIS, Krein:1953:OTF,
Krein:1954:OME, Levitan:1964:DDE, Gopinath:1971:ITE, Sondhi:1971:DVT,
Burridge:1980:GLM}.

The Cauchy problem \eqref{eqn:w1} can be equivalently rewritten on
$[0,\xmax]\times ]-\infty, \infty[$ as
\begin{equation}\label{eqn:w2}
    Ag+g_{tt}=\delta(x+0)\delta(t)_t, 
    \qquad  g_x|_{x=0}=0, \  g|_{x=\xmax}=0, \quad 
    g|_{t<0}=0.
\end{equation}

We introduce the weighted inner product $\llangle
\cdot, \cdot \rrangle$ on $L_2[0,\xmax]$, defined by
\begin{equation}\label{eqn:iprod2}
    \left\llangle u, w \right\rrangle  \equiv 
    \int_{0}^{\xmax}  u(x) w(x) \frac{1}{v^2(x)} \di{x}.
\end{equation}
We note $A$ is self adjoint and positive definite with respect to
$\left\llangle \cdot, \cdot \right\rrangle$; real functions of $A$
(continuous on the spectrum of $A$) are
self adjoint with respect to this weighted inner product as well.  

The solution of \eqref{eqn:w1} can be formally written via an operator
function as
\begin{equation}\label{eqn:specdec} 
    g(x,t) =\cos\left(t\sqrt{A}\right)\delta(x+0) = 
    \int_0^{\infty} \cos\left(t\sqrt{\lambda}\right)\rho(x,\lambda) 
    \di{\lambda},
\end{equation}
where 
\begin{equation}\label{eqn:rho}
    \rho(x,\lambda) \equiv \sum_{l=1}^\infty \delta(\lambda-\lambda_l) 
        \frac{z_l(0)}{v(0)^2}z_l(x)
\end{equation}
is the vector spectral measure associated with $A$ and $\left(\lambda_l,
z_l(x)\right)$ are eigenpairs of $A$. Here we have used the fact that,
because $A$ is self-adjoint with respect to the inner product
$\llangle\cdot,\cdot\rrangle$, the eigenfunctions can be chosen to be
orthonormal, i.e., $\llangle z_l,z_k\rrangle = \delta_{lk}$ where
$\delta_{lk}$ is the Kronecker delta.  Note also that the eigenfunctions
$z_l$ must satisfy the homogeneous Neumann--Dirichlet boundary
conditions associated with $A$, namely $z_{l,x}(0) = 0$ and $z_l(\xmax)
= 0$.

We use the Green's function from \eqref{eqn:w1} to study a problem with
a variable source wavelet $q(t)_t$ (in place of $\delta(t)_t$ in
\eqref{eqn:w2}).  We assume $q\in L_1]-\infty,\infty[$ is an even,
sufficiently smooth approximation of $\delta(t)$ with nonnegative
Fourier transform
\begin{equation}\label{eqn:Fq}
    \tq\left(s^2\right) \equiv \cF(q)(s) = 
        \int_{0}^\infty 2\cos(ts) q(t) \di{t}.
\end{equation}
To fix the idea, we use the Gaussian 
\begin{equation}\label{eqn:gaussian}
    q(t) = \frac{1}{\sigma\sqrt{2\pi}} 
    \exp\left(-\frac{t^2}{2\sigma^2}\right)
\end{equation}
for some $\sigma>0$; in this case, 
\begin{equation}\label{eqn:gaussian_FT}
    \tq\left(s^2\right) = \exp\left(-\frac{\sigma^2s^2}{2}\right).
\end{equation}
We should choose $\sigma$ to be small so that \eqref{eqn:gaussian} gives
a good approximation to $\delta(t)$.  Physically, $\sigma$ gives a
measure of the duration of the source wavelet $q(t)_t$ in time, and, as
can be seen from \eqref{eqn:gaussian_FT}, $\sigma$ is inversely related
to the bandwidth of this wavelet\footnote{Strictly speaking, the
Gaussian pulse has an infinite bandwidth; however, for all practical
purposes, the decay of $\widetilde{q}$ is rapid enough that we may
speak of an ``effective bandwidth,'' namely values of $s^2$ beyond which
$\widetilde{q}(s^2)$ is sufficiently small.}.  As we will see (most
prominently in \S~\ref{sec:numerics}), the time-domain measurement
sampling rate is closely related to $\sigma$.

This choice of $q$ yields the equation 
\[
    A\ha{g} + \ha{g}_{tt} = 
        \delta(x+0)q(t)_t, \qquad  
    \ha{g}_x|_{x=0}=0, \ \ha{g}|_{x=\xmax}=0, 
    \quad \lim_{t\to -\infty}\ha{g}=0
\]
on $[0,\xmax]\times ]-\infty ,\infty[$.  The solution to this equation
can be written via a convolution integral as
\begin{equation}\label{eqn:g_convolution} 
    \ha{g}(x,t) = \int_{-\infty}^{\infty} H(t-t')g(x,t-t')q(t') \di{t'},
\end{equation}
where the Green's function $g$ solves \eqref{eqn:w1} and $H$ is the
Heaviside step function.

Let $\ha{u}(x,t) \equiv \ha{g}(x,t)+\ha{g}(x,-t)$.  Then, using $\ha{g}
=\cF^{-1} \left[\cF (Hg)\cF(q)\right]$ (which follows from
\eqref{eqn:g_convolution} and the convolution theorem for Fourier
transforms) and \eqref{eqn:Fq}, we obtain
\begin{equation}\label{eqn:Fe}
    \ha{u}(x,t) = \frac{2}{\pi}\int_{0}^\infty \cos(ts)\Re 
        \left[\cF (Hg)\cF(q)\right] \di{s}
    = \frac{2}{\pi}\int_{0}^\infty \cos(ts)\Re 
        \left[\cF (Hg)\right]\tq\left(s^2\right) \di{s}.
\end{equation}
For $q=\delta(t)$, from \eqref{eqn:w2}, \eqref{eqn:specdec},
and \eqref{eqn:g_convolution} we have
\[
    \ha{u}(x,t)=g(x,t)=2\int_0^{\infty} \cos(ts)\rho\left(x,s^2\right)s
\di{s}
\]
for $t\ge 0$.  Comparing this with \eqref{eqn:Fe} (and taking
$\tq\left(s^2\right)=1$), we find $\Re \left[\cF (Hg)\right] = \pi
\rho\left(x,s^2\right) s$.  Combining this with \eqref{eqn:Fe}, for
general $q$ we have
\begin{equation}
    \begin{split}\label{eqn:ufun}
        \ha{u}(x,t) 
	    &= 2\int_{0}^\infty\cos(ts)\rho\left(x,s^2\right)s
	        \tq\left(s^2\right)\di{s} \\
            &= \int_{0}^\infty\cos\left(t\sqrt{\lambda}\right)
	        \rho(x,\lambda) 
    	        \tq(\lambda)\di{\lambda}\\ 
	    &= \cos\left(t\sqrt{A}\right)\tq(A)\delta(x+0).
    \end{split}
\end{equation}
This implies $\ha{u}$ solves the following Cauchy
problem  on $[0,\xmax]\times [0,\infty[$:
\begin{equation}\label{eqn:cauchyb}
    A\ha{u}+\ha{u}_{tt}=0, \qquad  
    \ha{u}_x|_{x=0}=0, \ \ha{u}|_{x=\xmax}=0, \quad 
    \ha{u}|_{t=0}=\tq(A)\delta(x+0), \ \ha{u}_t|_{t=0}=0.
\end{equation}
 
Our measurements are defined for $t \in [0,T]$ by $f(t) \equiv
\ha{u}(0,t)$.  In practice, we only take measurements at the discrete
times $k\tau$ for $k = 0,\ldots,2n-1$, where $(2n-1)\tau = T$ and $\tau$
is the sampling timestep.  We choose a time discretization step
$\tau > 0$ consistent with the Nyquist--Shannon sampling of the cutoff
frequency of $\tq$, i.e., we take $\tau \sim \sigma$.  Our goal is to
solve the following problem.
\begin{problem}\label{eqn:problem_1}
    Estimate $v|_{[0, \tx^{-1}(T)]}$ from $f_k \equiv \ha{u}(0,
    k\tau)$, $k=0,\ldots,2n-1$,  provided  $\tx^{-1}(T)\le\xmax$.
\end{problem}

We will see that the choice of $\tau$ influences the quality of the
inversion results.  


\section{Continuum interpretation}\label{sec:continuum}

The solution \eqref{eqn:ufun} at the discrete times $k\tau$ is 
\begin{equation}\label{eqn:snapshot}
    \begin{aligned}
        \ha{u}(x, k \tau) &= \cos\left(k \tau \sqrt{A}\right)\ti{q}(A) 
	    \delta(x+0) \\
        &= \cos\left(k\arccos \cos\left(\tau\sqrt{A}\right)\right) 
	    \tq(A) \delta(x+0) \\
        &= T_k\left(\cos\left(\tau\sqrt{A}\right)\right) \tq(A) 
	    \delta(x+0),
    \end{aligned}
\end{equation}
where $T_k$ is the $k\textsuperscript{th}$ Chebyshev polynomial of the
first kind.

We define the propagation operator
$P\equiv\cos\left(\tau\sqrt{A}\right)$.  Then, from the spectral
representation \eqref{eqn:ufun}, we can equivalently rewrite
\eqref{eqn:snapshot} as
\begin{equation}\label{eqn:u_eta} 
    \ha{u}(x,k\tau) = T_k(P)\tq(A)\delta(x+0) = 
        \int_{-1}^1 T_k(\mu) \eta(x,\mu) \di{\mu},
\end{equation}
where 
\begin{multline}\label{eqn:eta}
    \eta(x,\mu) \equiv 2\sum_{j=-\infty}^\infty \mathrm{sgn}(j)
        \tq \left(\frac{(\arccos(\mu) +2j\pi)^2}{\tau^2} \right) 
        \frac{\arccos(\mu)+2j\pi}{\tau^2} \cdot \\ 
        \rho \left(x,\frac{(\arccos(\mu) + 2j\pi)^2}{\tau^2}\right) 
        \frac{1}{\sqrt{1-\mu^2}}
\end{multline}
and we take $\mathrm{sgn}(0) \equiv 1$;  the infinite summation is due
to the multiplicity of $\arccos$ (see
\S~\ref{subsec:derivation_of_eta} in the appendix for a derivation
of \eqref{eqn:u_eta}--\eqref{eqn:eta}).  Then the data are given by
\begin{equation}\label{eqn:data}
    f_k = \int_{-1}^1 T_k(\mu) \eta_0 (\mu) d\mu,
\end{equation}
where $\eta_0(\mu) \equiv \eta(0,\mu)$.  

We define
\begin{equation}\label{eqn:c}
    c \equiv f_0 = \int_{-1}^1 \eta_0(\mu) \di{\mu}.
\end{equation}
If we assume $\tq$ is positive (this assumption holds for the Gaussian
source $q(t)$ in \eqref{eqn:gaussian}), then \eqref{eqn:eta} and
\eqref{eqn:c} imply $c^{-1} \eta_0(\mu) \di{\mu}$ is a probability
measure.  We also conjecture that $\int_{-1}^s
c^{-1}\eta_0(\mu)\di{\mu}$ has at least $n$ points of increase on
$[-1,1]$.  This can be proven if the wavespeed $v$ is regular enough,
but for the sake of brevity and clarity we provide a qualitative
explanation in \S~\ref{subsec:eta0_increase}.

\begin{definition}\label{def:snapshots}
    Suppose $\tq(A)$ is positive definite (this is true for the Gaussian
    source in \eqref{eqn:gaussian}, for example).  Let $u(x,t)$ be the
    solution to the following Cauchy problem on $[0,\xmax]\times
    [0,\infty[$:
    \begin{equation}\label{eqn:cauchys}
        Au + u_{tt} = 0, \qquad 
        u_x|_{x=0} = 0, \ u|_{x=\xmax} = 0, \quad
	u|_{t=0} = b, \ 
	    u_t|_{t=0} = 0,
    \end{equation}
    where
    \begin{equation}\label{eqn:b}
        b(x) \equiv v(0)\tq(A)^{1/2}\delta(x+0).
    \end{equation}
    (This equation is equivalent to \eqref{eqn:cauchyb} except for the
    initial condition --- in fact, $\ha{u}(x,t) =
    v(0)^{-1}\tq(A)^{1/2}u(x,t)$.)
    Then, for $k = 0,\ldots,2n-1$, the \emph{snapshots} are defined by 
    \begin{equation}\label{eqn:snapshots_def}
	    u_k(x) \equiv u(x,k\tau) 
	    = \cos\left(k\tau\sqrt{A}\right)b(x) 
	    = T_k\left(\cos\left(\tau\sqrt{A}\right)\right)b(x)
	    = T_k(P)b(x).
    \end{equation}
\end{definition}

From the definition of the snapshots and the fact that functions of $A$
(such as $\tq(A)^{1/2}$) are self adjoint with respect to the inner
product $\llangle \cdot, \cdot \rrangle$, the data satisfy
\begin{equation}\label{eqn:f_k}
    f_k = \left\llangle u_0, u_k \right\rrangle = \left\llangle b,
    T_k(P)b\right\rrangle \eqfor k = 0,\ldots,2n-1.
\end{equation}

Recall that
\begin{equation}\label{eqn:umatrix}
    U \equiv \left[u_0(x), u_1(x), \ldots, u_{n-1}(x)\right].
\end{equation}
Sometimes for shorthand and for $w, u \in L_2[0,\xmax]$ we will write
$w^* u\equiv\left\llangle w, u \right\rrangle$, so by referring to
\eqref{eqn:umatrix} as a matrix we imply the corresponding
multiplication rules. In particular, multiplication from the left by
another matrix $W=\left[w_0(x),\ldots,w_{n-1}(x)\right]$ of the same
form is defined as
\begin{equation}\label{eqn:WstarU}
    W^* U \equiv 
    \begin{bmatrix}
        \llangle w_0, u_0 \rrangle & \llangle w_0, u_1 \rrangle & 
	    \ldots & \llangle w_0, u_{n-1} \rrangle \\
        \llangle w_1, u_0 \rrangle & \llangle w_1, u_1 \rrangle & 
	    \ldots & \llangle w_1, u_{n-1} \rrangle \\
        \vdots & \vdots & \ddots & \vdots \\
        \llangle w_{n-1}, u_0\rrangle & \llangle w_{n-1}, u_1\rrangle & 
	    \ldots & \llangle w_{n-1}, u_{n-1} \rrangle
    \end{bmatrix} 
    \in \mathbb{R}^{n \times n}.
\end{equation}

If our assumption that $\int_{-1}^s c^{-1}\eta_0(\mu)
    \di{\mu}$ has at least $n$ points of increase is satisfied, then
    $\rank U = n$ and $\Range U$ is the Krylov subspace 
\[
    \mathcal{K}_n(u_0, P) = \hbox{span} 
    \left\{ u_0, Pu_0, \ldots, P^{n-1} u_0 \right\};
\]
see, e.g., \S~3.2.1 of the book by Liesen \cite{Liesen:2012:KSM} and
references therein. In particular, $U^*U$ is symmetric and positive
definite since $U$ is of full rank.

In the remainder of this section, we derive an algorithm for
orthogonalizing the snapshots.  As we will see, the orthogonalized
snapshots are localized in some sense, so they provide the key to our
inversion algorithm.  


\subsection{First-order finite-difference Galerkin formulation}
\label{subsec:finite_difference}

Because the snapshots can be written in terms of Chebyshev polynomials
as in \eqref{eqn:snapshots_def} and the Chebyshev polynomials satisfy a
three-term recurrence relation, the snapshots satisfy the following
second-order time-stepping Cauchy problem in operator form:
\begin{equation}\label{eqn:tstepping}
    \frac{u_{k+1} - 2 u_k + u_{k-1}}{\tau^2} = \xi(P) u_k,\qquad  
    u_0 = b,\ u_{-1}=u_{1},
\end{equation}
where  
\begin{equation}\label{eqn:xi}
    \xi(x) \equiv -\frac{2}{\tau^2}(1-x).  
\end{equation}
From a Taylor expansion (for regular enough
$u$), we obtain 
\[
    \xi(P) u=
    -\frac{2}{\tau^2}\left[I-\cos\left(\tau\sqrt{A}\right)\right]u=
    -Au+O\left(\|(\tau A)^2u\|\right),
\]
i.e., \eqref{eqn:tstepping} can be viewed as an explicit time
discretization of \eqref{eqn:cauchys} that reproduces the snapshots
exactly.  

We now state several useful lemmas; the proofs which are not given here
are contained in the appendix.  In the first lemma, we
transform \eqref{eqn:cauchys} to slowness coordinates.
\begin{lemma}\label{lem:trans_to_slow}
    Suppose $u$ solves \eqref{eqn:cauchys}, and let 
    \[
        \tu(\tx,t) \equiv u(x(\tx),t), 
        \quad
        \tv(\tx) \equiv v(x(\tx)),
        \eqand{and}
        \txmax \equiv \tx(\xmax),
    \]
    where the (invertible) slowness coordinate transformation $\tx(x)$
    is defined in \eqref{eqn:slowness}.  Then $\tu$ is the solution of
    the following Cauchy problem on $[0,\txmax]\times [0,\infty[$:
    \begin{equation}\label{eqn:cauchyslow}
        \tA \tu+\tu_{tt}=0, \qquad  
        \tu_\tx|_{\tx=0}=0, \ \tu|_{\tx = \txmax}=0, \quad 
	\tu|_{t=0}=\ti{b}, \ 
        \tu_t|_{t=0}=0, 
    \end{equation}
    where
    \[
        \ti{b}(\tx) \equiv \tq\left(\tA\right)^{1/2}\delta(\tx+0) 
	    \eqand{and}
	    \tA \tu \equiv -\tv\frac{\partial}{\partial \tx}
	        \left(\frac{1}{\tv}\frac{\partial\tu}{\partial \tx}\right)
    \]
    with the Neumann--Dirichlet boundary conditions in
    \eqref{eqn:cauchyslow}.  The operator $\tA$ is self adjoint and
    positive definite with respect to the inner product
    $\ootv{\cdot}{\cdot}$, where
    \[
        \ootv{\tu}{\tw} \equiv \int_0^{\txmax} 
            \tu(\tx)\tw(\tx)\frac{1}{\tv(\tx)}\di{\tx}.
    \]
\end{lemma}

We now define a dual variable, $\tw$, that will be useful in the
remainder of the paper.
\begin{definition}\label{def:w}
    We define the \emph{dual variable}, denoted by $\tw$, as the
    solution of the following Cauchy problem on $[0,\txmax]\times
    [0,\infty[$:
    \begin{equation}\label{eqn:cauchyw}
        \tC\tw + \tw_{tt} = 0, \qquad
        \tw|_{\tx = 0} = 0, \ \tw_\tx|_{\tx = \txmax} = 0, \quad
        \tw|_{t = 0} = 0, \ \tw_t|_{t=0} = \frac{1}{\tv}
	\frac{\partial\ti{b}}{\partial\tx},
    \end{equation}
    where
    \[
        \tC \tw \equiv -\frac{1}{\tv}\frac{\partial}{\partial \tx}
	    \left(\tv\frac{\partial\tw}{\partial \tx}\right)
    \]
    with the Dirichlet--Neumann boundary conditions in
    \eqref{eqn:cauchyw}. 
    \footnote{In physical coordinates, the operator
    $C$ is given by $Cw = -\frac{d}{dx}\left(v^2\frac{dw}{dx}\right)$
    with the boundary conditions $w|_{x=0} = 0$ and $w_x|_{x=\xmax} =
    0$.}
    The operator $\tC$ is self adjoint and
    positive definite with respect to the inner product
    $\iptv{\cdot}{\cdot}$, where
    \[
        \iptv{\tu}{\tw} \equiv \int_0^{\txmax} 
            \tu(\tx)\tw(\tx)\tv(\tx)\di{\tx}.
    \]
\end{definition}
The Cauchy problems \eqref{eqn:cauchyslow} and \eqref{eqn:cauchyw} can
be rewritten in first-order form as in the following lemma.
\begin{lemma}\label{lem:1stc}
    Suppose $\tu$ and $\tw$ are the solutions to the following Cauchy
    problem on $[0,\txmax]\times [0,\infty[$:
    \begin{equation}\label{eqn:1stc} 
        \tw_\tx= \frac{1}{\tv} \tu_t, \ \tu_\tx= \tv\tw_t, \qquad 
        \tu|_{\tx = \txmax} = 0, \ \tw|_{\tx=0}=0, \quad
	\tu|_{t=0}=\ti{b}, \ 
	\tw|_{t=0}=0.
    \end{equation}
    Then $\tu$ solves \eqref{eqn:cauchyslow} and $\tw$ solves
    \eqref{eqn:cauchyw}.
\end{lemma}

The next definition is an extension of Definition~\ref{def:snapshots}.
\begin{definition}\label{def:snapshots_pd}
    Let $\tu$ and $\tw$ be the solutions to \eqref{eqn:1stc} (so $\tu$
    is the solution to \eqref{eqn:cauchyslow} and $\tw$ is the solution
    to \eqref{eqn:cauchyw}).  Then, for $k = 0,\ldots,2n-1$, the
    \emph{primary snapshots} are $\tu_k \equiv \tu(\tx,k\tau)$, and the
    \emph{dual snapshots} are $\tw_k \equiv \tw(\tx,(k+1/2)\tau)$.
\end{definition}

Note that the primary snapshots, $\tu_k$, are simply the snapshots from
Definition~\ref{def:snapshots}, namely $u_k$, transformed into
slowness coordinates; i.e., $\tu_k(\tx) = u_k(x(\tx))$.  

In the next lemma, we give expressions and finite-difference recursions
for the primary and dual snapshots.  
\begin{lemma}\label{lem:snapshots_uw}
    Suppose $\tu$, $\tw$ are the solutions to \eqref{eqn:1stc}.  Then,
    for $k = 0,\ldots,2n-1$, the primary snapshots are given by 
    \begin{equation*}
        \tu_k(\tx) = T_k\left(\ti{P}\right)\tu_0(\tx), 
    \end{equation*}
    where $\ti{P} \equiv \cos\left(\tau\sqrt{\tA}\right)$ and
    $\tu_0(\tx) = \ti{b}(\tx)= \tq\left(\tA\right)^{1/2}\delta(\tx+0)$.
    This implies the primary snapshots satisfy the recursion 
    \begin{equation}\label{eqn:primary_recursion}
        \frac{\tu_{k+1} - 2\tu_k +
	\tu_{k-1}}{\tau^2} = \xi\left(\ti{P}\right)\tu_k 
	\eqfor k = 0,\ldots,2n-2, \qquad \tu_0 =
	\ti{b}, \ \tu_{1} = \tu_{-1},
    \end{equation}
    where $\xi$ is defined in \eqref{eqn:xi}.

    Similarly, for $k = 0,\ldots,2n-1$, the dual snapshots are given by 
    \begin{equation}\label{eqn:dual_formula}
        \tw_k(\tx) =
            \left[\Ts_k\left(\ti{P}_C\right) +
	        \Ts_{k-1}\left(\ti{P}_C\right)\right]\tw_0,
    \end{equation}
    where $\ti{P}_C \equiv \cos\left(\tau\sqrt{\tC}\right)$, $\Ts_k$ is
    the $k\textsuperscript{th}$ Chebyshev polynomial of the second kind
    (with $\Ts_{-1} = 0$ and $\Ts_{-2} = -1$),
    and $\tw_0 = \tw(\tx,0.5\tau).$ 
    This implies the dual snapshots satisfy the recursion
    \begin{equation}\label{eqn:dual_recursion}
        \begin{aligned}
	    &\frac{\tw_{k+1} - 2\tw_k +
	    \tw_{k-1}}{\tau^2} = \xi\left(\ti{P}_C\right)\tw_k 
	    \eqfor k = 0,\ldots,2n-2, \\
	    &\tw_0 + \tw_{-1} = 0, \ 
	    \tw_{0} =\sin\left(0.5\tau\sqrt{\tC}\right)\tC^{-1/2}
	        \frac{1}{\tv}
		\frac{\partial\ti{b}}{\partial\tx}.
	\end{aligned}
    \end{equation}
\end{lemma}

In the following lemma, we rewrite the recursions from
Lemma~\ref{lem:snapshots_uw} in first-order form.
\begin{lemma}\label{lem:leapfrog}
    The second-order time-stepping schemes \eqref{eqn:primary_recursion}
    and \eqref{eqn:dual_recursion} can be equivalently rewritten as the
    first-order ``leapfrog'' discretization of \eqref{eqn:1stc}.  In
    particular, 
    \begin{equation}\label{eqn:tstepping1st} 
        \begin{cases}
	    \dfrac{\tw_{k}-\tw_{k-1}}{\tau} = \dfrac{1}{\tv}\Lt \tu_k
	    &\text{for } k = 0,\ldots,2n-1, \myspace
	    \dfrac{\tu_{k+1}-\tu_{k}}{\tau} = -\tv\Lt^T\tw_k
	        &\text{for } k = 0,\ldots,2n-2, \myspace
		\tu_0=\ti{b}, \ 
	        \tw_{0}+\tw_{-1} = 0;
	\end{cases}
    \end{equation}
    here $\Lt^T$ is the adjoint of $\Lt$ with respect to the standard
    inner product on $L_2[0,\txmax]$, 
    \begin{equation*}
        \Lt = \frac{2}{\tau}\cdot\frac{\partial}{\partial\tx} \tA^{-1/2}
	    \sin\left(0.5\tau\sqrt{\tA}\right),
	\eqand{and} 
	\Lt^T =	-\frac{2}{\tau}\cdot\frac{1}{\tv} 
	    \sin\left(0.5\tau\sqrt{\tA}\right)\tA^{-1/2}\tv
	    \frac{\partial}{\partial\tx}.
    \end{equation*}
\end{lemma}

In particular, Lemma~\ref{lem:leapfrog} implies the operators
$\xi\left(\ti{P}\right)$ and $\xi\left(\ti{P}_C\right)$ may be factored
as
\begin{equation}\label{eqn:xi_factored}
    \xi\left(\ti{P}\right) = -\tv\Lt^T\frac{1}{\tv}\Lt
    \eqand{and}
    \xi\left(\ti{P}_C\right) = -\frac{1}{\tv}\Lt\tv\Lt^T.
\end{equation}

The upshot of this section is that the snapshots in
Definition~\ref{def:snapshots_pd} may be generated via finite-difference
schemes --- the second-order finite-difference schemes are given in
Lemma~\ref{lem:snapshots_uw} while the equivalent first-order
finite-difference scheme is given in Lemma~\ref{lem:leapfrog}.  This
theme permeates the remainder of this section --- as we will see, all of
our first-order algorithms and recursions have second-order
equivalents.  


\subsection{Orthogonalization of the snapshots}
\label{subsubsec:orthogonalization}

It turns out the orthogonalized snapshots are localized (we will justify
this in later sections), so they are useful as a basis for an inversion
method.  In particular, the (squared) norm of each orthogonalized
snapshot contains information about the magnitude of the velocity near
the point about which that orthogonalized snapshot is localized
(specifically, the center of mass of the corresponding squared
orthogonalized snapshot).  We discuss our inversion algorithm in more
detail in \S~\ref{sec:inversion}; for now, we focus on orthogonalizing
the snapshots.

Lemma~\ref{lem:snapshots_uw} implies the first $n$ primary
and dual snapshots span the Krylov subspaces
\begin{equation*}
    \ti{\mathcal{K}}^{u}_n\left(\tu_0,\ti{P}\right) \equiv 
        \mathrm{span} \left\{\tu_0,
	\ti{P}\tu_0,\ldots,\ti{P}^{n-1}\tu_0\right\}
\end{equation*}
and
\begin{equation*}
    \ti{\mathcal{K}}^{w}_n\left(\tw_0,\ti{P}_C\right) \equiv 
        \mathrm{span}\left\{\tw_0,
        \ti{P}_C\tw_0,\ldots,\ti{P}_C^{n-1}\tw_0\right\},
\end{equation*}   
respectively.  The classical method for constructing an orthonormal
basis of a Krylov subspace is the Lanczos algorithm
\cite{Parlett:1998:SEP}, and the algorithm we use is a first-order
equivalent of the Lanczos algorithm.  We begin by defining some useful
operators.  
\begin{definition}\label{def:L_dtau}
    We define the operator $\cL$ by 
    \begin{equation*}
        \cL \equiv \cLout.
    \end{equation*}
    Then the time-stepping scheme \eqref{eqn:tstepping1st} can be
    written as
    \begin{equation}\label{eqn:leapfrog_matrix}
        \cL\myvector{\tu_k}{\tw_k} = \dtau \myvector{\tu_k}{\tw_k}
	\eqfor k = 0,\ldots,2n-1,
    \end{equation}
    where
    \begin{equation}\label{eqn:dtau}
        \dtau\myvector{\tu_k}{\tw_k} \equiv \frac{1}{\tau}
        \myvector{\tu_{k+1}-\tu_k}{\tw_k-\tw_{k-1}}.
    \end{equation}
    (Technically speaking, $\tu_{2n}$ is not defined --- we may define
    it through \eqref{eqn:leapfrog_matrix} for completeness.)
    We define the inner product $\ootviptv{\cdot}{\cdot}$ by
    \begin{equation*}
        \ootviptv{\myvector{\tu^a}{\tw^a}}{\myvector{\tu^b}{\tw^b}}
	\equiv
	\ootv{\tu^a}{\tu^b} + \iptv{\tw^a}{\tw^b}.
    \end{equation*}
\end{definition}
The operator $\cL$ is anti-self-adjoint with respect to the inner
product $\ootviptv{\cdot}{\cdot}$, i.e., 
\[
    \ootviptv{\cL\myvector{\tu^a}{\tw^a}}{\myvector{\tu^b}{\tw^b}}
    = 
    -\ootviptv{\myvector{\tu^a}{\tw^a}}{\cL\myvector{\tu^b}{\tw^b}}.
\]
Next, we project the operator $\cL$ onto the Krylov subspaces spanned by
the snapshots, namely $\ti{\mathcal{K}}_n^u\left(\tu_0, \ti{P}\right)$
and $\ti{\mathcal{K}}_n^w\left(\tw_0, \ti{P}_C\right)$.  Before
presenting the algorithm, we introduce some notation.

We denote the orthogonalized primary and dual snapshots by $\ou_j$ and
$\ow_j$, respectively, for $j = 1,\ldots,n$.  (Note that we have shifted
the index by $1$ --- the snapshots $\tu_k$ and $\tw_k$ are indexed from
$k = 0$ to $k = n-1$.)  We store the orthogonalized snapshots in
``vectors'' of the form 
\begin{equation}\label{eqn:oU}
    \oU_{2j-1} = \myvector{\ou_j}{0} 
    \eqand{and}
    \oU_{2j}   = \myvector{0}{\ow_j} \eqfor j = 1,\ldots,n,
\end{equation}
or, even more compactly, in a ``matrix''
\begin{equation}\label{eqn:oQ}
    \cQ \equiv
    \begin{bmatrix}
        \oU_1, \ldots, \oU_{2n}
    \end{bmatrix}
    = 
    \begin{bmatrix}
    \ou_1 &    0  & \ou_2 &    0  & \ldots & \ou_n &    0 \\
       0  & \ow_1 &    0  & \ow_2 & \ldots &    0  & \ow_n 
    \end{bmatrix}.
\end{equation}
The Lanczos algorithm constructs a tridiagonal matrix $\bcT \in
\mathbb{R}^{2n\times 2n}$ such that 
\begin{equation}\label{eqn:crux_alg_2}
    \cL\cQ = \cQ\bcT + \frac{1}{\g_n}\oU_{2n+1}\la{e}_{2n}^T,
\end{equation}
where $\g_n$ is a constant we define later and $\oU_{2n+1}$ appears
because, in general, Lanczos tridiagonalization is run on a family of
snapshots that may be infinite (or with dimension at least
$2n+1$ --- see, e.g., \cite{Parlett:1998:SEP}); $\oU_{2n+1}$ will not be
needed for the remainder of the paper.  Since $\cL$ is anti-self-adjoint
and the columns of $\cQ$ are to be orthogonal, the diagonal components
of $\bcT$ must be $0$.  To obtain the desired orthogonality properties,
we take
\begin{equation}\label{eqn:bcT}
    \bcT = \bcO\la{\Gamma}^{-1},
\end{equation}
where 
\begin{equation}\label{eqn:OGamma}
    \bcO \equiv
    \begin{bmatrix}
        0 & -1     &        &    \\
	1 & 0      & \ddots &    \\
	  & \ddots & \ddots & -1 \\
	  &        & 1      & 0
    \end{bmatrix}
    = -\bcO^T \in \mathbb{R}^{2n\times 2n},
    \quad
    \la{\Gamma} \equiv \diag(\ghat_1, \g_1, \ghat_2, \g_2, \ldots,
        \ghat_n, \g_n),
\end{equation}
and, for $j = 1,\ldots,n$, 
\begin{equation}\label{eqn:gammas}
    \ghat_j \equiv \normootv{\ou_j}^{-2} \equiv 
        \ootv{\ou_j}{\ou_j}^{-1}
    \eqand{and}
    \g_j \equiv \normiptv{\ow_j}^{-2} \equiv
        \iptv{\ow_j}{\ow_j}^{-1}.
\end{equation}
Then \eqref{eqn:oU}--\eqref{eqn:gammas} give the first-order algorithm
for the orthogonalization of the first $n$ primary and dual snapshots,
which is summarized in Algorithm~\ref{alg:algorithm_2}.
\begin{algorithm}[H]
    \caption{Orthogonalization of Snapshots}
	\label{alg:algorithm_2}
	\begin{algorithmic}
	    \Require $\tu(\tx,0) = \ti{b}(\tx)$,
		    $\tv$, $\txmax$, $n$,
		    $\Lt$, and $\Lt^T$
	    \Ensure $\ghat_j$, $\g_j$, and orthogonalized snapshots
	        $\ou_j$, $\ow_j$ for $j = 1,\ldots,n$
	    \State Set $\ow_0 = 0$ and $\ou_1 = \ti{b}$.
	    \For{$j = 1,\ldots,n$}
            \begin{enumerate}
                \item 
		    $\ghat_j = \dfrac{1}{\normootv{\ou_j}^2} = 
	                \dfrac{1}{\ds\int_0^{\txmax} (\ou_j)^2 
			\dfrac{1}{\tv}\di{\tx}}$;
                \item 
		    $\ow_j = \ow_{j-1} +\ghat_j\dfrac{1}{\tv}\Lt\ou_j$;
	        \item 
		    $\g_j = \dfrac{1}{\normiptv{\ow_j}^2} = 
	                \dfrac{1}{\ds\int_0^{\txmax}
			(\ow_j)^2\tv\di{\tx}}$;
                \item 
		    $\ou_{j+1} = \ou_{j} - \g_j\tv\Lt^T\ow_j$.
	    \end{enumerate}    
	    \EndFor
	\end{algorithmic}
\end{algorithm}

We pause to consider a couple of important features of
Algorithm~\ref{alg:algorithm_2}.  First, note that the recursion steps
(steps $2$ and $4$) resemble a finite-difference algorithm that
exactly computes the orthogonalized snapshots, since 
\[
    \frac{\ou_{j+1}-\ou_j}{\g_j} = -\tv\Lt^T\ow_j 
    \eqand{and}
    \frac{\ow_j-\ow_{j-1}}{\ghat_j} = \frac{1}{\tv}\Lt\ou_j.
\]
Second, if $\ou_j$ and $\ow_j$ are localized in some sense (as we
claimed above), then, due to steps $1$ and $3$, $\ghat_j$ and $\g_j$ are
related to localized averages of the velocity (roughly speaking). This
is a key insight for our reconstruction algorithm --- $\ghat_j$ and
$\g_j$ give us estimates of pointwise values of $v$ near where the
squared orthogonalized snapshots are localized, i.e., on the optimal
grid defined by the centers of mass of the squared orthogonalized
snapshots.  Admittedly, this explanation is not complete; we will add
more details in later sections.  Third, in
Algorithm~\ref{alg:algorithm_2} we assume $v$ (hence $\tv$) is known; in
\S~\ref{subsec:Galerkin}, we compute $\ghat_j$, $\g_j$ from the measured
data without any \emph{a priori} knowledge of $v$.  Finally, the
following proposition summarizes the important properties of
Algorithm~\ref{alg:algorithm_2}.
\begin{proposition}\label{prop:alg_2}
    Suppose $\ou_j$, $\ow_j$ ($j = 1,\ldots,n$) are obtained via
    Algorithm~\ref{alg:algorithm_2}.  Then $\ootv{\ou_i}{\ou_j} =
    \ghat_j^{-1}\delta_{ij}$ and $\iptv{\ow_i}{\ow_j} =
    \g_j^{-1}\delta_{ij}$ for $i$, $j = 1,\ldots,n$.  Moreover, 
    \[
        \mathrm{span}\left\{\ou_1,\ldots,\ou_n\right\} =
            \ti{\mathcal{K}}_n^u\left(\tu_0,\ti{P}\right)
	\eqand{and}
        \mathrm{span}\left\{\ow_1,\ldots,\ow_n\right\} =
            \ti{\mathcal{K}}_n^w\left(\tw_0,\ti{P}_C\right).
    \]
\end{proposition}

The next two lemmas show that the first-order algorithm in
Algorithm~\ref{alg:algorithm_2} is equivalent to the Lanczos algorithm.  
\begin{lemma}\label{lem:u_Lanczos}
    Suppose the functions $\ou_j$ ($j = 1,\ldots,n$) are constructed via
    Algorithm~\ref{alg:algorithm_2}.  Then $\ou_j =
    \ghat_j^{-1/2}\vartheta_j$, where the functions $\vartheta_j$
    are obtained from the following Lanczos algorithm:
	\begin{algorithmic}
	    \Require $\ou_1 \equiv \tu(\tx,0) =
	    \ti{b}(\tx)$,
		    $\tv$, $\txmax$, $n$,
		    and $\xi\left(\ti{P}\right)$
	    \Ensure $\ghat_j$ and normalized, orthogonalized primary 
	            snapshots $\vartheta_j$ for $j = 1,\ldots,n$
            \State Set $\vartheta_0 = 0$ and $\vartheta_1 =
		    \dfrac{\ou_1}{\normootv{\ou_1}}$.
	    \For{$j = 1,\ldots,n$}
            \begin{enumerate}
                \item 
		    $a_j^u =\ootv{\vartheta_j} 
		        {\xi\left(\ti{P}\right)\vartheta_j}$;
                \item 
		    $r = \left[\xi\left(\ti{P}\right) - 
		        a_j^uI\right]\vartheta_j -
		        b_{j-1}^u\vartheta_{j-1}$;
	        \item 
		    $b_j^u = \sqrt{\ootv{r}{r}}$;
                \item 
		    $\vartheta_{j+1} = \dfrac{r}{b_j^u}$.
	    \end{enumerate}    
	    \EndFor
	\end{algorithmic}
    Moreover, the Lanczos coefficients
    $a_j^u$, $b_j^u$ from the above algorithm are related to
    $\ghat_j$, $\g_j$ from Algorithm~\ref{alg:algorithm_2} by 
    \begin{equation}\label{eqn:ab_u}
        \begin{cases}
	    a_j^u = -\dfrac{1}{\ghat_j}\left(\dfrac{1}{\g_{j-1}} +
	        \dfrac{1}{\g_j}\right) &\text{for } 
		j = 1,\ldots,n,\myspace
	    b_j^u = \dfrac{1}{\g_j\sqrt{\ghat_j\ghat_{j+1}}} 
	        &\text{for } j = 1,\ldots,n-1,
	\end{cases}
    \end{equation}
    where we have taken $\g_0 \equiv \infty$.  
\end{lemma}
\begin{lemma}\label{lem:w_Lanczos}
    Suppose the functions $\ow_j$ ($j = 1,\ldots,n$) are constructed via
    Algorithm~\ref{alg:algorithm_2}.  Then $\ow_j =
    \g_j^{-1/2}\varrho_j$, where the functions $\varrho_j$
    are obtained from the following Lanczos algorithm:
	\begin{algorithmic}
	\Require $\ow_1 = \ghat_1\dfrac{1}{\tv}\Lt\ou_1$ (from
	    Algorithm~\ref{alg:algorithm_2}), $\tv$, 
	    $\txmax$, $n$, and $\xi\left(\ti{P}_C\right)$
	    \Ensure $\g_j$ and normalized, orthogonalized dual 
	            snapshots $\varrho_j$ for $j = 1,\ldots,n$
            \State Set $\varrho_0 = 0$ and $\varrho_1 =
	    \dfrac{\ow_1}{\normiptv{\ow_1}}$.
	    \For{$j = 1,\ldots,n$}
            \begin{enumerate}
                \item 
		    $a_j^w =\iptv{\varrho_j} {\xi\left(\ti{P}_C\right)
		        \varrho_j}$;
                \item 
		    $r = \left[\xi\left(\ti{P}_C\right) - 
		        a_j^wI\right]\varrho_j -
		        b_{j-1}^w\varrho_{j-1}$;
	        \item 
		    $b_j^w = \sqrt{\iptv{r}{r}}$;
                \item 
		    $\varrho_{j+1} = \dfrac{r}{b_j^w}$.
	    \end{enumerate}    
	    \EndFor
	\end{algorithmic}
    Moreover, the Lanczos coefficients
    $a_j^w$, $b_j^w$ from the above algorithm are related to
    $\ghat_j$, $\g_j$ from Algorithm~\ref{alg:algorithm_2} by 
    \begin{equation*}
        \begin{cases}
	    a_j^w = -\dfrac{1}{\g_j}\left(\dfrac{1}{\ghat_j} +
	        \dfrac{1}{\ghat_{j+1}}\right) &\text{for } 
		j = 1,\ldots,n-1,\myspace
	    b_j^w = \dfrac{1}{\ghat_{j+1}\sqrt{\g_j\g_{j+1}}} 
	        &\text{for } j = 1,\ldots,n-1.
	\end{cases}
    \end{equation*}
\end{lemma}

Recall that, before orthogonalization, the primary and dual snapshots
can be represented in terms of Chebyshev polynomials of the operators
$\ti{P}$ and $\ti{P}_C$, respectively (see
Lemma~\ref{lem:snapshots_uw}).  The next lemma and the remark following
it give representations of the orthogonalized primary and dual snapshots
in terms of polynomials of the operators $\xi\left(\ti{P}\right)$ and
$\xi\left(\ti{P}_C\right)$, respectively.  The true value of
Lemma~\ref{lem:poly_alg_2}, however, is that it provides a proper
normalization for the derivation of explicit formulas for the continued
fraction coefficients (i.e., $\ghat_j$ and $\g_j$) in terms of the data
in both the scalar (1D) and matrix (2D and higher) cases.  We relegate
the proofs to the appendix.
\begin{lemma}\label{lem:poly_alg_2}
    Suppose the orthogonalized snapshots $\ou_j$ and $\ow_j$ ($j =
    1,\ldots,n$) are obtained via Algorithm~\ref{alg:algorithm_2}.  Then
    \begin{equation*}
        \ou_j = q_j^u\left(\xi\left(\ti{P}\right)\right)\ou_1, 
	\eqand{where}
	q_j^u(0) = 1
    \end{equation*}
    and $q_j^u$ is a polynomial of degree $j-1$;
    similarly,
    \begin{equation*}
        \ow_j = q_j^w\left(\xi\left(\ti{P}_C\right)\right)\ow_1,
	\eqand{where}
    q_j^w(0) = \frac{1}{\ghat_1}\ds\sum_{k=1}^j\ghat_k
    \end{equation*}
    and $q_j^w$ is a polynomial of degree $j-1$.  
\end{lemma}
\begin{remark}\label{rem:polynomials}
    Using the fact that, in spatial coordinates $x$, $\vartheta_1(x) =
    \ghat_1^{1/2}b(x)$, one can show $q_j^u =
    \ghat_j^{-1/2}\ghat_1^{1/2}q_j^{\xi}$, where $\{q_j^{\xi}\}_{j=1}^n$
    is the set of orthonormal polynomials generated by
    Algorithm~\ref{alg:Lanczos} (below) with the inner product
    \[
        \left\langle p, q \right\rangle_{y,\xi(\theta)} \equiv 
	    \frac{1}{c} \ds\sum_{j=1}^n y_j^2 
	    p\left(\xi(\theta_j)\right)q\left(\xi(\theta_j)\right)
    \]
    in place of the inner product
    $\left\langle \cdot, \cdot \right\rangle_{y,\theta}$.  
\end{remark}


\section{Transformation of the time-domain data to an equivalent 
finite-difference reduced-order model}\label{sec:transformation}

Our goal in this section is to construct a finite-difference scheme
involving a data-driven reduced-order model for the propagator $P =
\cos\left(\tau\sqrt{A}\right)$ that reproduces the data \eqref{eqn:data}
exactly.  The coefficients of this finite-difference scheme (which is
also our ROM) are essentially localized averages of the velocity.  Thus
the construction of the ROM is the core of our inversion method, since
it transforms the time-domain data (which is all we have) into a ``more
usable'' form.


\subsection{Chebyshev moment problem in Galerkin--Ritz formulation}
\label{subsec:Galerkin_Ritz}

We solve the data-interpolation problem by constructing a Gaussian
quadrature rule with nodes $\theta_j$ and weights $y_j^2$ for the weight
$\eta_0$ (defined in \eqref{eqn:data}); that is, we find spectral nodes
$\theta_j$ and weights $y_j^2$ such that
\begin{equation}\label{eqn:chebmatch}
    \int_{-1}^1 T_k(\mu) \eta_0(\mu) \di{\mu}
    = \sum_{j=1}^{n} y_j^2 T_k(\theta_j) 
    = f_k \eqfor k = 0,\ldots,2n-1.
\end{equation}
This is the classical quadrature problem (in the Chebyshev basis), and
the existence and uniqueness of its solution are given by the following
well-known result --- see, e.g., Theorems 1.7, 1.19 (which can be
extended to discrete measures), and 1.46 in the book by Gautschi
\cite{Gautschi:2004:OPC}.
\begin{lemma}\label{lem:points_of_increase} 
    Let  $c^{-1} \eta_0(\mu) \di{\mu}$ be a (probability) measure such
    that $\int_{-1}^s c^{-1}\eta_0(\mu)\di{\mu}$ has at least $n$ points
    of increase on $[-1,1]$ (collectively, such points are also known as
    the support or spectrum of the measure $c^{-1}\eta_0(\mu)\di{\mu}$).
    Then \eqref{eqn:chebmatch} has a unique solution with positive $y_j$
    and noncoinciding $\theta_j \in (-1,1)$.
\end{lemma} 

As we discussed in \S~\ref{sec:continuum}, for regular enough wavespeed
$v$ it can be shown that $c^{-1}\eta_0(\mu)\di{\mu}$ satisfies the
hypothesis of Lemma~\ref{lem:points_of_increase} --- see
\S~\ref{subsec:eta0_increase} in the appendix.

There are numerous algorithms for the quadrature problem
\eqref{eqn:chebmatch} (see, e.g., the end of \S~1.4.1 and Chapter 3 in
\cite{Gautschi:2004:OPC}); however, for the sake of the continuum
interpretation of our approach we give an algorithm based on the
Galerkin projection method onto Krylov subspaces.  The proofs of the
remaining lemmas in this section are given in the appendix.

The following lemma gives the Galerkin representation of $u_k$ and $f_k$
in the Krylov subspace $\mathcal{K}_n(u_0,P)$.
\begin{lemma}\label{lem:data_lemma}
    If $\eta_0$ satisfies the hypothesis of
    Lemma~\ref{lem:points_of_increase}, then
    \begin{equation}\label{eqn:state}
    u_k= UT_k(\la{H}) \la{e}_1 \eqfor k=0,\ldots,n-1,
    \end{equation}
    and 
    \begin{equation}\label{eqn:match}
        f_k 
        =\la{e}_1^T(U^* U)T_k(\la{H}) \la{e}_1
        \eqfor k=0,\ldots,2n-1,
    \end{equation}
    where 
    \begin{equation}\label{eqn:H} 
        \la{H} \equiv (U^* U)^{-1}(U^*P U)\in\bfR^{n\times n}.
    \end{equation}
\end{lemma}

We give the spectral decomposition of the matrix $\la{H}$ in the next
lemma.
\begin{lemma}\label{lem:specH}
    Suppose $\eta_0$ satisfies the hypothesis of
    Lemma~\ref{lem:points_of_increase} and $\la{H}$ is defined as in
    \eqref{eqn:H}.  Then $\la{H}$ is self adjoint with respect to the
    inner product $\usuh{\cdot}{\cdot}$, defined by 
    \[
        \usuh{\la{x}}{\la{z}} \equiv
	    \left[(U^*U)^{1/2}\la{x}\right]^T
	    \left[(U^*U)^{1/2}\la{z}\right]
	    = \la{x}^T(U^*U)\la{z} \eqfor 
	    \la{x}, \la{z} \in \mathbb{R}^n.
    \]
    The spectral decomposition of $\la{H}$ can be written as
    \begin{equation}\label{eqn:generaleig}
        \la{H}  =  \la{\Phi}\la{\Theta}\la{\Phi}^T U^* U,
    \end{equation}
    where $\la{\Theta}$ is a diagonal matrix of the eigenvalues of
    $\la{H}$ and $\la{\Phi}$ is the $U^* U$-orthonormal eigenvector
    matrix, i.e., $\la{\Phi}^T U^* U \la{\Phi} = \la{I}$.
\end{lemma}

Substituting \eqref{eqn:generaleig} into \eqref{eqn:match} we obtain
\begin{equation}\label{eqn:quadrature} 
    f_k= \la{\chi}^T T_k(\la{\Theta})\la{\chi} \eqfor k=0,\ldots,2n-1, 
    \eqand{where} \la{\chi}\equiv \la{\Phi}^T U^*U \la{e}_1.
\end{equation}
Comparing \eqref{eqn:quadrature} and \eqref{eqn:chebmatch}, we derive 
\begin{equation}\label{eqn:wn} 
    \diag \ {\theta_i}=\la{\Theta} \eqand{and} 
    (y_1,\ldots,y_n)^T=\la{\chi}.
\end{equation}
In other words, once we know $\la{\Theta}$ and $\la{\chi}$ we may
compute the nodes $\theta_j$ and weights $y_j^2$ for the Gaussian
quadrature \eqref{eqn:chebmatch}.

The matrices $U^* U$ and $U^*P U$ (and, hence, $\la{H}$ via
\eqref{eqn:H}) can be computed in terms of the data via the following
lemma (the proof is given in \S\ref{subsec:lem3_proof} of the
appendix).
\begin{lemma}\label{lem:TplusH}
    We use the notation $\la{T}$(first column, first row) for Toeplitz
    matrices and $\la{H}$(first column, last row) for Hankel matrices.
    Then if we set
    \[
        \begin{split}
	    \la{T}^0 &\equiv\la{T}([f_{0}, f_{1}, f_{2},\ldots,f_{n-1}], 
    	          [f_{0}, f_{1}, f_{2}, \ldots, f_{n-1}]), \\
	    \la{T}^+ &\equiv\la{T}([f_{1}, f_{2}, f_{3}, \ldots, f_{n}], 
    	          [f_{1}, f_{0}, f_{1}, \ldots, f_{n-2}]), \\
	    \la{T}^- &\equiv\la{T}([f_{1}, f_{0}, f_{1},\ldots,f_{n-2}], 
    	          [f_{1}, f_{2}, f_{3}, \ldots, f_{n}]), \\
	    \la{H}^0 &\equiv\la{H}([f_{0}, f_{1}, f_{2},\ldots,f_{n-1}], 
    	          [f_{n-1}, f_{n}, f_{n+1}, \ldots, f_{2n-2}]), \\
	    \la{H}^+ &\equiv\la{H}([f_{1}, f_{2}, f_{3}, \ldots, f_{n}], 
    	          [f_{n}, f_{n+1}, f_{n+2}, \ldots, f_{2n-1}]), \\
	    \la{H}^- &\equiv\la{H}([f_{1}, f_{0}, f_{1},\ldots,f_{n-2}], 
    	          [f_{n-2}, f_{n-1}, f_{n}, \ldots, f_{2n-3}]), 
        \end{split}
    \]
    we get the expressions
    \begin{equation}\label{eqn:datamatrix}
        U^*PU=\frac{1}{4} 
            \left( \la{T}^+ + \la{T}^- + \la{H}^+ + \la{H}^- \right)
	    \eqand{and}
	U^*U= \frac{1}{2} \left( \la{T}^0 + \la{H}^0 \right).
    \end{equation}
\end{lemma}

In summary, formulas \eqref{eqn:H}--\eqref{eqn:datamatrix} provide the
algorithm for computing $y_j$ and $\theta_j$ from the data for
$j=1,\ldots,n$.  

Finally, substituting  \eqref{eqn:generaleig} into \eqref{eqn:state} we
obtain
\begin{equation}\label{eqn:ritz} 
    u_k = Z T_k(\la{\Theta})\la{\chi} \eqfor k=0,\ldots,n-1,
\end{equation}
where $Z=U\la{\Phi}$. By construction, $Z \la{e}_j$ and $\theta_j$ are
the Ritz pairs of $P$ on the Krylov subspace $\mathcal{K}_n(u_0, P)$.  


\subsection{Finite-difference recursion}\label{subsec:fdr}
Let us find a symmetric, tridiagonal matrix
\begin{equation}\label{eqn:P_n_def}
    \la{P}_n=
    \begin{bmatrix}
        \alpha_1 & \beta_1  &             &             \\
	\beta_1  & \alpha_2 & \ddots      &             \\ 
	         & \ddots   & \ddots      & \beta_{n-1} \\
                 &          & \beta_{n-1} & \alpha_n
    \end{bmatrix}
    = \la{P}_n^T \in \bfR^{n\times n}
\end{equation}
such that
\begin{equation}\label{eqn:match-1}
    \la{b}_n^T T_k (\la{P}_n) \la{b}_n =
    \sum_{j=1}^{n} y_j^2 T_k(\theta_j)= f_k
    \eqfor k=0,\ldots,2n-1,
\end{equation}
where $c$ is defined in \eqref{eqn:c} and $\la{b}_n \equiv
\sqrt{c}\la{e}_1$.  Taking $k = 0$ in \eqref{eqn:match-1} gives 
\begin{equation}\label{eqn:c_again}
    c = \sum_{j=1}^{n} y_j^2 = f_0 = \int_{-1}^1 \eta_0(\mu) \di{\mu}.
\end{equation}
The expression on the left in \eqref{eqn:match-1} is the ROM for the
data as expressed in \eqref{eqn:f_k}.  We will see that $\la{P}_n$ and
$\la{b}_n$ are the projections (up to scaling for $\la{b}_n$ of the
propagator $P = \cos\left(\tau\sqrt{A}\right)$ and the
source/measurement distribution $b$, respectively, onto the space
spanned by the (orthogonalized) snapshots, namely
$\mathcal{K}_n(u_0,P)$; i.e., $\la{P}_n$ is our ROM of $P$ and
$\la{b}_n$ is our ROM of $b$.

In \S~\ref{subsec:Galerkin_Ritz}, we constructed a Gaussian
quadrature with respect to the weight $\eta_0/c$ with nodes $\theta_j
\in [-1,1]$ and positive weights $y_j^2/c$ such that, for sufficiently
smooth functions $g$, 
\begin{equation}\label{eqn:GQ}
    \frac{1}{c}\sum_{j=1}^n y_j^2 g(\theta_j) \approx \int_{-1}^1
        g(\mu)\frac{\eta_0(\mu)}{c}\di{\mu};
\end{equation} 
this Gaussian quadrature rule is exact when $g$ is a polynomial of
degree less than or equal to $2n-1$. It is well known that the
eigenvalues and the squared first components of the (properly scaled)
eigenvectors of a symmetric, tridiagonal matrix $\la{P}_n$ with positive
off-diagonal entries --- a \emph{Jacobi matrix} --- are the nodes and
weights, respectively, of a Gaussian quadrature \cite{Golub:1969:CGQ,
Boley:1987:SMI}.   Thus our task
is to construct the Jacobi matrix $\la{P}_n$ with eigendecomposition 
\begin{equation}\label{eqn:P_n_eig}
    \la{P}_n \la{X} = \la{\Theta}\la{X},
\end{equation}
where the eigenvalues of $\la{P}_n$ are $\theta_j$ and the eigenvectors
$\la{X}_j$ satisfy $\la{X}_i^T\la{X}_k = \delta_{ik}$ (where
$\delta_{ik}$ is the Kronecker delta symbol) and
\begin{equation}\label{eqn:X_j_1}
    (\la{e}_1^T\la{X}_j)^2 = y_j^2/c.
\end{equation}

The entries of the Jacobi matrix $\la{P}_n$ are the coefficients of the
three-term recurrence relation satisfied by the set of polynomials
$\mathcal{P}_n = \{q_0, q_1, \ldots, q_{n-1}\}$, where $q_k$ is a
polynomial of degree less than or equal to $k$ and the polynomials in
$\mathcal{P}_n$ are orthonormal with respect to the weight
$\eta_0(\mu)/c$, i.e., 
\[
    \left\langle q_i, q_k \right\rangle_{\eta_0/c} \equiv 
        \int_{-1}^1 q_i(\mu)q_k(\mu) \frac{\eta_0(\mu)}{c} \di{\mu} 
	= \delta_{ik}.
\] 
Moreover, the
Gaussian quadrature \eqref{eqn:GQ} computes the inner product with
weight $\eta_0/c$ between any two polynomials in this orthonormal set
exactly (since $q_iq_k$ is a polynomial of degree $i+k \le 2n-2$), so 
\[
    \left\langle q_i, q_k \right\rangle_{\eta_0/c} = 
    \frac{1}{c}\sum_{j=1}^n y_j^2q_i(\theta_j)q_k(\theta_j) = 
    \delta_{ik}.
\]

The Jacobi matrix $\la{P}_n$ may be constructed via the Lanczos
algorithm in Algorithm~\ref{alg:Lanczos} (below), which is equivalent to
running the three-term recurrence relation for the set of orthonormal
polynomials $\mathcal{P}_n$.  The appropriate inner product is given by
the normalized spectral measure 
\[
    \langle p,q\rangle_{y,\theta} \equiv
    \frac{1}{c}\sum_{j=1}^n y_j^2 p(\theta_j) q(\theta_j),
\]
which is simply the Gaussian quadrature \eqref{eqn:GQ} applied to
$\left\langle p, q \right\rangle_{\eta_0/c}$ (which is exact
for the polynomials in Algorithm~\ref{alg:Lanczos}).  
\begin{algorithm}[H]
    \caption{Lanczos Algorithm for Computing $\alpha_j$, $\beta_j$.}
    \label{alg:Lanczos}
    \begin{algorithmic}
        \Require $\theta_j$, $y_j$, $j = 1,\ldots,n$
        \Ensure $\alpha_j$ ($j = 1,\ldots,n$) and $\beta_j$ ($j =
	    1,\ldots,n-1$), i.e., the nonzero elements of $\la{P}_n$
	\State Set $q_0(x) \equiv 0$ and $q_1(x) \equiv 1$.
        \For{$j = 1,\ldots,n$}
            \begin{enumerate}
                \item $\alpha_j = \ipyt{q_j}{x q_j}
                    = \left\langle q_j, x q_j\right\rangle_{\eta_0/c}$;
	        \item $r=(x-\alpha_j)q_j-\beta_{j-1}q_{j-1}$;
		\item $\beta_j = \sqrt{\ipyt{r}{r}} = 
		    \sqrt{\left\langle r, r \right\rangle_{\eta_0/c}}$;
                \item $q_{j+1}=\dfrac{r}{\beta_j}$.
	    \end{enumerate}    
        \EndFor
    \end{algorithmic}
\end{algorithm}

Finally, the Chebyshev polynomials of the first kind satisfy the
three-term recursion
\[
    T_{k+1}(x) = 2x T_k(x) - T_{k-1}(x), \qquad
    T_0 = 1, \ T_{-1} = T_1.
\]
This yields the following second-order finite-difference Cauchy problem
for the vector $\la{\varsigma}_k \equiv T_k (\la{P}_n)\la{b}_n$:
\begin{equation}\label{eqn:tsteppingd}
    \frac{\la{\varsigma}_{k+1} - 2 \la{\varsigma}_k +
        \la{\varsigma}_{k-1}}{\tau^2} = 
        \xi(\la{P}_n)\la{\varsigma}_k,
        \qquad \la{\varsigma}_0=\la{b}_n, \
        \la{\varsigma}_{-1}=\la{\varsigma}_{1}
\end{equation}
($\xi$ is defined in \eqref{eqn:xi}).  The recursion
\eqref{eqn:tsteppingd} is the reduced-order version of the recursion
\eqref{eqn:tstepping}; in particular, the $n\times n$ Jacobi matrix
$\la{P}_n$ is our ROM of the propagator $P =
\cos\left(\tau\sqrt{A}\right)$ and $\la{b}_n$ is our ROM of the
source/measurement distribution $b$.  According to \eqref{eqn:f_k}, for
$k = 0,\ldots,2n-1$, our measurements may be written as $f_k =
\left\llangle b, u_k\right\rrangle$, where $u_k$ satisfies
\eqref{eqn:tstepping}.  Similarly, we define the measurements for our
reduced-order recursion in \eqref{eqn:tsteppingd} by 
\begin{equation*}
    f_k^{(n)} \equiv \ltrn{\la{b}_n}{\la{\varsigma}_k}
        = \la{b}_n^T T_k\left(\la{P}_n\right)\la{b}_n \eqfor k =
        0,\ldots,2n-1.
\end{equation*}
Then, according to \eqref{eqn:match-1}, we have $f_k^{(n)} = f_k$ for $k =
0,\ldots,2n-1$, i.e., our reduced-order model matches the data exactly.

We conclude this section with the following lemma, which states that the
reduced-order model matrix $\la{P}_n$ is in fact the projection of $P$
onto the space spanned by the (orthogonalized) snapshots.
\begin{lemma}\label{lem:projection}
    The reduced-order model Jacobi matrix $\la{P}_n$, constructed via
    Algorithm~\ref{alg:Lanczos}, and the vector $\la{b}_n =
    \sqrt{c}\la{e}_1$ are (up to scaling for $\la{b}_n$) the orthogonal
    projections of $P$ and $b$, respectively, onto the Krylov subspace
    \[
        \mathcal{K}_n(u_0,P) = \mathrm{span}\{u_0,\ldots,u_{n-1}\} =
            \mathrm{span}\{\ou_1,\ldots,\ou_n\},
    \]
    i.e., $\la{P}_n = V^*PV$ and
    $\la{b}_n = \frac{1}{\sqrt{c}}V^*b$.
\end{lemma}
\begin{proof}
    The Lanczos algorithm we use to orthogonalize the snapshots, given
    in Lemma~\ref{lem:u_Lanczos}, may be written as
    \begin{equation}\label{eqn:xi_Lanczos}
        \xi(P)V = V\xi\left(\la{T}_n\right) +
        b_{n+1}^u\vartheta_{n+1}\la{e}_n^T,
    \end{equation}
    where $V \equiv \left[\vartheta_1(x), \ldots, \vartheta_n(x)\right]$
    (we have transformed the normalized, orthogonalized snapshots
    $\vartheta_j$ to spatial coordinates $x$) satisfies $V^*V =
    \la{I}_{n\times n}$, $\vartheta_{n+1}$ is orthogonal to
    $\vartheta_j$ for $j = 1,\ldots,n$, and the Jacobi matrix
    \begin{equation}\label{eqn:xi_Tn}
        \xi\left(\la{T}_n\right) = 
        \begin{bmatrix}
            a_1^u & b_1^u  &           &           \\
            b_1^u & a_2^u  & \ddots    &           \\
                  & \ddots & \ddots    & b_{n-1}^u \\
                  &        & b_{n-1}^i & a_n^u     \\
        \end{bmatrix}.
    \end{equation}
    Using \eqref{eqn:xi}, \eqref{eqn:xi_Lanczos} may be rewritten as
    \begin{equation}\label{eqn:P_Lanczos}
        PV = V\la{T}_n 
            + \frac{\tau^2}{2}b_{n+1}^u\vartheta_{n+1}\la{e}_n^T;
    \end{equation}
    $\la{T}_n$ is also a Jacobi matrix, since $\la{T}_n =
    \la{I}_{n\times n} + \frac{\tau^2}{2}\xi\left(\la{T}_n\right)$.
    From \eqref{eqn:P_Lanczos}, we have
    \begin{equation}\label{eqn:PtoTn}
        \la{T}_n = V^*PV,
    \end{equation}
    i.e., $\la{T}_n$ is the projection of $P$ onto
    $\mathcal{K}_n(u_0,P)$.  Thus our goal is to show $\la{T}_n =
    \la{P}_n$.  
    
    The columns of the matrix $Z = U\la{\Phi}$, defined in
    \eqref{eqn:ritz}, form an orthonormal basis of
    $\mathcal{K}_n(u_0,P)$ --- they span $\mathcal{K}_n(u_0,P)$ since
    the columns of $U$ span $\mathcal{K}_n(u_0,P)$ and $\la{\Phi}$ is
    nonsingular, and they are mutually orthogonal since, by
    Lemma~\ref{lem:specH},
    \[
        Z^*Z = \la{\Phi}^TU^*U\la{\Phi} = \la{I}_{n\times n}.
    \]
    Moreover, from \eqref{eqn:H} and \eqref{eqn:generaleig} we have
    \begin{equation}\label{eqn:ZtoTheta}
        Z^*PZ = \la{\Phi}^TU^*PU\la{\Phi}
            = \la{\Phi}^TU^*U\la{H}\la{\Phi}
            = \la{\Phi}^TU^*U\la{\Phi}\la{\Theta}
                \la{\Phi}^TU^*U\la{\Phi}
            = \la{\Theta}.
    \end{equation}
    
    Now, since the columns of $V$ and $Z$ both form orthonormal bases of
    the Krylov subspace $\mathcal{K}_n(u_0,P)$, there is an orthogonal
    matrix $\la{Q}^T_n \in \mathbb{R}^{n\times n}$ such that
    \begin{equation}\label{eqn:ZtoV}
        V = Z\la{Q}^T_n.
    \end{equation}
    Then \eqref{eqn:PtoTn}--\eqref{eqn:ZtoV} imply
    \begin{equation}\label{eqn:Tneigen}
        \la{T}_n 
            = V^*PV 
            = \la{Q}_nZ^*PZ\la{Q}^T_n = \la{Q}_n\la{\Theta}\la{Q}^T_n;
    \end{equation}
    because the $\theta_j$ are distinct (by
    Lemma~\ref{lem:points_of_increase}), \eqref{eqn:Tneigen} is the
    unique unitary eigendecomposition of $\la{T}_n$.  In particular, the
    eigenpairs of $\la{T}_n$ are $\left(\theta_j,
    \la{Q}_n\la{e}_j\right)$ for $j = 1,\ldots,n$.  By \eqref{eqn:ZtoV}
    and \eqref{eqn:quadrature}--\eqref{eqn:wn}, the squared first
    components of the eigenvectors of $\la{T}_n$ are
    \begin{multline*}
        \left(\la{e}_1^T\la{Q}_n\la{e}_j\right)^2
            = \left(\la{e}_1^TV^*Z\la{e}_j\right)^2
            = \left[\left(V\la{e}_1\right)^*U\la{\Phi}\la{e}_j\right]^2
            \\
            = \left[\left(\frac{1}{\sqrt{c}}U\la{e}_1\right)^*U
                \la{\Phi}\la{e}_j\right]^2
            = \frac{1}{c}\left(\la{\chi}^T\la{e}_j\right)^2
            = \frac{y_j^2}{c}.
    \end{multline*}
    Recalling \eqref{eqn:P_n_eig}--\eqref{eqn:X_j_1}, we find that the
    eigenvalues and squared first components of the normalized
    eigenvectors of the Jacobi matrices $\la{T}_n$ and $\la{P}_n$ are
    the same.  Therefore, by the uniqueness of the solution to the
    Jacobi inverse eigenvalue problem (see, e.g., the survey article
    \cite{Boley:1987:SMI} by Boley and Golub and references therein),
    $\la{T}_n = \la{P}_n$; i.e., $\la{P}_n = V^*PV$ is the orthogonal
    projection of $P$ onto $\mathcal{K}_n(u_0,P)$.  
    
    Finally, since the columns of $V$ are orthogonal and the first
    column of $V$ is $b$ (see Algorithm~\ref{alg:algorithm_2}), we have,
    by \eqref{eqn:match-1}--\eqref{eqn:c_again}, 
    \[
        V^*b = b^*b\la{e}_1 = c\la{e}_1 = \sqrt{c}\la{b}_n.
    \]
\end{proof}
\begin{remark}\label{rem:QR}
    The result of Lemma~\ref{lem:projection} suggests the following
    alternative method for computing the reduced-order model $\la{P}_n$.
    Proposition~\ref{prop:GS} implies the matrix $V$ may be constructed
    via Gram--Schmidt orthogonalization; this results in the
    factorization $U = V\bm{\mathcal{R}}$, where $\bm{\mathcal{R}} \in
    \mathbb{R}^{n\times n}$ is an invertible, upper-triangular matrix.
    The matrix $\bm{\mathcal{R}}$ may be computed via a Cholesky
    factorization of the known, symmetric, positive-definite matrix
    $U^*U$ because
    \[
        U^*U = \bm{\mathcal{R}}^TV^*V\bm{\mathcal{R}} =
        \bm{\mathcal{R}}^T\bm{\mathcal{R}}.
    \]
    Then, by Lemma~\ref{lem:projection}, 
    \[
        U^*PU = \bm{\mathcal{R}}^TV^*PV\bm{\mathcal{R}} =
        \bm{\mathcal{R}}^T\la{P}_n\bm{\mathcal{R}}, 
    \]
    from which we obtain
    \[
        \la{P}_n = \bm{\mathcal{R}}^{-T}(U^*PU)\bm{\mathcal{R}}^{-1}.
    \]
    One may also obtain $\la{P}_n$ directly from $\la{H} =
    (U^*U)^{-1}(U^*PU)$ via $\la{P}_n =
    \bm{\mathcal{R}}\la{H}\bm{\mathcal{R}}^{-1}$.
\end{remark}
\begin{remark}\label{rem:GS_causality}
    We emphasize that the Gram--Schmidt procedure used to orthogonalize
    the snapshots respects causality, since each successive snapshot is
    orthogonalized only with respect to the previous snapshots.  The
    importance of this from a physical perspective cannot be
    understated, since the time-domain solutions of the wave equation
    are causal --- all of the linear algebraic
    tools we employ must respect this causality.
\end{remark}


\subsection{Galerkin approximation and algorithm to compute
$\ghat_j$, $\g_j$}\label{subsec:Galerkin}

In the previous section, we computed the entries of the matrix
$\la{P}_n$, namely $\alpha_j$ ($j = 1,\ldots,n$) and $\beta_j$ ($j =
1,\ldots,n-1$), from the data.  Now we want to convert the set of
$\alpha_j$ and $\beta_j$ to $\ghat_j$ and $\g_j$, since $\ghat_j$ and
$\g_j$ are localized averages of the velocity and thus give us direct
information about the unknown velocity.  Although this may be done via
the formulas from Lemma~\ref{lem:u_Lanczos} (after transforming the
$\alpha_j$, $\beta_j$ to $a_j^u$, $b_j^u$ using \eqref{eqn:xi_Tn}), we
prefer the algorithm derived here as it gives deeper insight into the
relationship between the discrete ROM and the continuous problem. In
particular, we use renormalized versions of the orthogonalized snapshots
$\ou_j$, $\ow_j$ as the test and trial functions for a Galerkin method
for the system \eqref{eqn:tstepping1st}.  The coefficients of the
Galerkin method satisfy a finite-difference recursion, and the
eigenvalue problem for this recursion leads to an algorithm that
computes $\ghat_j$ and $\g_j$.  For the remainder of this section, we
assume that eigenvectors of symmetric matrices are normalized to have
Euclidean norm $1$.  

We begin by considering the following Galerkin approximation to
$\tu_k$ and $\tw_k$:  
\begin{equation}\label{eqn:Galerkin}
    \ukn \equiv \ds\sum_{j=1}^n \tmu_{j,k}\ghat_j\ou_j \eqand{and}
    \wkn \equiv \ds\sum_{j=1}^n \tomega_{j,k}\g_j\ow_j \eqfor 
    k = 0,\ldots,2n-1.
\end{equation}
We define $\Skn \equiv [\tmu_{1,k},
\tomega_{1,k}, \tmu_{2,k}, \tomega_{2,k}, \ldots, \tmu_{n,k}, 
\tomega_{n,k}]^T$.  Then
\[
    \myvectorw{\ukn}{\wkn} = \cQ\la{\Gamma}\Skn,
\]
where $\cQ$ is defined in \eqref{eqn:oQ} and $\la{\Gamma}$ is defined in
\eqref{eqn:OGamma}.  In combination with \eqref{eqn:dtau}, a calculation
shows that 
\begin{equation}\label{eqn:dtau_commutes}
    \dtau\cQ\la{\Gamma}\Skn = \cQ\la{\Gamma}\dtau^S\Skn,
\end{equation}
where 
\begin{equation}\label{eqn:dtau_commutes_long}
    \dtau^S\Skn \equiv \dfrac{1}{\tau} 
    \begin{bmatrix}
        \tmu_{1,k+1}-\tmu_{1,k} \\
        \tomega_{1,k}-\tomega_{1,k-1} \\
        \vdots \\ 
        \tmu_{n,k+1}-\tmu_{n,k} \\
        \tomega_{n,k}-\tomega_{n,k-1}
    \end{bmatrix}.
\end{equation}

Recall that $\tu_k$ and $\tw_k$ are the solutions of
\eqref{eqn:leapfrog_matrix}.  Substituting $\ukn$ and $\wkn$ into
\eqref{eqn:leapfrog_matrix} and requiring the resulting equation to be
orthogonal to the columns of $\cQ\la{\Gamma}$ with respect to the inner
product $\ootviptv{\cdot}{\cdot}$ gives the Galerkin method
\[
    \la{\Gamma}\cQ^*\left(\cL\cQ\la{\Gamma}\Skn - 
        \dtau \cQ\la{\Gamma}\Skn\right) = 0.
\]
Then \eqref{eqn:crux_alg_2} (i.e., Algorithm~\ref{alg:algorithm_2}),
\eqref{eqn:bcT}, and \eqref{eqn:dtau_commutes}
imply this is equivalent to 
\[
    \la{\Gamma}\cQ^*\left(\cQ\bcO\la{\Gamma}^{-1} + 
    \frac{1}{\g_n}\oU_{2n+1}\la{e}_{2n}^T\right)\la{\Gamma}\Skn
    - \la{\Gamma}\cQ^*\cQ\la{\Gamma}\partial_{\tau}^S\Skn = 0.
\]
Finally, Algorithm~\ref{alg:algorithm_2} implies $\cQ^*\cQ =
\la{\Gamma}^{-1}$, so the above equation is equivalent to
\begin{equation}\label{eqn:Galerkin_matrix}
    \la{\Gamma}^{-1}\bcO\Skn - \dtau^S\Skn = 0 \eqfor k =
    0,\ldots,2n-1.
\end{equation}
The Galerkin method \eqref{eqn:Galerkin_matrix} is equivalent to the
following finite-difference scheme for the spectral coefficients
$\tmu_{j,k}$, $\tomega_{j,k}$: 
\begin{equation}\label{eqn:Galerkin_recursion}
    \begin{cases}
        \begin{aligned}
            \dfrac{\tmu_{j,k+1}-\tmu_{j,k}}{\tau}
		&= \dfrac{\tomega_{j-1,k}-\tomega_{j,k}}{\ghat_j} \\
            \dfrac{\tomega_{j,k}-\tomega_{j,k-1}}{\tau}
		&= \dfrac{\tmu_{j,k}-\tmu_{j+1,k}}{\g_j}
	\end{aligned}
            &\text{for } j = 1,\ldots,n, \ k = 0,\ldots,2n-1, \\
	\tmu_{n+1,k} = 0, \quad \tomega_{0,k} = 0, \\
        \tmu_{j,0} = \ghat_1^{-1}\delta_{j1}, \quad
	    \tomega_{j,0}+\tomega_{j,-1} = 0.
    \end{cases}
\end{equation}
The boundary conditions $\tmu_{n+1,k} = 0$ and $\tomega_{0,k} = 0$ are
enforced to ensure that the recursions in
\eqref{eqn:Galerkin_recursion} are equivalent to
\eqref{eqn:Galerkin_matrix} for $j = n$ and $j = 1$, respectively.  The
initial conditions $\tmu_{j,0} = \ghat_1^{-1}\delta_{j1}$ and
$\tomega_{j,0} + \tomega_{j,-1} = 0$ are the projections of the
corresponding initial conditions from \eqref{eqn:tstepping1st}: for $i =
1,\ldots,n$ we require
\[
    \ootv{\tu_0^{(n)} - \tu_0}{\ghat_i\ou_i} = 0 
    \Leftrightarrow 
    \ds\sum_{j=1}^n \tmu_{j,0}\ghat_j\ghat_i\ootv{\ou_j}{\ou_i} - 
        \delta_{i1}=0
    \Leftrightarrow
    \tmu_{j,0} = \ghat_1^{-1}\delta_{j1}
\]
and
\[
    \iptv{\left(\tw_0^{(n)}+\tw_{-1}^{(n)}\right)}{\g_i\ow_i} = 0
    \Leftrightarrow 
    \tomega_{j,0}+\tomega_{j,-1} = 0.
\]
Because $\mathrm{span}\{\tu_0,\ldots,\tu_{n-1}\} =
\mathrm{span}\{\ou_1,\ldots,\ou_n\}$, we have $\ukn = \tu_k$ for $k =
0,\ldots,n-1$; similarly, $\wkn = \tw_k$ for $k = 0,\ldots,n-1$.  

We will now derive an algorithm for computing $\ghat_j$, $\g_j$ that is
based on the eigenproblem for the recursion
\eqref{eqn:Galerkin_recursion}.  First, note \eqref{eqn:f_k}
implies 
\begin{equation}\label{eqn:ghat_1}
    \ghat_1^{-1} = \ootv{\ou_1}{\ou_1} = \left\llangle u_0,
    u_0 \right\rrangle = \left\llangle b, b \right\rrangle = f_0 = c,
\end{equation}
where $c$ is defined in \eqref{eqn:c} (and, hence, is known from our
measurements).  Next, we define $\ti{\la{\mu}}_k \equiv
[\tmu_{1,k},\ldots,\tmu_{n,k}]^T$.  We eliminate $\tomega_{j,k}$ from
the recursion \eqref{eqn:Galerkin_recursion} to find that
$\ti{\la{\mu}}_k$ satisfies the second-order recursion
\begin{equation}\label{eqn:mut_second_order}
    \dfrac{\ti{\la{\mu}}_{k+1} - 2\ti{\la{\mu}}_k +\ti{\la{\mu}}_{k-1}}
        {\tau^2} 
    = \la{M}\ti{\la{\mu}}_k \eqfor k = 0,\ldots,2n-1, \qquad
    \ti{\la{\mu}}_0 = \ghat_1^{-1}\la{e}_1, \ \ti{\la{\mu}}_{-1} =
        \ti{\la{\mu}}_1,
\end{equation}
where $\bM \equiv \bDhat^{-1}\bG$, $\bDhat \equiv
\diag(\ghat_1,\ldots,\ghat_n)$, and $\la{G}\in\mathbb{R}^{n\times n}$ is
the Jacobi matrix defined by
\begin{equation*}
    \bG \equiv 
    \begin{bmatrix}
        -\g_1^{-1} & \g_1^{-1} & & \myspace
	\g_1^{-1}  & -\left(\g_1^{-1} + \g_2^{-1}\right) & \ddots &
	    \myspace
	& \ddots & \ddots & \g_{n-1}^{-1} \myspace
	& & \g_{n-1}^{-1} & -\left(\g_{n-1}^{-1} + \g_{n}^{-1}\right)
    \end{bmatrix}.
\end{equation*}
The boundary conditions that are implicit in the definition of $\bM$
(which follow from \eqref{eqn:Galerkin_recursion}) are
\begin{equation}\label{eqn:mu_bdy}
    \tmu_{n+1,k} = 0 
    \eqand{and} 
    \frac{\tmu_{0,k} - \tmu_{1,k}}{\g_0} = 0.
\end{equation}
\begin{remark}\label{rem:optimal_grid}
    The recursion \eqref{eqn:mut_second_order}--\eqref{eqn:mu_bdy} may
    also be viewed as a centered-difference discretization of
    \eqref{eqn:cauchyslow} on a staggered grid with $\ti{\mu}_{j,k} =
    \tu\left(\tx^j\right)$, $\ghat_j = \frac{\tv^j}{\ha{h}^j}$, and
    $\g_j = \ha{v}^j\ti{h}^j$ (see \S~\ref{subsec:leapfrog} for
    more details, in particular \eqref{eqn:discrete_tA}); this matches
    the optimal grid discretization utilized in \cite[equation
    (2.8)]{Borcea:2002:OFD} (with $\sigma$ in that paper replaced by
    $1/\tv$).
\end{remark}

Although $\bM$ is not symmetric, it is self adjoint and negative
definite with respect to the inner product $\ipgh{\cdot}{\cdot}$, where 
\[
    \ipgh{\la{x}}{\la{z}} \equiv \la{x}^T\bDhat\la{z} = 
        \ds\sum_{i=1}^n x_i z_i \ghat_i, \qquad \la{x}, \la{z} \in
        \mathbb{R}^n.
\]
In particular, we may symmetrize $\bM$ as follows:
\begin{equation}\label{eqn:tM}
    \bMt \equiv \bDhat^{1/2}\bM\bDhat^{-1/2} =
        \bDhat^{-1/2}\bG\bDhat^{-1/2} = \bMt^T.
\end{equation}
We make the change of variables $\ti{\la{\varsigma}}_k \equiv
\ha{\la{D}}^{1/2}\ti{\la{\mu}}_k$ in the recursion
\eqref{eqn:mut_second_order} to find $\ti{\la{\varsigma}}_k$ satisfies
\begin{equation}\label{eqn:ti_varsigma}
    \frac{\ti{\la{\varsigma}}_{k+1} - \ti{\la{\varsigma}}_k 
        + \ti{\la{\varsigma}}_{k-1}}{\tau^2} = \bMt\ti{\la{\varsigma}}_k
        \eqfor k = 0,\ldots,2n-1, \qquad \ti{\la{\varsigma}}_0 =
        \ghat_1^{-1/2}\la{e}_1 = \la{b}_n, \ \ti{\la{\varsigma}}_{-1} =
        \ti{\la{\varsigma}}_1,
\end{equation}
where $\la{b}_n$ is defined in \eqref{eqn:tsteppingd}.  We now prove
$\bMt = \xi\left(\la{P}_n\right)$, i.e., we prove
\eqref{eqn:ti_varsigma} and \eqref{eqn:tsteppingd} are equivalent.

The primary Galerkin approximation from \eqref{eqn:Galerkin} may be
written
\begin{equation*}
    \ti{u}_k^{(n)} = V\ti{\la{\varsigma}}_k,
\end{equation*}
where $V = \left[\vartheta_1,\ldots,\vartheta_n\right] =
\left[\ou_1,\ldots,\ou_n\right]\ha{\la{D}}^{1/2}$ is constructed via the
Lanczos algorithm in Lemma~\ref{lem:u_Lanczos}.  Applying the Galerkin
method to \eqref{eqn:primary_recursion} (by inserting $\tu_k^{(n)} =
V\ti{\la{\varsigma}}_k$ into \eqref{eqn:primary_recursion} and
multiplying on the left by $V^*$), we find $\ti{\la{\varsigma}}_k$ also
satisfies the recursion \eqref{eqn:ti_varsigma} with $\bMt =
V^*\xi\left(\ti{P}\right)V = \xi\left(\la{P}_n\right)$ by
Lemma~\ref{lem:projection}.  Thus $\ghat_j$, $\g_j$ may be computed by
comparing $\bMt$ and $\xi\left(\la{P}_n\right)$, the latter of which is
known.  In particular, recalling \eqref{eqn:xi}, \eqref{eqn:P_n_def},
and \eqref{eqn:tM}, we find $\ghat_1 = c^{-1}$ (from
\eqref{eqn:ghat_1}), $\g_1 =
\left[\frac{2}{\tau^2}\left(1-\alpha_1\right)\ghat_1\right]^{-1}$, 
\[
    \ghat_j = \dfrac{\tau^4}{4\beta_{j-1}^2\ghat_{j-1}\g_{j-1}^2},
    \eqand{and} 
    \g_j = \left[\dfrac{2}{\tau^2}\left(1-\alpha_j\right)\ghat_j -
        \dfrac{1}{\g_{j-1}}\right]^{-1} \eqfor j = 2,\ldots,n.
\]

We now present an alternative (equivalent) algorithm for computing
$\ghat_j$, $\g_j$.  This algorithm is a simplification and
beautification of the Lanczos algorithm we have not seen in the
literature, and we utilize a matrix version of it
(Algorithm~\ref{alg:algorithm_1_2D}) for multidimensional problems.  In
the interest of space, we defer its derivation to
\S~\ref{subsec:algorithm_1_derivation} in the appendix.
\begin{algorithm}[H]
    \caption{Computation of $\ghat_j$, $\g_j$}
	\label{alg:algorithm_1}
	\begin{algorithmic}
        \Require $y_l$, $\lambda_l = -\xi(\theta_l)$ for $l =
            1,\ldots,n$ 
	    \Ensure $\ghat_j$, $\g_j$, $j = 1,\ldots,n$
	    \State Set $\obomega_0 = \la{0}$ and
	        $\obmu_1 = \sqrt{0.5}\cdot\left[y_1, y_1, 
		y_2, y_2, \ldots, y_n, y_n 
		\right]^T$.
	    \For{$j = 1,\ldots,n$}
            \begin{enumerate}
                \item 
		    $\ghat_j = \dfrac{1}{\normltrtn{\obmu_j}^2} = 
		        \dfrac{1}{\ds\sum_{i=1}^{2n} \left(\la{e}_i^T
		        \obmu_j\right)^2}$;
                \item 
		    $\obomega_j = \obomega_{j-1} + \ghat_j
		        \la{L}\obmu_j$;
	        \item 
		    $\g_j = \dfrac{1}{\normltrtn{\obomega_j}^2} = 
		        \dfrac{1}{\ds\sum_{i=1}^{2n}
		        \left(\la{e}_i^T\obomega_j\right)^2}$;
                \item 
		    $\obmu_{j+1} = \obmu_{j} - \g_j\la{L}\obomega_j$.
	    \end{enumerate}    
	    \EndFor
	\end{algorithmic}
\end{algorithm}


\section{Inversion algorithm}\label{sec:inversion}

Algorithm~\ref{alg:algorithm_2} (and, equivalently, the Galerkin scheme
from \S~\ref{subsec:Galerkin}) yields the averaging formulas 
\begin{equation}\label{eqn:gammas_again}
    \ghat_j = \frac{1}{\int_0^{\txmax} \left(\ou_j\right)^2
        \frac{1}{\tv} \di{\tx}} 
    \eqand{and}
    \g_j = \frac{1}{\int_0^{\txmax} \left(\ow_j\right)^2 
        \tv \di{\tx}}.
\end{equation}
Lemmas~\ref{lem:u_Lanczos} and \ref{lem:w_Lanczos} imply that the weight
functions $\ou_j$ and $\ow_j$ (up to normalization factors) can be
computed via the Lanczos process with the operators
$\xi\left(\ti{P}\right)$ and $\xi\left(\ti{P}_C\right)$, respectively,
and localized initial conditions.

The proposition below states that the orthogonalized snapshots $\ou_j$
and $\ow_j$ may be equivalently computed via Gram--Schmidt
orthogonalization of the snapshots $\tu_k$ and $\tw_k$, respectively.
One of the well-known interpretations of the
Marchenko--Krein--Gel'fand--Levitan (MKGL) method is that it is a
probing via Gram--Schmidt orthogonalization of the triangular matrix of
the snapshots (the matrix $U$ defined in \eqref{eqn:umatrix})
\cite{Newton:1974:GLM}.  Assuming that $u_0 = b$ is an approximation of
a delta function, due to causality the snapshot matrix $U$ will be an
approximation to a triangular matrix; after Gram--Schmidt
orthogonalization, the orthogonalized snapshots $\ou_j$ and $\ow_j$ will
be localized functions.  This is a result of the fact from linear
algebra that the $QR$-factorization of a full-rank, upper triangular
matrix $\la{U}$ has $\la{Q} = \la{I}$, where $\la{I}$ is the identity
matrix (the rectangular identity matrix if $\la{U}$ is rectangular with
more rows than columns). The proof of the proposition is given in the
appendix.
\begin{proposition}\label{prop:GS}
    Suppose the orthogonalized snapshots $\ou_j$ and $\ow_j$ are
    obtained via Algorithm~\ref{alg:algorithm_2}.  Let $\ougs_j$ denote
    the $j\textsuperscript{th}$ orthogonalized snapshot obtained via the
    Gram--Schmidt algorithm, i.e., 
    \[
        \ougs_j = \tu_{j-1} - \ds\sum_{i=1}^{j-1} c_{ij}^u\ougs_i, 
	\eqand{where}
	c_{ij}^u \equiv \ootv{\tu_{j-1}}{\dfrac{\ougs_i}
	    {\normootv{\ougs_i}}}\dfrac{1}{\normootv{\ougs_i}}.
    \]
    Then $\ougs_j = (d_j^u)^{-1}\ou_j$, where
    \[
        d_j^u \equiv \dfrac{1}{1 - \ds\sum_{i=1}^{j-1} \ghat_i
	\ootv{\tu_{j-1}}{\ou_i}}.
    \]
    Similarly, let $\owgs_j$ denote the $j\textsuperscript{th}$
    orthogonalized snapshot obtained via the Gram--Schmidt algorithm, so 
    \begin{equation}\label{eqn:w_GS}
        \owgs_j = \tw_{j-1} - \ds\sum_{i=1}^{j-1} c_{ij}^w\owgs_i,
	\eqand{where}
	c_{ij}^w \equiv	\iptv{\tw_{j-1}}{\dfrac{\owgs_i}
	    {\normiptv{\owgs_i}}}\dfrac{1}{\normiptv{\owgs_i}}.
    \end{equation}
    Then $\owgs_j = (d_j^w)^{-1}\ow_j$, where
    \[
        d_j^w \equiv \dfrac{\ds\sum_{i=1}^j \ghat_i}
	{(2j-1)\dfrac{\tau}{2} - \ds\sum_{i=1}^{j-1}
	\left(\g_i\iptv{\tw_{j-1}}{\ow_i}\ds\sum_{k=1}^i\ghat_k\right)}.
    \]
\end{proposition}

In addition, in slowness coordinates $\tx$, the orthogonalized snapshots
$\ou_j$ and $\ow_j$ depend weakly on the velocity $\tv$ for small
$\sigma$ (assuming $\tau$ is of the same order as $\sigma$); moreover,
$\ou_j$ and $\ow_j$ are asymptotically proportional to $\tv$ and
$\frac{1}{\tv}$, respectively.  The weak dependence of $\ou_j$ and
$\ow_j$ on $\tv$ and the aforementioned asymptotic behavior of $\ou_j$
and $\ow_j$ can be justified via the Wentzel--Kramers--Brillouin (WKB)
limit.  

We next define a reference velocity that is useful in our inversion
scheme.  
\begin{definition}\label{def:background}
    Let $v^0(x)$ be a (smooth enough) \emph{reference velocity} with
    $v^0(0) = v(0)$.  Then the \emph{reference slowness (traveltime)
    coordinate transformation} is defined by 
    \[ 
        \tx^0(x) \equiv \int_0^x \frac{1}{v^0(x')}\di{x'}.  
    \]  
    The \emph{reference primary} and \emph{dual orthogonalized
    snapshots} $\ou_j^0$ and $\ow_j^0$ and \emph{reference coefficients}
    $\ghat_j^0$ and $\g_j^0$ are computed via
    Algorithm~\ref{alg:algorithm_2} with $\tv$ replaced by $\tv^0$
    (including in the definition of $\tA$).   The reference coefficients
    may be equivalently computed via Algorithm~\ref{alg:algorithm_1}.  
\end{definition}

To see why we require $v^0(0) = v(0)$, note that the PDE in
\eqref{eqn:w2} is equivalent to $g_{xx} - \frac{1}{v^2}g_{tt} =
-v(0)^2\delta(x+0)\delta(t)_t$.  We thus take $v^0(0) = v(0)$ to ensure
that we use the same forcing term for the true and reference velocity
systems.  

Because $\ou_j$ and $\ow_j$ are localized and asymptotically
proportional to $\tv$ and $1/\tv$, respectively,
\eqref{eqn:gammas_again} implies that $\ghat_j$ gives an estimate of
$1/\tv$ near the center of mass of $\ou_j^2$ while $\g_j$ gives an
estimate of $\tv$ near the center of mass of $\ow_j^2$.  Although $\ou_j$
and $\ow_j$ are not known \emph{a priori}, as discussed above they are
weakly dependent on the velocity.  Thus the center of mass of $\ou_j^2$
(respectively, $\ow_j^2$) is well approximated by the center of mass of
$\left(\ou_j^0\right)^2$ (respectively, $\left(\ow_j^0\right)^2$).  

Our inversion algorithm proceeds in two steps.  First, we approximate
the centers of mass of the squared orthogonalized snapshots, for $j =
1,\ldots,n$, by
\begin{equation}\label{eqn:center_of_mass}
        \tx_j^0 \equiv 
        \ghat_j^0\int_0^{\txmax^0} \tx^0
            \left[\ou_j^0\left(\tx^0\right)\right]^2
            \frac{1}{\tv^0\left(\tx^0\right)}\di{\tx^0}, \qquad 
        \ha{\tx}_j^0 \equiv 
        \g_j^0\int_0^{\txmax^0} \tx^0
            \left[\ow_j^0\left(\tx^0\right)\right]^2
            \tv^0\left(\tx^0\right)\di{\tx^0},
\end{equation}
where $\txmax^0 \equiv \tx^0(\xmax)$.  Next, we approximate the velocity
at the preimage of the primary and dual grid points in
\eqref{eqn:center_of_mass} by
\begin{equation}\label{eqn:vel_inv}
    v\left(\tx^{-1}\left(\tx_j^0\right)\right) 
        = \tv\left(\tx_j^0\right) 
        \approx \tv^0\left(\tx_j^0\right)\frac{\ghat_j^0}{\ghat_j}
        \eqand{and} 
    v\left(\tx^{-1}\left(\ha{\tx}_j^0\right)\right)
        = \tv\left(\ha{\tx}_j^0\right)
        \approx \tv^0\left(\ha{\tx}_j^0\right) \frac{\g_j}{\g_j^0}.
\end{equation}
\begin{remark}\label{rem:optimal_grids} 
    Formulas \eqref{eqn:center_of_mass} and \eqref{eqn:vel_inv} will be
    simplified  for $v^0\equiv 1$, in which case $\tx^0=x^0$.  In this
    case, $\ghat_j^0$ and $\g_j^0$ correspond to dual and primary steps,
    respectively, of optimal grids \cite{Borcea:2002:OFD}. That is,
    formulas \eqref{eqn:center_of_mass} and \eqref{eqn:vel_inv} are
    similar to the formulas for optimal grid inversion
    \cite{Borcea:2013:RNA}, except in the latter case $\ha{\tx}^0_j $
    and $\tx^0_j$ are defined as $\sum_{i=1}^j\ghat_i^0$ and
    $\sum_{i=1}^{j-1}\g_i^0$, respectively, for $j=1,\ldots,n$.  When
    $\sigma/\tau$ is close to $\sqrt{2}/4$, these definitions can be
    quite close, but generally they may differ significantly, in which
    case \eqref{eqn:center_of_mass} and \eqref{eqn:vel_inv} will give
    more accurate results than the conventional optimal grid approach.
    One can conjecture that \eqref{eqn:center_of_mass} and
    \eqref{eqn:vel_inv} give a second-order approximation of smooth $v$
    with respect to the width of $\ou_j$ and $\ow_j$, which can be
    measured as $\ghat_j^0$ and $\g_j^0$, respectively.  Generally,
    formulas \eqref{eqn:center_of_mass} and \eqref{eqn:vel_inv} can be
    extended to ``conventional'' optimal grids, in which case we can
    also conjecture that they would produce nodal values very close to
    those of conventional optimal grids \cite{Borcea:2002:OFD}.
\end{remark}

Finally, we may approximately invert the traveltime coordinate
transformation to convert the traveltime grid nodes $\ha{\tx}^0_j$ and
$\tx_j^0$ to physical coordinates.  In particular, since the traveltime
coordinate transformation is given by \eqref{eqn:slowness}, the inverse
traveltime coordinate transformation is
\begin{equation}\label{eqn:inverse_tt}
    \tx^{-1}\left(\tx\right) = \ds\int_0^{\tx} \tv\left(\tx'\right)
    \di{\tx'}.
\end{equation}
Since we only know $\tv$ at the traveltime grid nodes $\ha{\tx}_j^0$ and
$\tx_j^0$, we approximate the above integral via a right-endpoint
Riemann sum.  We obtain the following formulas for the approximate
physical grid nodes, where we take $\tx_0^0 = 0$:
\begin{equation}\label{eqn:physical_grid_nodes}
    \left\{
    \begin{aligned}
        \ha{x}_j^0 &= \ds\sum_{i=1}^j
        \left(\ha{\tx}_i^0-\tx_{i-1}^0\right)
        \tv\left(\ha{\tx}_i^0\right) +
        \ds\sum_{i=1}^{j-1}\left(\tx_i^0-\ha{\tx}_i^0\right)
        \tv\left(\tx_i^0\right) \approx
        \tx^{-1}\left(\ha{\tx}_j^0\right)
        \\
        x_j^0 &= \ds\sum_{i=1}^j
        \left(\ha{\tx}_i^0-\tx_{i-1}^0\right)
        \tv\left(\ha{\tx}_i^0\right) +
        \ds\sum_{i=1}^{j}\left(\tx_i^0-\ha{\tx}_i^0\right)
        \tv\left(\tx_i^0\right) \approx
        \tx^{-1}\left(\tx_j^0\right)
    \end{aligned}
    \right.
    \eqfor j = 1,\ldots,n.
\end{equation}

Our inversion algorithm is summarized in
Algorithm~\ref{alg:1D_inversion}.
\begin{algorithm}[H]
    \caption{1D Inversion Algorithm}
	\label{alg:1D_inversion}
	\begin{algorithmic}
        \Require measured data $f_k (k = 0,\ldots,2n-1)$, reference
        velocity $v^0$
        \Ensure approximations of
            $v\left(\tx^{-1}\left(\ha{\tx}_j^0\right)\right)$ and
            $v\left(\tx^{-1}\left(\tx_j^0\right)\right)$
        \begin{itemize}
            \item[1.] Compute the grid nodes $\ha{\tx}_j^0$ and $\tx_j^0$
                for $j=1,\ldots,n$.
                \begin{itemize}
                    \item[a.] Compute the reference primary and dual
                        snapshots by solving \eqref{eqn:1stc} with $\tv$
                        replaced by $\tv^0$ (including in the traveltime
                        coordinate transformation) using finite
                        differences, for example.  
                    \item[b.] Orthogonalize the reference snapshots via
                        Algorithm~\ref{alg:algorithm_1} to
                        obtain $\ou^0_j$, $\ow^0_j$, $\ghat^0_j$, and
                        $\g^0_j$ for $j = 1,\ldots,n$.
                    \item[c.] Compute the traveltime grid nodes
                        $\ha{\tx}_j^0$ and $\tx_j^0$ from
                        \eqref{eqn:center_of_mass}
                        using the trapezoidal rule, for example. 
                \end{itemize}
            \item[2.] Compute $c = f_0$ and $\theta_j$, $y_j$ ($j =
                1,\ldots,n$) using \eqref{eqn:datamatrix} and
                \eqref{eqn:H}--\eqref{eqn:wn}.
            \item[3.] Compute $\ghat_j$, $\g_j$ ($j = 1,\ldots,n$) via
                Algorithm~\ref{alg:algorithm_1}.  
            \item[4.] Compute the approximation of the velocity on the
                traveltime grid, i.e., $\tv\left(\ha{\tx}_j\right)$ and
                $\tv\left(\tx_j\right)$, from \eqref{eqn:vel_inv}.
            \item[5.] Approximately convert the traveltime grid nodes
                $\ha{\tx}_j^0$ and $\tx_j^0$ to physical grid nodes
                $\ha{x}_j$ and $x_j$ 
                using \eqref{eqn:physical_grid_nodes}.
            \item[6.] Combine the results from steps $4$ and $5$ to obtain
                the estimate of the velocity at the (approximate)
                physical grid nodes, namely 
                $v\left(\ha{x}_j\right) \approx
                \tv\left(\ha{\tx}_j^0\right)$ and 
                $v\left(x_j\right) \approx
                \tv\left(\tx_j^0\right)$.
        \end{itemize}
   	\end{algorithmic}
\end{algorithm}


\section{Numerical experiments}\label{sec:numerics}

We now present some numerical results to illustrate the main ideas of
the paper.  In all of our simulations, we used a uniform reference
velocity given by $v^0(x) \equiv v(0)$.  A comparison of the
performance of 2D reverse time migration (RTM) and a 2D
backprojection method closely related to the method described in this
paper may be found in \cite{Mamonov:2015:NSI}.

In Figure~\ref{fig:test_velocity}(a), we plot the snapshot matrix $U$
defined in \eqref{eqn:umatrix}.  In Figure~\ref{fig:test_velocity}(b),
we plot the orthogonalized snapshots $\ou_j$ constructed using
Algorithm~\ref{alg:algorithm_2}; note the localization of the
orthogonalized snapshots.  In Figures~\ref{fig:test_velocity}(a) and
(b), we have scaled the snapshots so that
$\normootv{\tu_j}=\normootv{\ou_j} = 1$.  The velocity we used in the
simulation is represented by the solid, black line in
Figure~\ref{fig:test_velocity}(c).  We mapped the grid points $\tx_j^0$
and $\ha{\tx}_j^0$ to the spatial grid by approximately inverting the
map $\tx(x)$ via \eqref{eqn:physical_grid_nodes}.  The approximations to
$v\left(\tx^{-1}\left(\tx_j^0\right)\right)$ and
$v\left(\tx^{-1}\left(\ha{\tx}_j^0\right)\right)$ are represented by
blue circles and green squares, respectively.  We chose $\sigma = 0.01$
and $\tau = 2.5\sigma$ for these simulations.  At this point, we do not
have a rigorous method for optimally choosing $\tau$; as mentioned
above, we conjecture that we should choose $\tau$ to be consistent with
the Nyquist--Shannon sampling limit of $\tq$, so $\tau \sim \sigma$.
Below we will see that even certain choices of $\tau \sim \sigma$ lead
to good reconstructions while other choices of $\tau \sim \sigma$ can
lead to very poor reconstructions.  As a measure of the stability of our
algorithm, we computed the condition number of the matrix $U^*U$ (see
\eqref{eqn:umatrix} and \eqref{eqn:H}).  For the above parameters, we
have $\mathrm{cond}(U^*U) \approx 61.76$.  

If $\tau$ is too large, the inversion procedure produces poor results.
Figures~\ref{fig:test_velocity}(d), (e), and (f) are the analogues of
Figures~\ref{fig:test_velocity}(c), (b), and (a), respectively, in the
case where $\tau = 3.5\sigma$.  The orthogonalized snapshots in
Figure~\ref{fig:test_velocity}(f) ($\tau = 3.5\sigma$) are not as
localized as those in Figure~\ref{fig:test_velocity}(b) ($\tau =
2.5\sigma$); the quality of the inversion suffers as well.  However, the
algorithm is stable in the sense that $\mathrm{cond}(U^*U) \approx
13.13$.  

Finally, we ran a simulation with $\tau = 0.5\sigma$.  In this case the
algorithm runs into stability issues, a problem heralded by the fact
that $\mathrm{cond}(U^*U) \approx 1.55 \times 10^9$.  

These numerical experiments suggest that an appropriate value of $\tau$
may be chosen by first selecting a relatively large value of $\tau
\sim\sigma$ and decreasing it until $\mathrm{cond}(U^*U)$ becomes too
large.

These results can be understood from a physical point of view.  If
$\tau$ is too large, the wave travels too far between consecutive
measurements, so the corresponding snapshots have disjoint supports.
Since our method obtains the image from the projection of the propagator
onto the subspace of the snapshots, if there are regions of the domain
not covered by the supports of the snapshots there is no way for us to
reconstruct the velocity in those regions.  If $\tau$ is too small, the
snapshots overlap too much and become almost linearly dependent, which
leads to a large condition number for the Gram matrix $U^*U$.
\begin{figure}
    \begin{center}
    \begin{tabular}{c c}
    \includegraphics[width=0.37\textwidth]{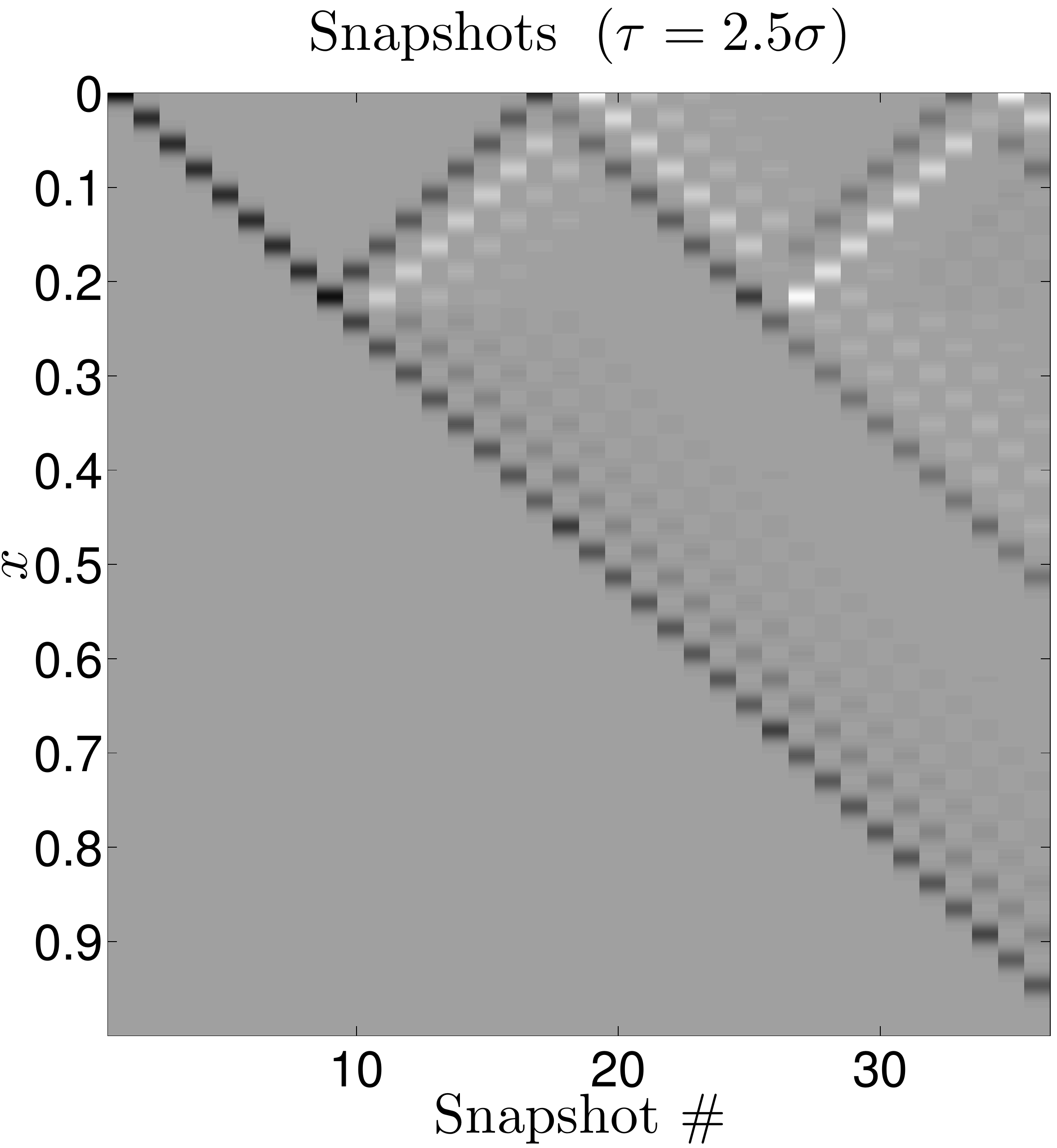} & 
	\includegraphics[width=0.37\textwidth]{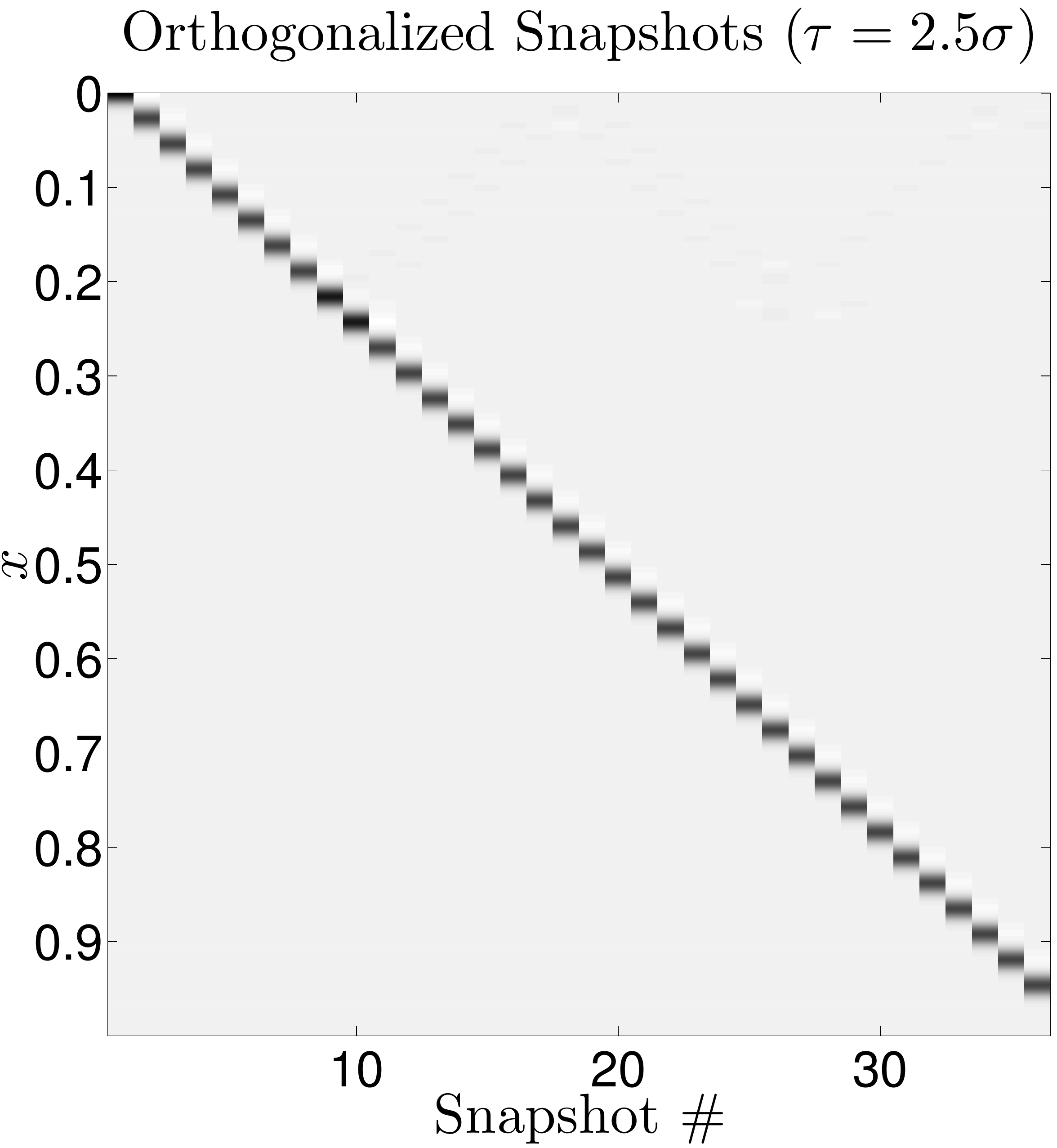} \\
	(a) & (b) \myspace
	\includegraphics[width=0.37\textwidth]{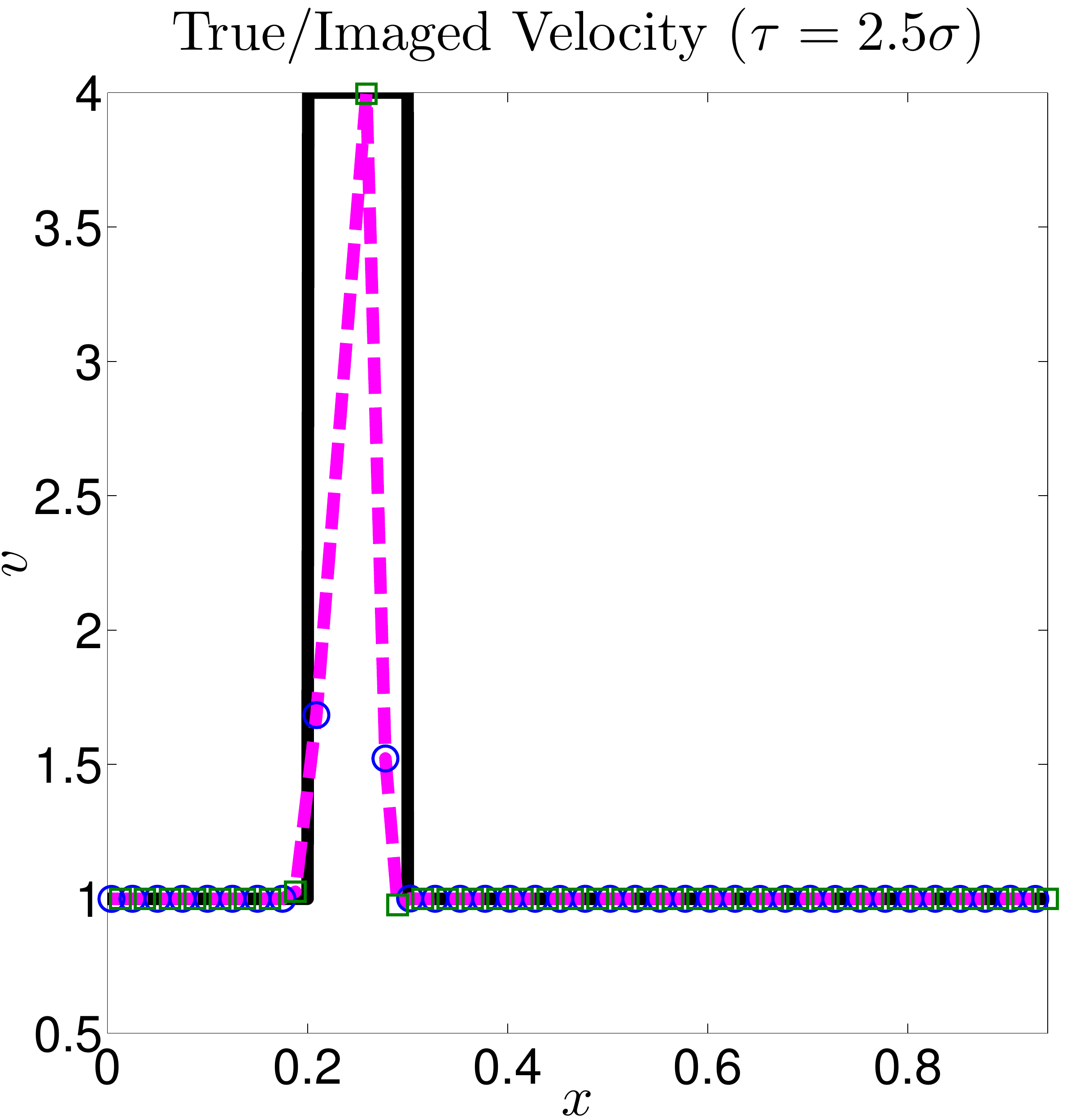} & 
	\includegraphics[width=0.37\textwidth]{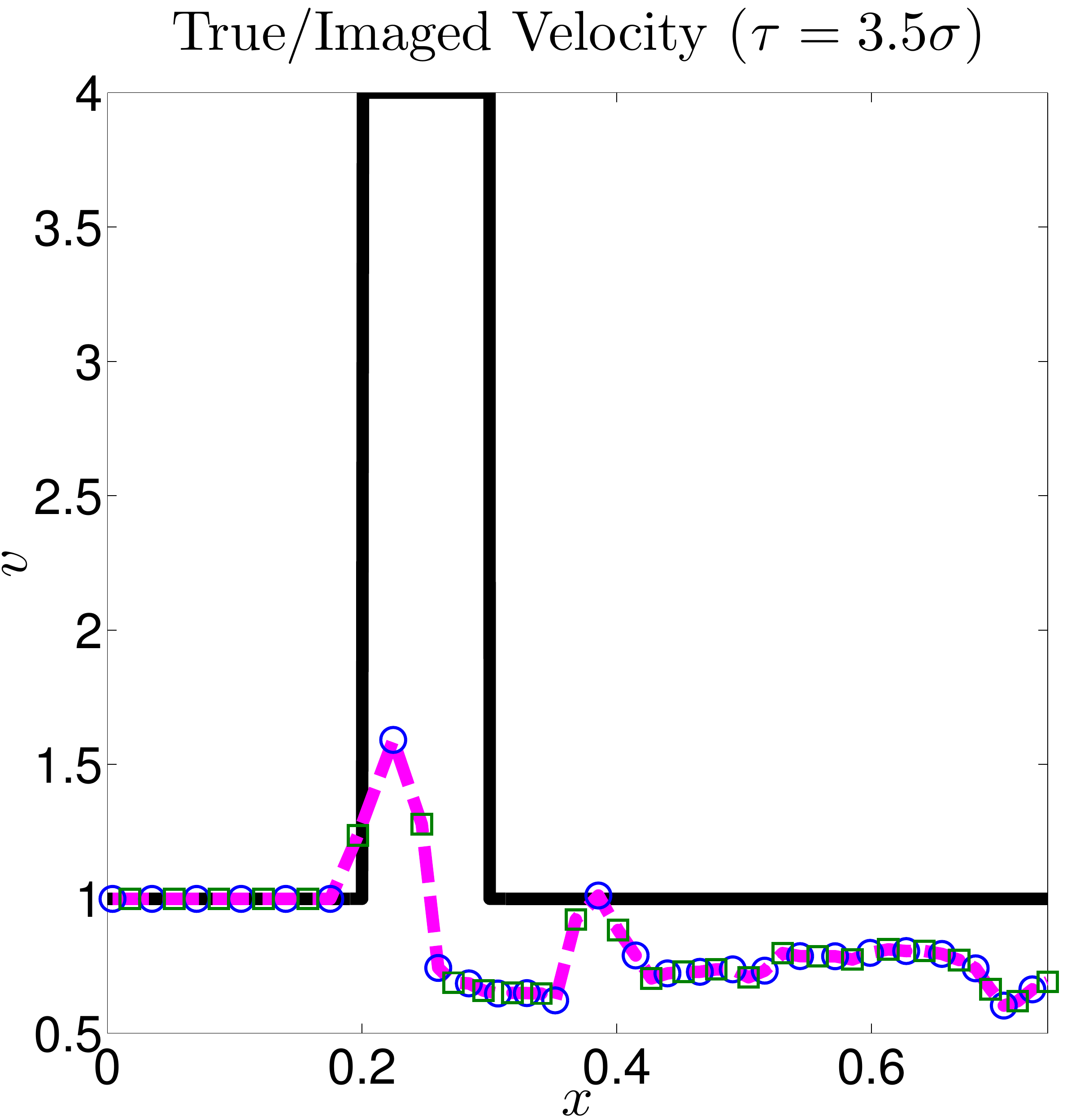} \\
	(c) & (d) \myspace
	\includegraphics[width=0.37\textwidth]{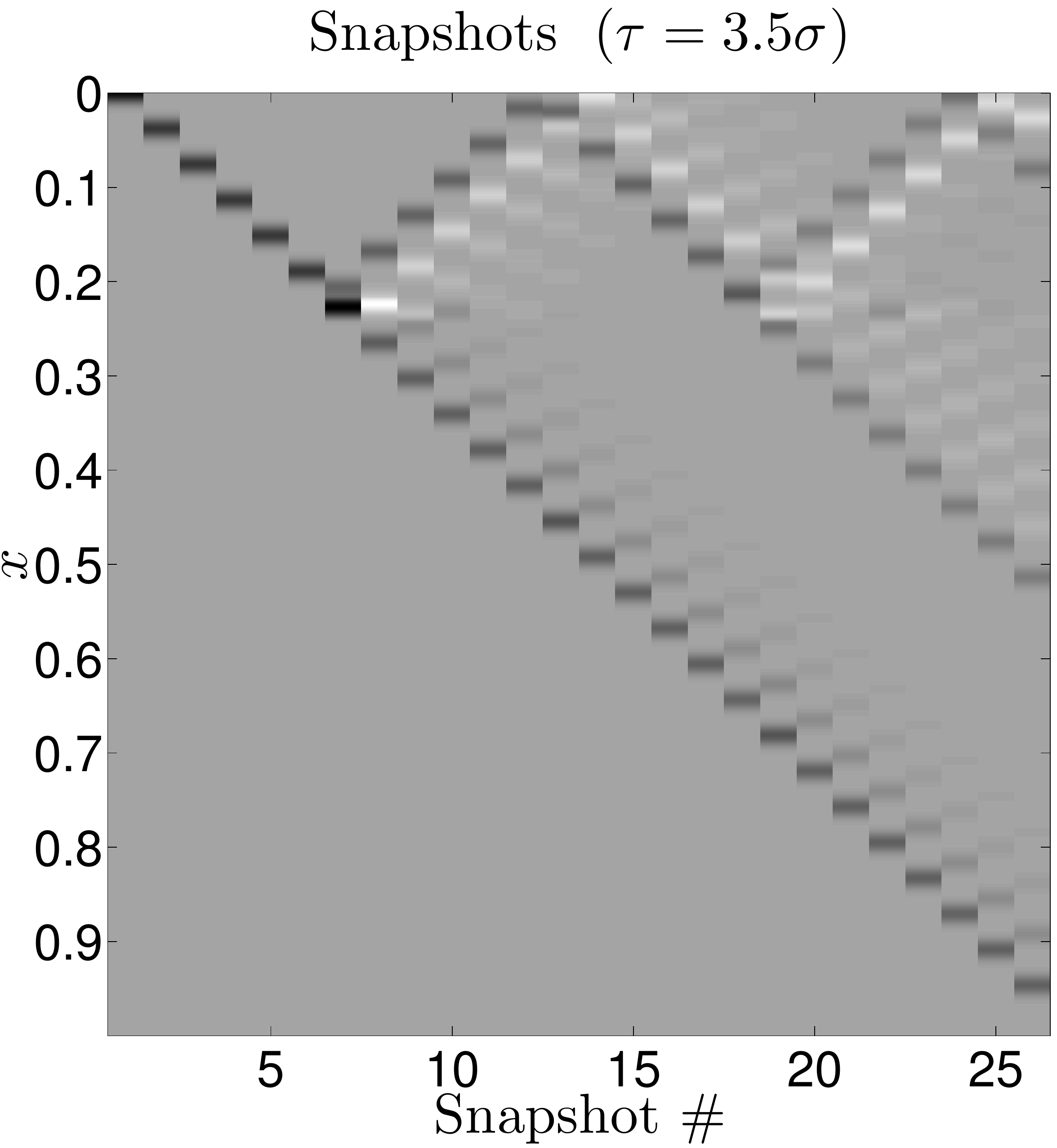} & 
	\includegraphics[width=0.37\textwidth]{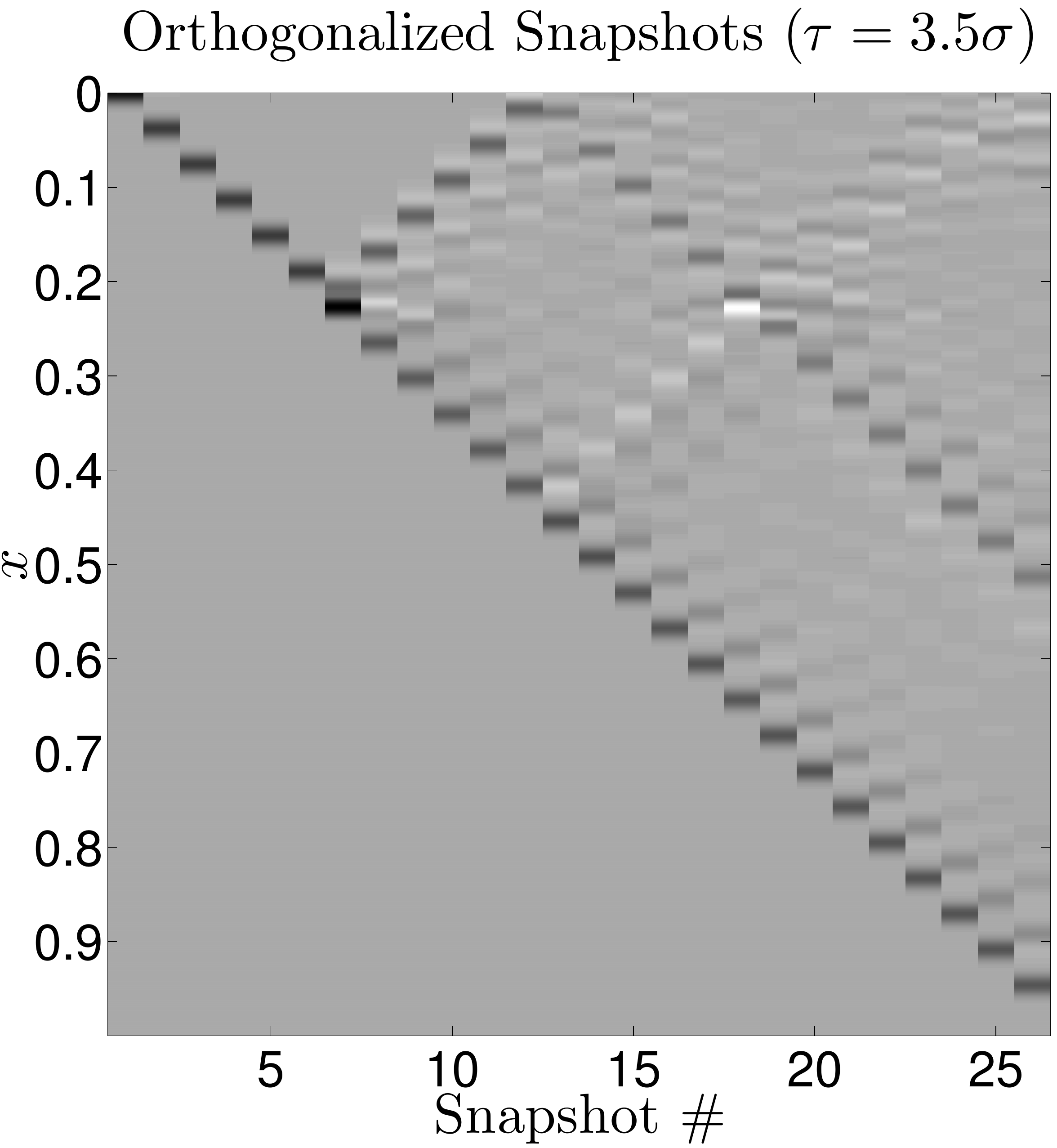} \\
	(e) & (f)
    \end{tabular}
    \caption{In this figure, we show that the choice of $\tau$ can have
    a large influence on the localization properties of the
    orthogonalized snapshots and the quality of the inversion.  (a) The
    primary snapshots $u_k$ for the velocity model illustrated in (c);
    (b) the orthogonalized primary snapshots $\ou_j$ generated by
    Algorithm~\ref{alg:algorithm_2} (converted to the spatial coordinate
    $x$); (c) the true velocity model (solid, black line) and inversion
    results for $\tau = 2.5\sigma$ --- the blue circles are
    approximately located at
    $\tx^{-1}\left(\tx_j^0\right)$ and the green squares are
    approximately located at $\tx^{-1}\left(\ha{\tx}_j^0\right)$.  (d)
    The true velocity model and inversion results when $\tau =
    3.5\sigma$; (e) the primary snapshots for the velocity model in (d);
    (f) the orthogonalized primary snapshots for the velocity model in
    (d).}
    \label{fig:test_velocity}
    \end{center}
\end{figure}

In Figure~\ref{fig:velocities_1_2}, we plot the primary snapshots,
orthogonalized primary snapshots, and inversion results for two
additional velocity models.  The first velocity model is illustrated in
by the solid, black line in Figure~\ref{fig:velocities_1_2}(c).  We
chose $\tau = \sigma$ for this simulation. The orthogonalized snapshots
in Figure~\ref{fig:velocities_1_2}(b) are quite localized.  In this
case, $\mathrm{cond}(U^*U) \approx 4.11 \times 10^3$.  

The second velocity model, illustrated in
Figure~\ref{fig:velocities_1_2}, consists of two smooth inclusions and a
discontinuous inclusion.  We chose $\tau = 1.5\sigma$, which gives
$\mathrm{cond}(U^*U) \approx 28.10$.  
\begin{figure}
    \begin{center}
    \begin{tabular}{c c}
        \includegraphics[width=0.38\textwidth]{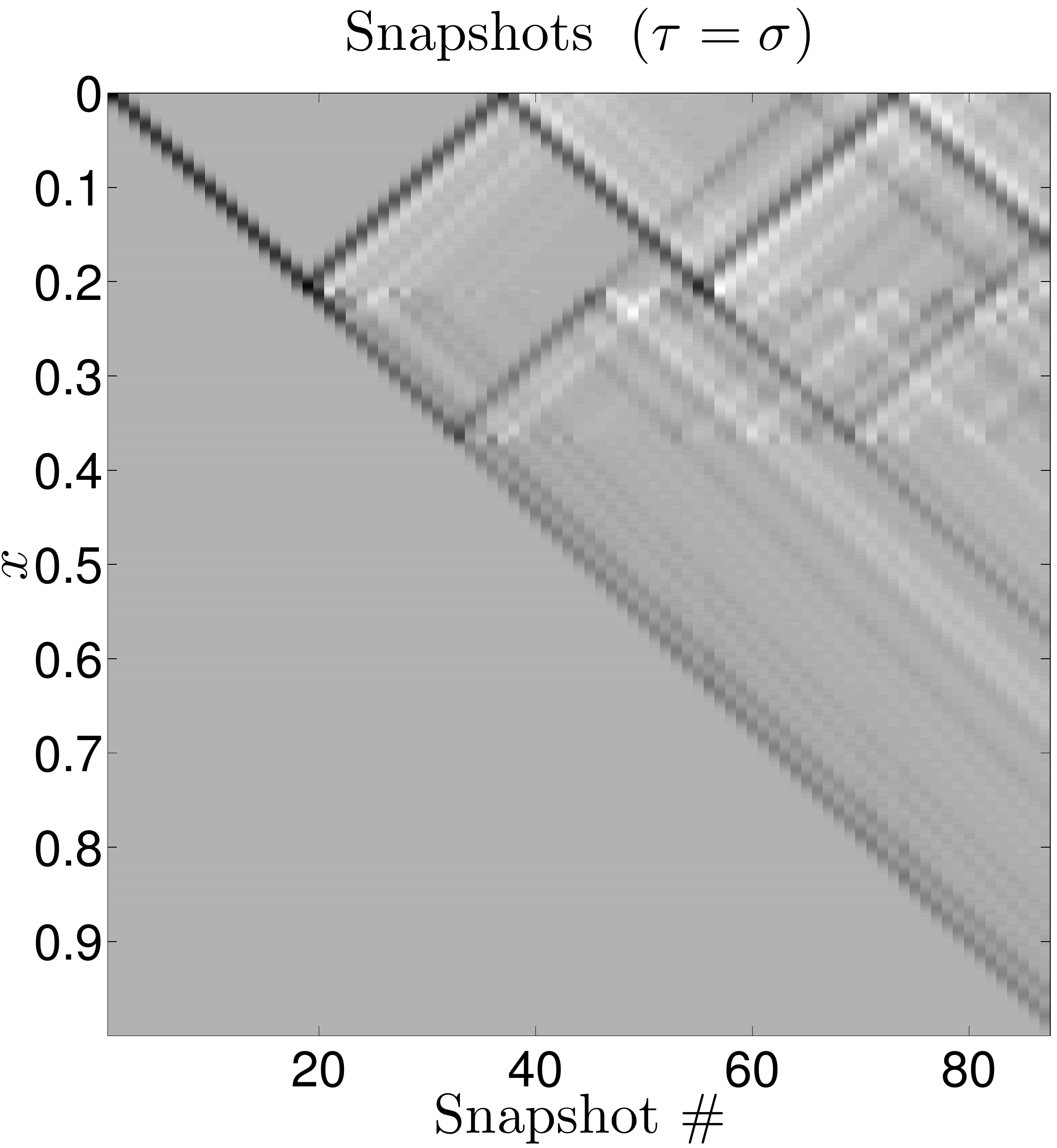} & 
	    \includegraphics[width=0.38\textwidth]{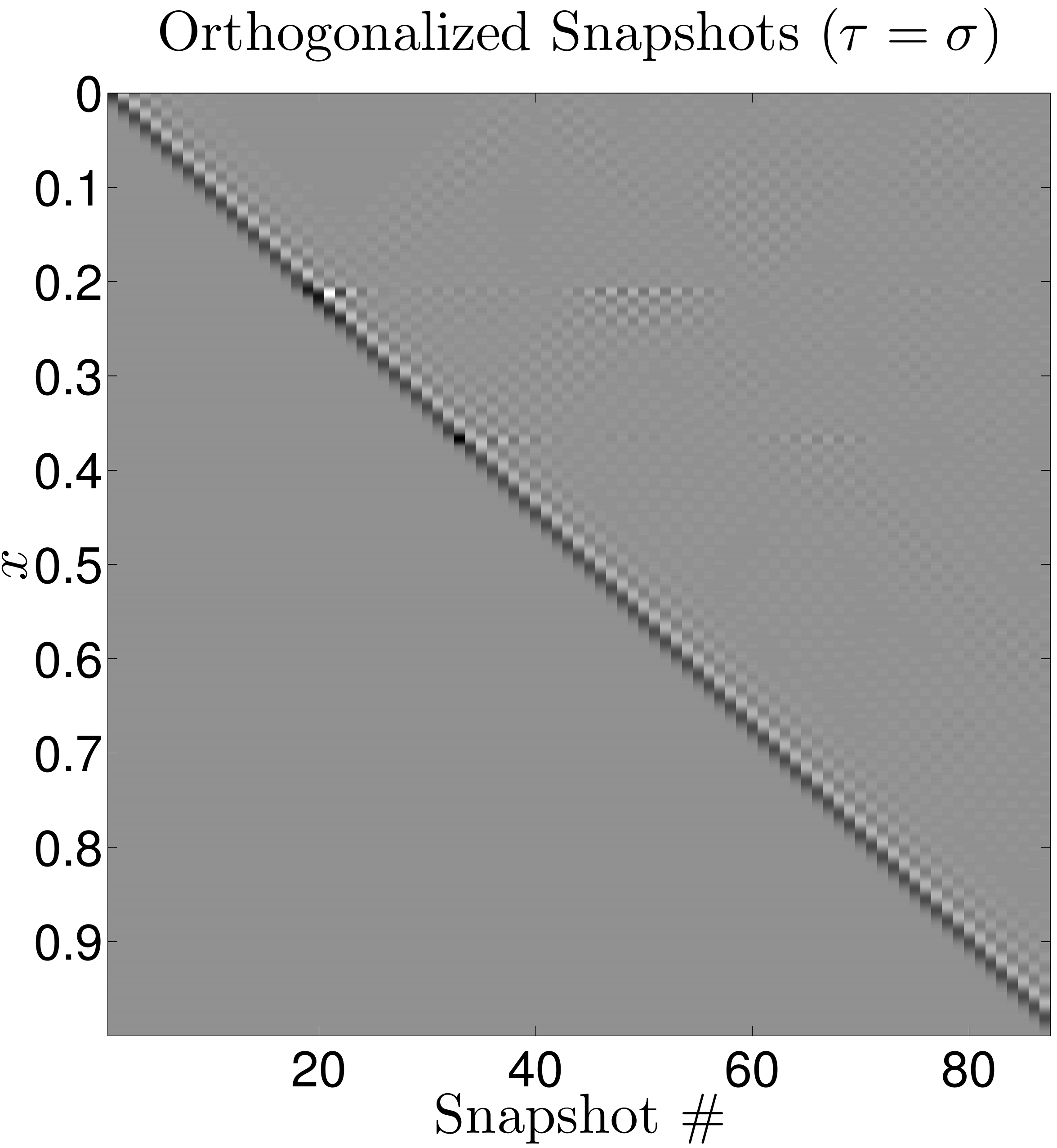} \\
	    (a) & (b) \myspace
	    \includegraphics[width=0.38\textwidth]{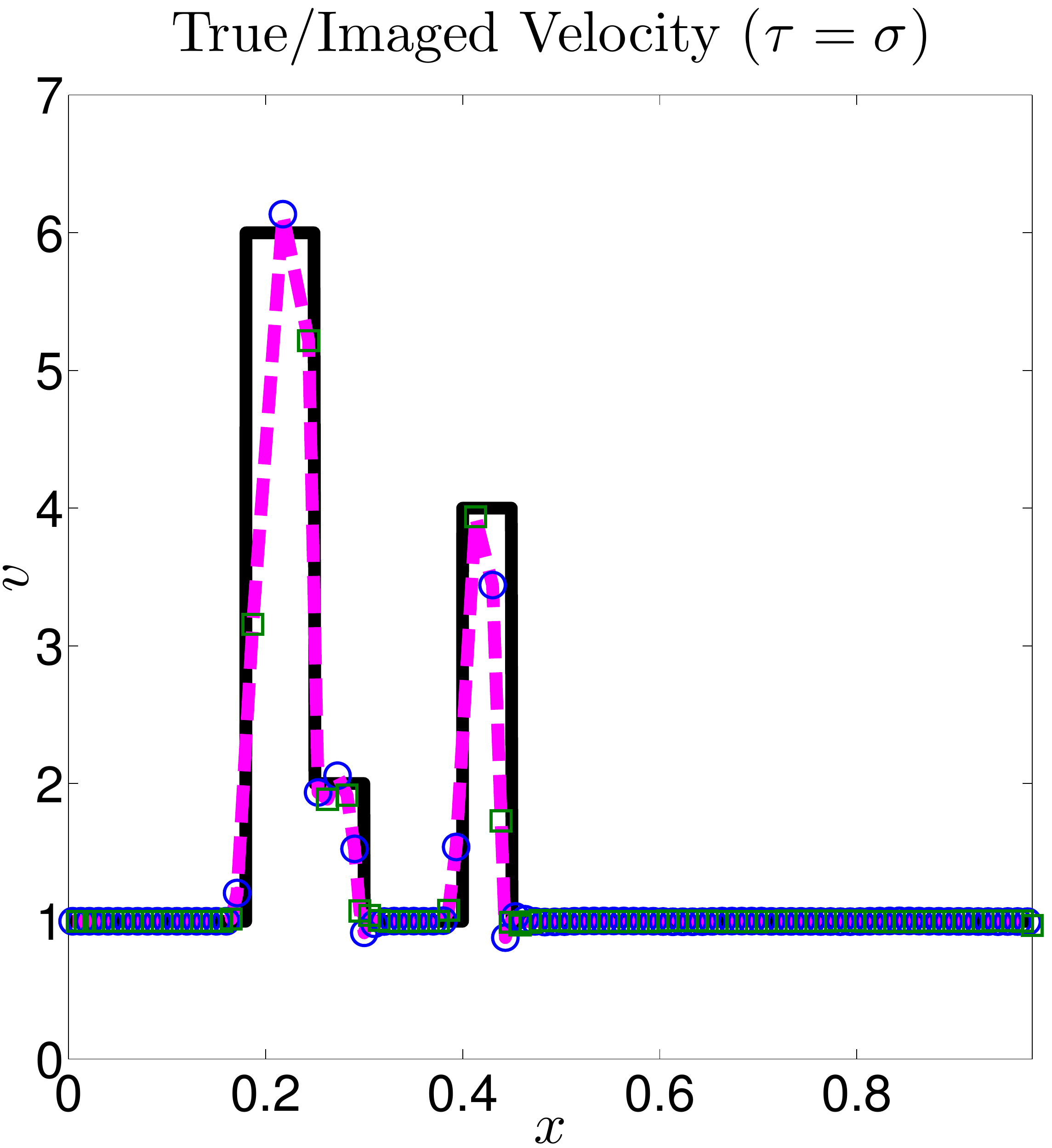} & 
	    \includegraphics[width=0.38\textwidth]{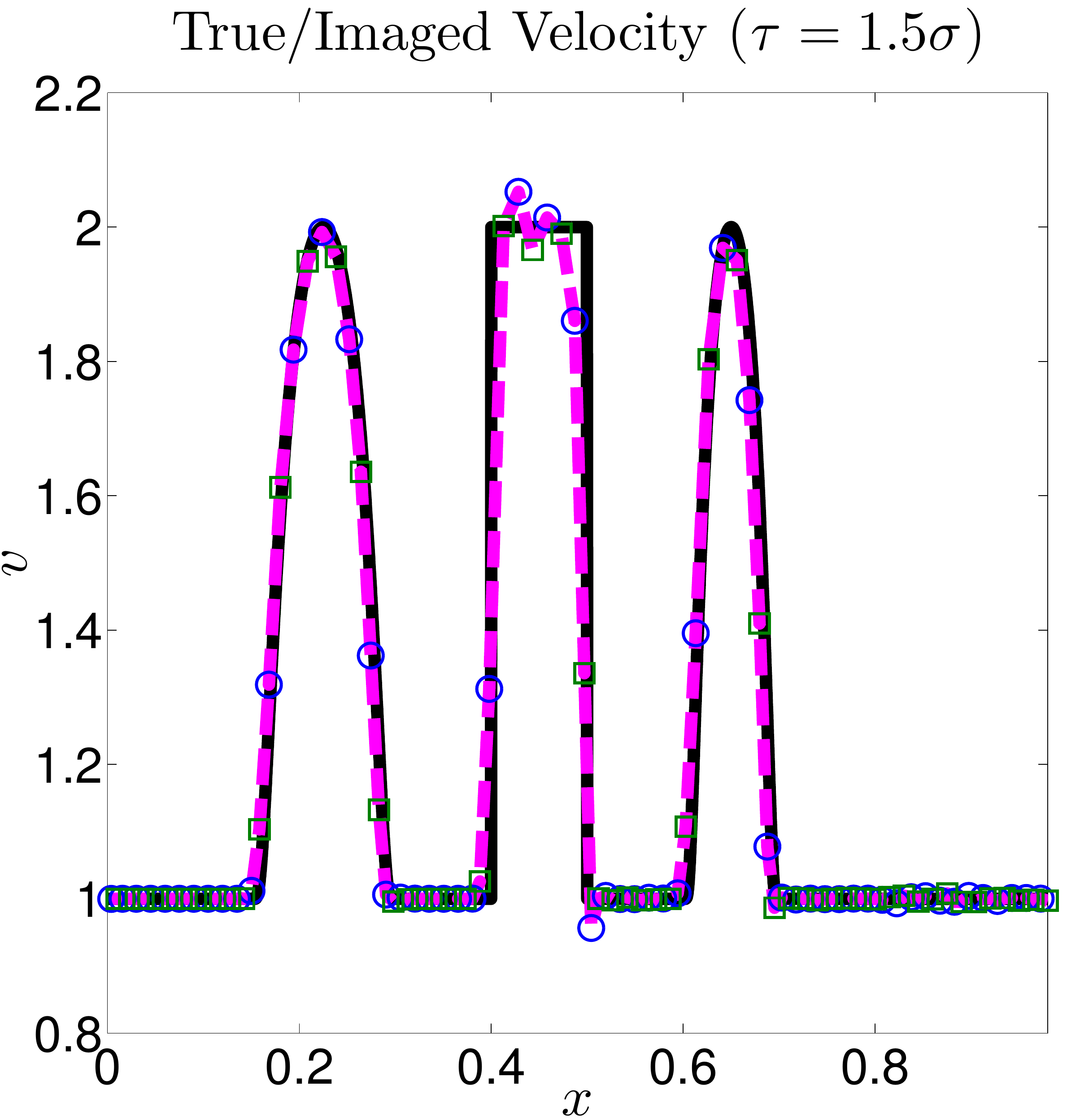} \\
	    (c) & (d) \myspace
	    \includegraphics[width=0.38\textwidth]{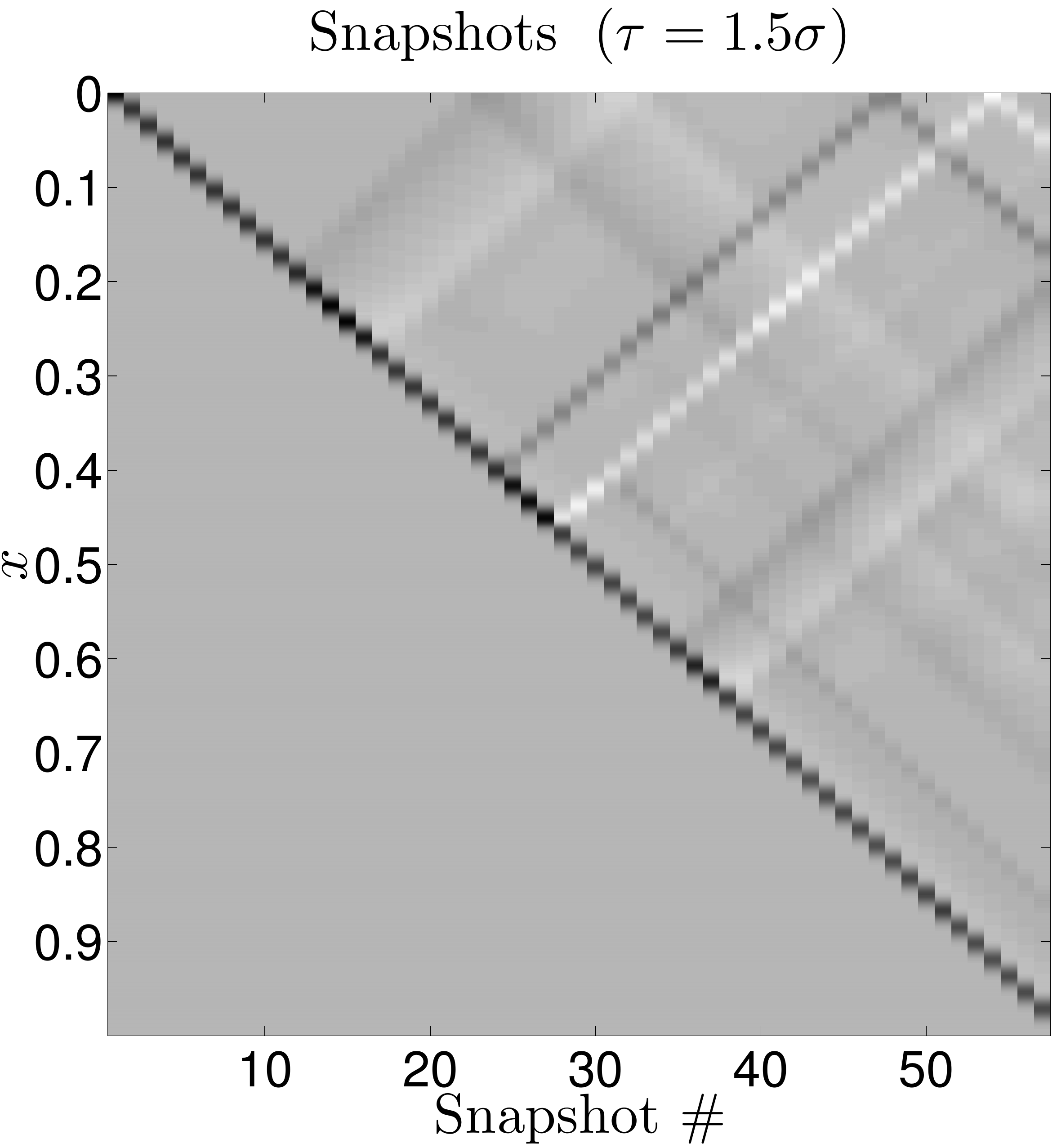} & 
	    \includegraphics[width=0.38\textwidth]{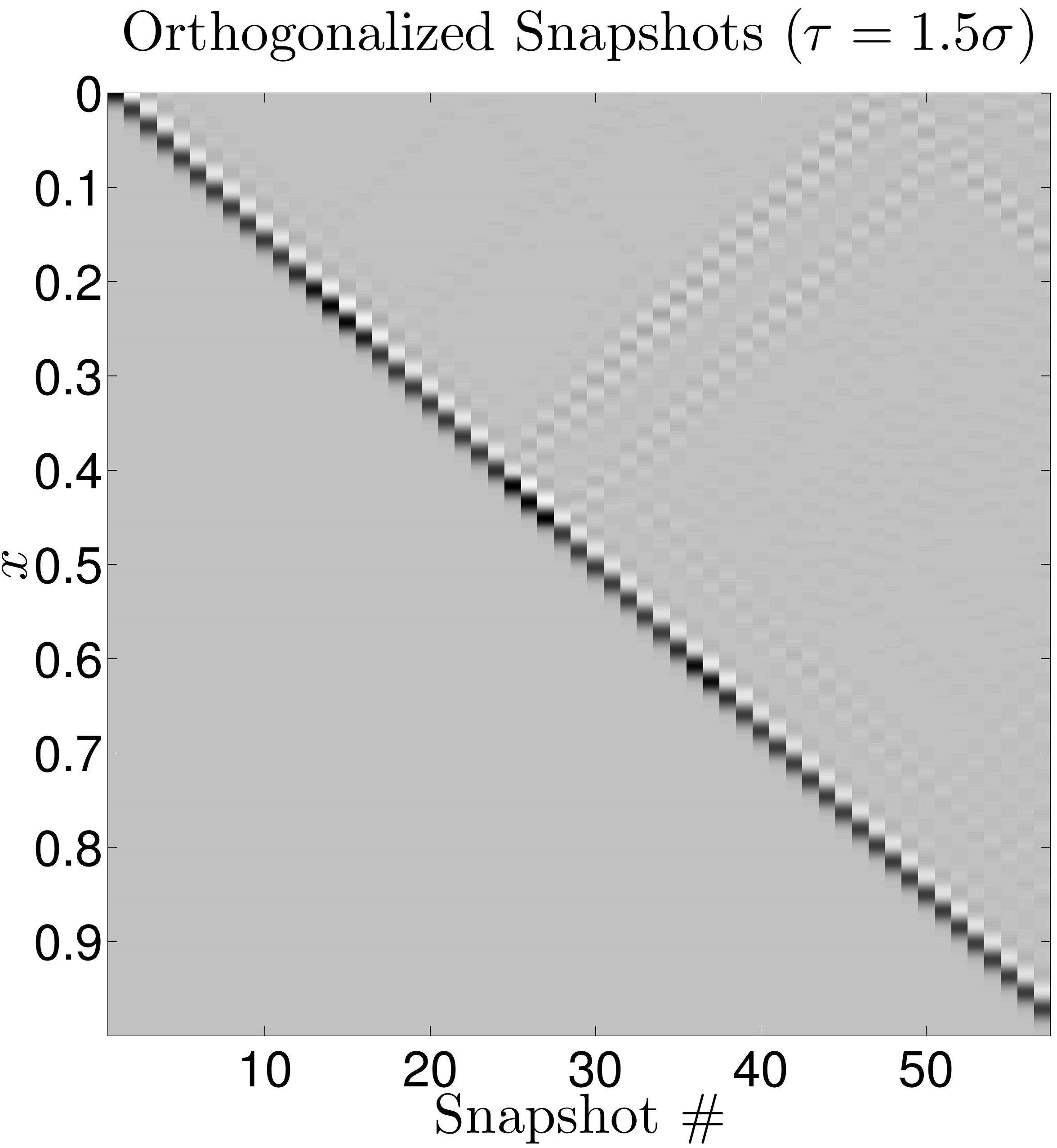} \\
	    (e) & (f)
    \end{tabular}
    \caption{(a) The primary snapshots for the velocity model in (c);
    (b) the orthogonalized primary snapshots for the velocity model in
    (c); (c) the velocity model is drawn as a solid, black line, while
    the inversion results for $\tau = \sigma$ are represented by the
    blue circles ($\tx^{-1}\left(\tx_j^0\right)$) and green squares
    ($\tx^{-1}\left(\ha{\tx}_j^0\right)$).  (d) Another velocity model
    and inversion results; (e) the primary snapshots for the velocity
    model in (d); (f) the primary orthogonalized snapshots for the
    velocity model in (d).}
    \label{fig:velocities_1_2}
    \end{center}
\end{figure}

Finally, we justify our use of the centers of mass of the reference
squared orthogonalized snapshots for the grid points in
\eqref{eqn:center_of_mass} instead of the centers of mass of the squared
orthogonalized snapshots for the true medium (which are unknown in
practice).  In Figure~\ref{fig:com}, the blue squares represent the true
centers of mass of the primary squared orthogonalized snapshots, i.e.,
the height of the $j\textsuperscript{th}$ blue square is 
\begin{equation}\label{eqn:true_com}
    \tx^{-1}\left(\ghat_j\int_0^{\txmax} \left[\ou_j(\tx)\right]^2
        \dfrac{1}{\tv(\tx)} \di{\tx} \right)
    =\ghat_j\int_0^{\xmax} \left[\ou_j(x)\right]^2
        \dfrac{1}{v(x)^2}\di{x}.  
\end{equation}
The green circles represent the centers of mass of the primary squared
orthogonalized snapshots for the (uniform) reference medium, i.e., the
height of the $j\textsuperscript{th}$ green circle is 
\begin{equation}\label{eqn:true_com_0}
    \tx^{-1}\left(\ghat_j^0\int_0^{\txmax^0} 
        \left[\ou^0_j(\tx^0)\right]^2
        \dfrac{1}{\tv^0(\tx^0)} \di{\tx^0} \right)
    = \ghat_j^0\int_0^{\xmax}\left[\ou^0_j(x)\right]^2
        \dfrac{1}{v^0(x)^2} \di{x}.  
\end{equation}
In practice, the map $\tx^{-1}$ cannot be computed exactly since $\tv$
is not known \emph{a priori}.  The red asterisks in Figure~\ref{fig:com}
represent the centers of mass of the reference squared orthogonalized
snapshots that are approximately converted to true coordinates using our
imaged velocity from \eqref{eqn:vel_inv} and a Riemann sum approximation
of the integral in \eqref{eqn:inverse_tt}, namely the formulas from
\eqref{eqn:physical_grid_nodes}; these are the grid points used in the
inversion scheme (and are those shown in
Figures~\ref{fig:test_velocity}(c) and (d) and
Figures~\ref{fig:velocities_1_2}(c) and (d)).  In particular,
Figure~\ref{fig:com}(a) corresponds to the velocity model in
Figure~\ref{fig:test_velocity}(c), Figure~\ref{fig:com}(b) corresponds
to the velocity model in Figure~\ref{fig:test_velocity}(d),
Figure~\ref{fig:com}(c) corresponds to the velocity model in
Figure~\ref{fig:velocities_1_2}(c), and Figure~\ref{fig:com}(d)
corresponds to the velocity model in Figure~\ref{fig:velocities_1_2}(d).
We note that the centers of mass agree quite well (to within a few
percent or less) if $\tau$ is chosen appropriately (as in
Figures~\ref{fig:com}(a), (b), and (d)), while they differ significantly
(around $28$\%) if $\tau$ is chosen poorly (Figure~\ref{fig:com}(b)).
There are even certain choices of $\tau$ for the velocity model in
Figure~\ref{fig:com}(b) for which the grid points are not monotonically
increasing --- in particular, the orthogonal snapshots have large values
far away from the peak centered near the ``optimal'' grid point, which
leads to a poor approximation of the true center of mass.  
\begin{figure}
    \begin{center}
    \begin{tabular}{c c}
        \includegraphics[width=0.42\textwidth]{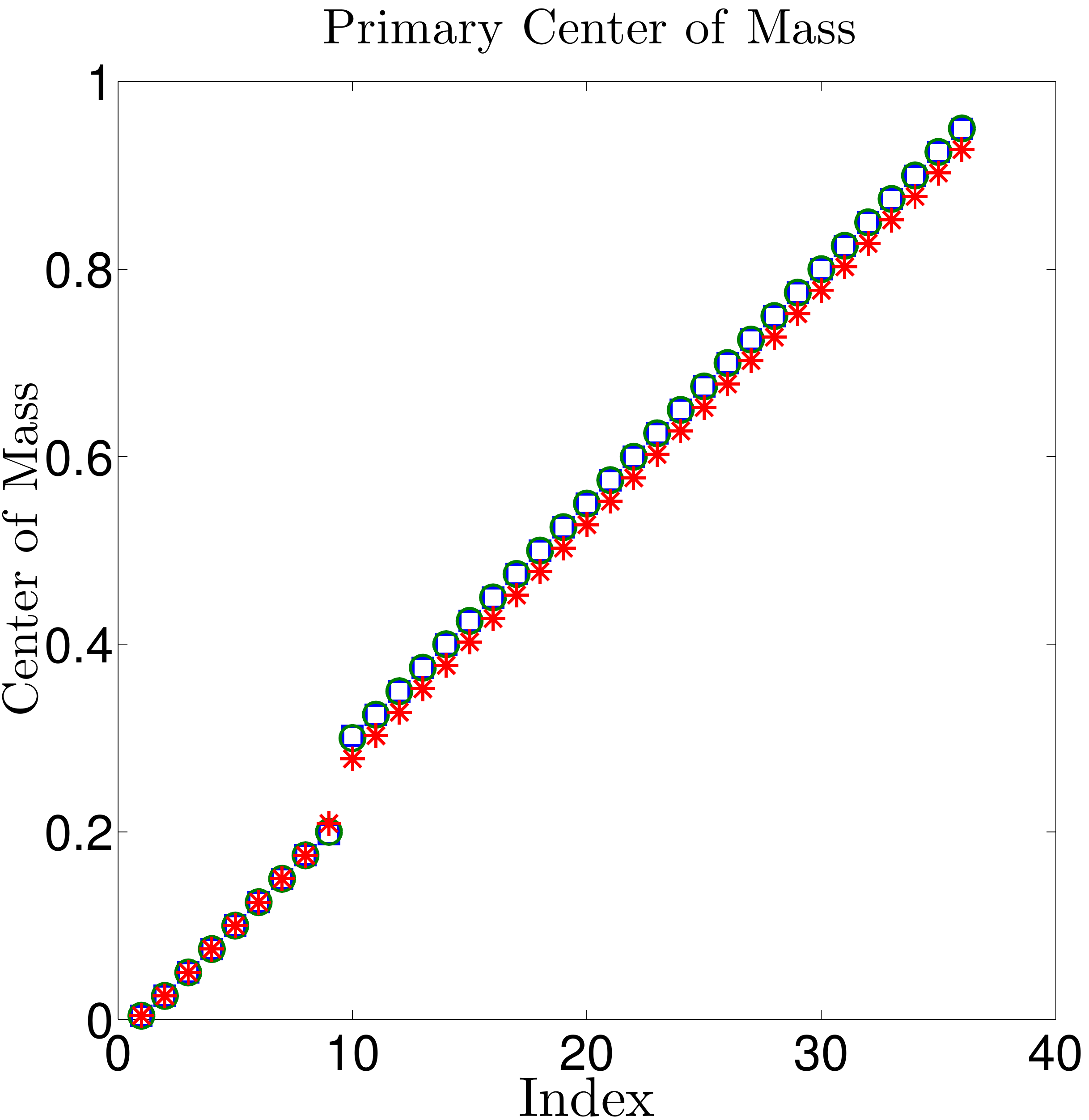} & 
	    \includegraphics[width=0.42\textwidth]{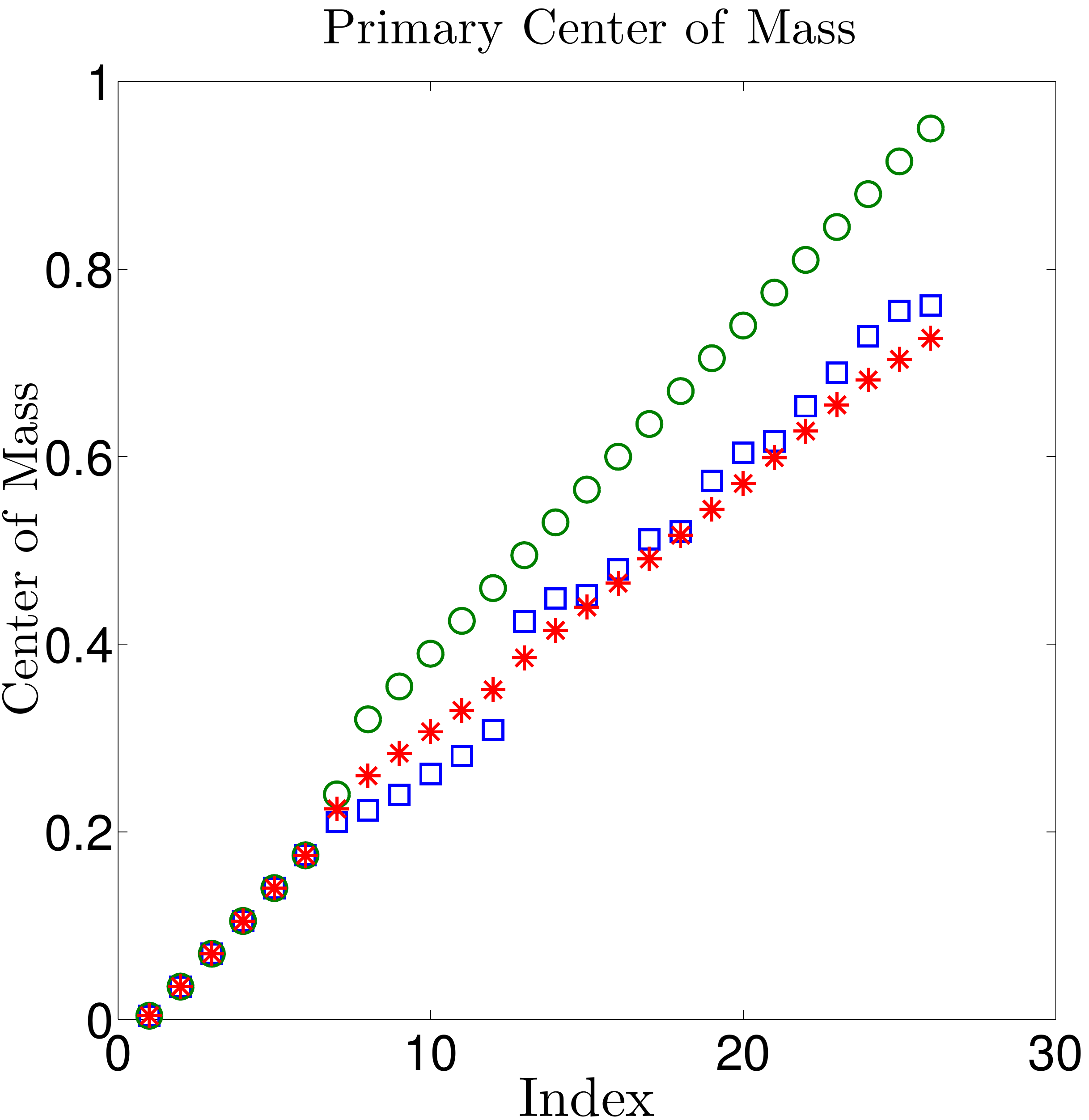} \\
	    (a) & (b) \myspace
	    \includegraphics[width=0.42\textwidth]{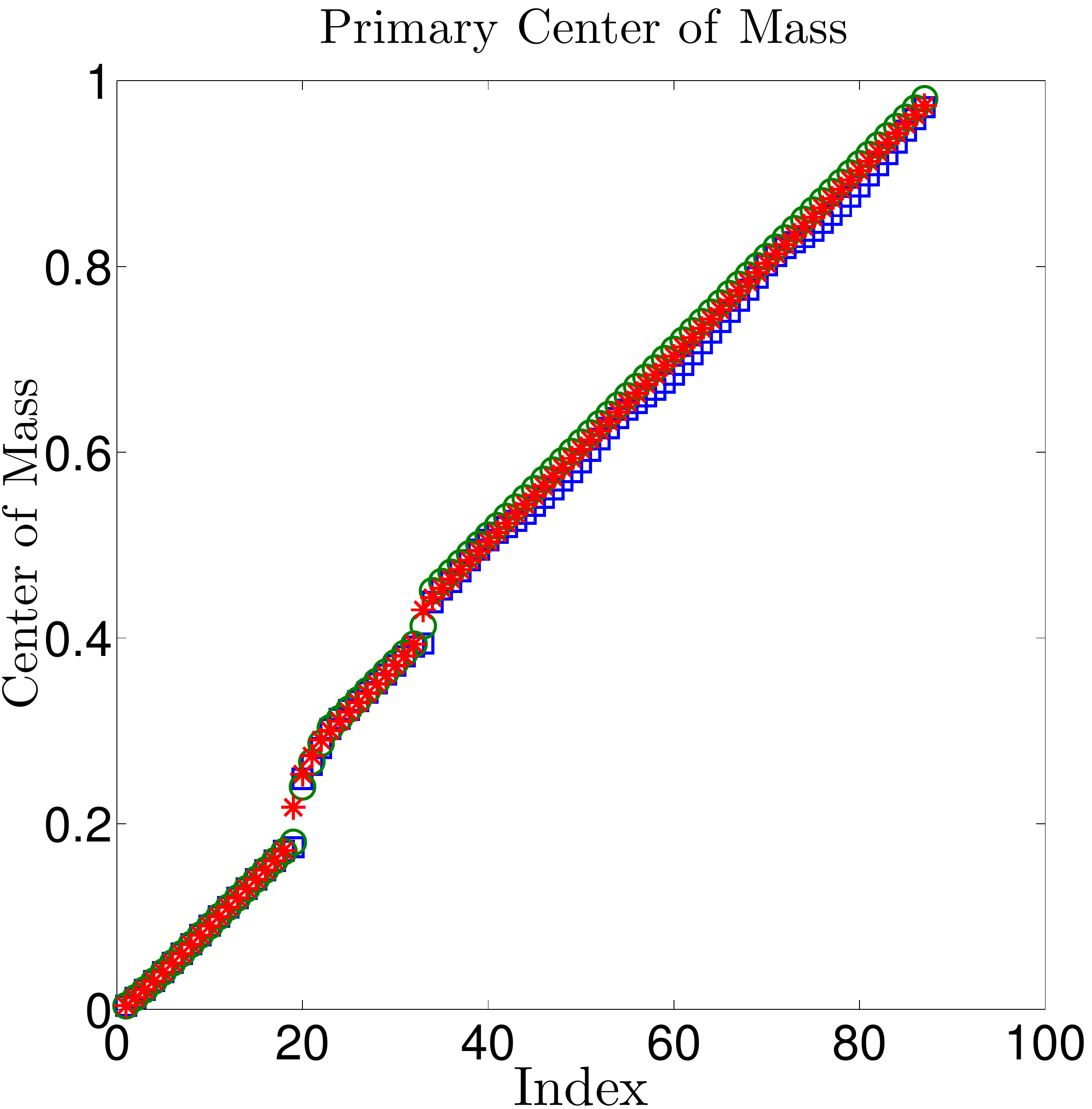} & 
	    \includegraphics[width=0.42\textwidth]{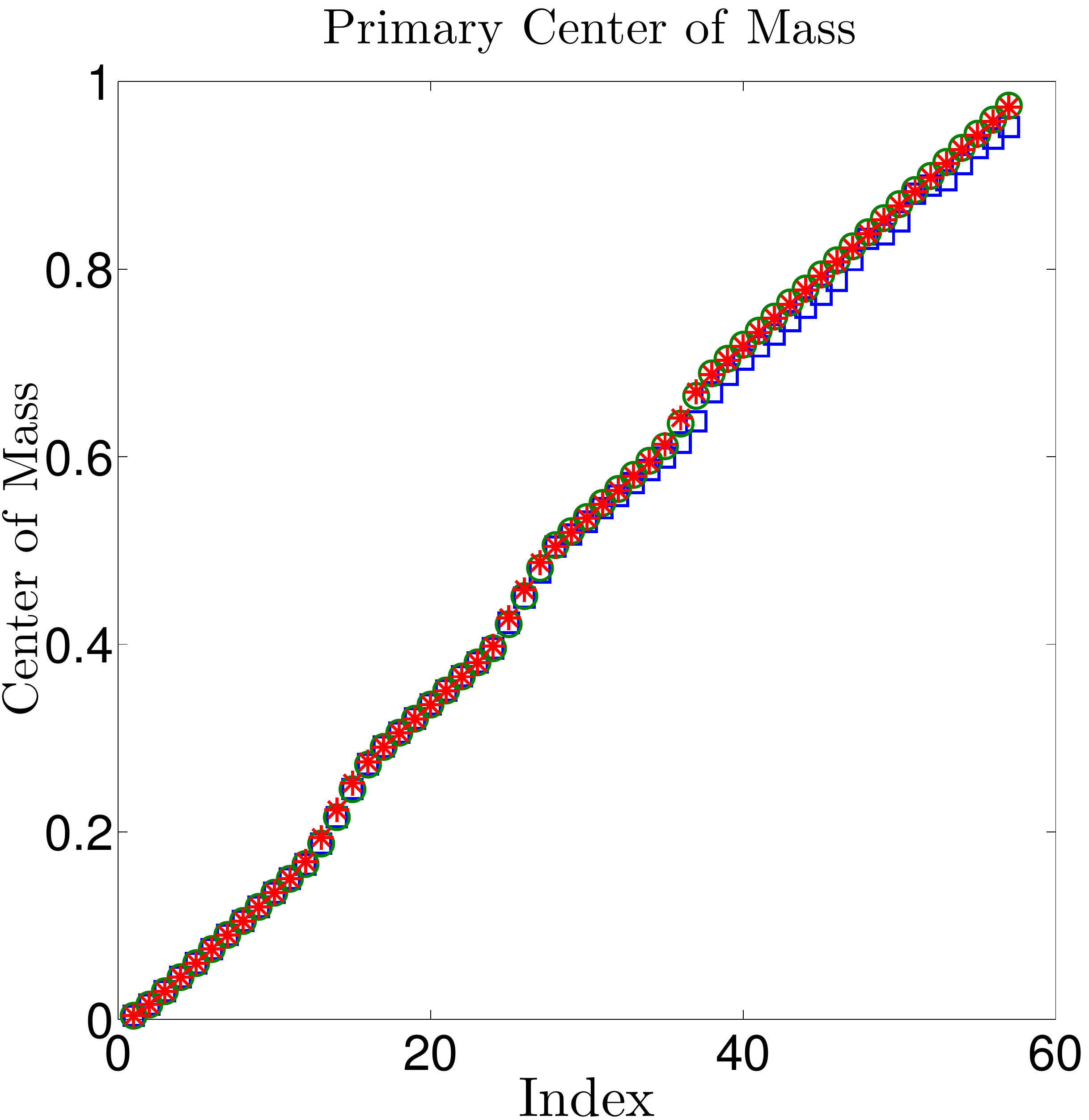} \\
	    (c) & (d)
    \end{tabular}
    \caption{In this figure, we plot the centers of mass and approximate
    centers of mass for (a) the velocity model from
    Figure~\ref{fig:test_velocity}(c); (b) the velocity model from
    Figure~\ref{fig:test_velocity}(d) --- the disagreement between the
    various centers of mass in this figure arises because the
    orthogonalized snapshots are not well localized (because we chose
    $\tau$ to be too large --- see
    Figure~\ref{fig:test_velocity}(d)--(f)); (c) the velocity model from
    Figure~\ref{fig:velocities_1_2}(c); (d) the velocity model from
    Figure~\ref{fig:velocities_1_2}(d).}
    \label{fig:com}
    \end{center}
\end{figure}


\section{Extension to two dimensions}\label{sec:2D}

In this section, we extend our results to two dimensions.  Because the
majority of the results from the one-dimensional case carry over without
significant modifications, we will keep our discussion relatively brief.


\subsection{Multi-input/multi-output formulation}\label{sec:problem_2D}

We begin by defining the region $\Omega \equiv [0,\xmax] \times
[-\ymax,\ymax]$, where we typically take $\ymax = \infty$.  We place $m$
sources at the points $(0,y^i)$ for $i = 1,\ldots,m$, which leads us to
consider the following Cauchy problem on $\Omega \times [0,\infty[$:
\begin{equation}\label{eqn:cauchyb_2D}
    A\ha{u}^i + \ha{u}^i_{tt} = 0, \qquad 
    \ha{u}|_{t=0} = \tq(A)\delta(x+0)\delta(y-y^i), \ \ha{u}^i_t|_{t=0}
    = 0,
\end{equation}
where we take $\tq$ is as in \eqref{eqn:gaussian}, and 
\begin{equation}\label{eqn:A_2D}
    A \equiv -v^2\left(\frac{\partial^2 }{\partial x^2} 
        + \frac{\partial^2 }{\partial y^2}\right)
\end{equation}
together with the boundary conditions
\begin{equation*}
    \ha{u}^i|_{y=\pm\ymax} = 0, \quad 
    \ha{u}^i_x|_{x=0} = 0, \ \ha{u}^i|_{x=\xmax} = 0.
\end{equation*}
We assume $v(0,y) = v(0,0)$ for $y \in [-\ymax,\ymax]$.  

For simplicity, we place the receivers at the same locations as the
sources.  Then, for $k = 0,\ldots,2n-1$, we organize our measurements in
a matrix $\la{F}_k \in \mathbb{R}^{m\times m}$ with $F_k^{ij} \equiv
\ha{u}^i(0,y^j,k\tau)$.  This is the square multi-input/multi-output 
(square MIMO) problem in control theory terminology.

\subsection{MIMO reduced-order model in block form}

For $i = 1,\ldots,m$, let $u^i$ be the solution to the following Cauchy
problem on $\Omega \times [0,\infty[$:
\begin{equation}\label{eqn:cauchys_2D}
    Au^i + u^i_{tt} = 0, \qquad u^i|_{t=0} = b^i, \ u^i_t|_{t=0} = 0,
\end{equation}
where
\begin{equation}\label{eqn:b_2D}
    b^i(x,y) \equiv v(0,0)\tq(A)^{1/2}\delta(x+0)\delta(y-y^i).
\end{equation}
For $k = 0,\ldots,2n-1$, we define the \emph{snapshots}
\begin{equation}\label{eqn:snapshots_def_2D}
    U_k \equiv \left[ u_k^1, \ldots, u_k^m \right] = T_k(P)B,
\end{equation}
where $P = \cos\left(\tau\sqrt{A}\right)$ is the propagation operator,
$B = \left[ b^1, \ldots, b^m \right]$, and 
\[
    u_k^i \equiv u^i(x,y,k\tau) = T_k(P)b^i.
\]
Then the measurement matrix 
\begin{equation}\label{eqn:F_k_2D}
    \la{F}_k = U_0^*U_k = B^*T_k(P)B,
\end{equation}
where $^*$ is defined as before (see \eqref{eqn:WstarU}) with the inner
product
\begin{equation*}
    \left\llangle u, w \right\rrangle
    \equiv \int_{\Omega} u(x,y) w(x,y) \frac{1}{v^2(x,y)} \di{x}\di{y}.
\end{equation*}

The measurement matrix can also be represented by 
\begin{equation}\label{eqn:data_2D}
    \la{F}_k = \int_{-1}^1 T_k(\mu)\la{\eta}_0(\mu) \di{\mu},
\end{equation}
where $\la{\eta}_0$ is an $m\times m$ matrix measure.  In particular,
$\eta^{ij}_0(\mu) = \eta^i(0,y^j,\mu)$ where $\eta^i(x,y,\mu)$ is
defined as in \eqref{eqn:eta} with $\rho$ replaced by 
\begin{equation}\label{eqn:rho_2D}
    \rho^i(x,y,\lambda) \equiv \ds\sum_{l=1}^{\infty}
    \delta(\lambda-\lambda_l)\frac{z_l(0,y^i)}{v(0,0)^2}z_l(x,y);
\end{equation}
$(\lambda_l, z_l)$ are eigenpairs of $A$ with $\left\llangle z_l, z_j
\right\rrangle = \delta_{lj}$.  Next we construct a generalized Gaussian
quadrature such that 
\begin{equation}\label{eqn:chebmatch_2D}
    \int_{-1}^1 T_k(\mu)\la{\eta}_0(\mu)\di{\mu} = \sum_{j=1}^n \la{Y}_j
    T_k(\la{\Theta}_j) \la{Y}_j^T = \la{F}_k \eqfor k = 0,\ldots,2n-1, 
\end{equation}
where $\la{Y}_j = \left[ \la{\Psi}_{1j}, \ldots, \la{\Psi}_{mj}\right]
\in \mathbb{R}^{m \times m}$, i.e., for $l = 1,\ldots,m$,
$\la{\Psi}_{lj} \in \mathbb{R}^m$ is the $l\textsuperscript{th}$ column
of the matrix $\la{Y}_j$, and $\la{\Theta}_j =
\diag(\theta_{1j}, \ldots, \theta_{mj}) \in \mathbb{R}^{m\times
m}$.  

We define the snapshot matrix $U = \left[ U_1, \ldots, U_n\right]$.
Then matrix versions of Lemmas~\ref{lem:data_lemma}--\ref{lem:TplusH}
hold.  In particular, if 
\begin{equation*}
    \la{H} = (U^*U)^{-1}(U^*PU) \in \mathbb{R}^{mn \times mn}
    \eqand{and} 
    \la{E}_1 = 
    \begin{bmatrix}
        \la{I}_{m\times m} \\
	\la{0}_{m\times m} \\
	\vdots \\
	\la{0}_{m\times m}
    \end{bmatrix}
    \in \mathbb{R}^{mn \times m},
\end{equation*}
then 
\begin{equation}\label{eqn:match_2D}
    \la{F}_k = \la{E}_1^T(U^*U)T_k(\la{H})\la{E}_1.
\end{equation}
Additionally, $\la{H}$ has the eigendecomposition 
\begin{equation}\label{eqn:generaleig_2D}
    \la{H} = \la{\Phi}\la{\Theta}\la{\Phi}^TU^*U,
\end{equation}
where $\la{\Theta}$, $\la{\Phi} \in \mathbb{R}^{mn\times mn}$ such that
$\la{\Phi}^TU^*U\la{\Phi} = \la{I}_{mn\times mn}$.  We emphasize 
$\la{H}$ is known by the matrix version of Lemma~\ref{lem:TplusH} (with
$f_k$ replaced by $\la{F}_k$ in the statement and proof of the lemma).
Substituting \eqref{eqn:generaleig_2D} into \eqref{eqn:match_2D} gives 
\begin{equation}\label{eqn:chi_2D}
    \la{F}_k = \la{\chi}^TT_k(\la{\Theta})\la{\chi} \eqfor k =
    0,\ldots,2n-1, \eqand{where} \la{\chi} \equiv
    \la{\Phi}^TU^*U\la{E}_1 \in \mathbb{R}^{mn\times m}. 
\end{equation}
Comparing this with \eqref{eqn:chebmatch_2D} gives
\begin{equation}\label{eqn:tc_2D}
    \diag \ \la{\Theta}_j = \la{\Theta} \eqand{and} 
    \la{\chi} = 
    \begin{bmatrix}
        \la{Y}_1^T \\
	\vdots \\
	\la{Y}_n^T
    \end{bmatrix}.
\end{equation}

We may also construct a symmetric, positive-semidefinite matrix $\la{C}
\in \mathbb{R}^{m \times m}$ and a symmetric, block-tridiagonal matrix 
\begin{equation}\label{eqn:P_n_def_2D}
    \la{P}_n = 
    \begin{bmatrix}
        \la{\alpha}_1 & \la{\beta}_1^T & & \\
	\la{\beta}_1 & \la{\alpha}_2 & \ddots & \\
	& \ddots & \ddots & \la{\beta}_{n-1}^T \\
	& & \la{\beta}_{n-1} & \la{\alpha}_n
    \end{bmatrix}
    = \la{P}_n^T \in \mathbb{R}^{mn\times mn}
\end{equation}
with $\la{\alpha}_j = \la{\alpha}_j^T \in \mathbb{R}^{m \times m}$ and
$\la{\beta}_j \in \mathbb{R}^{m\times m}$ such that
\begin{equation}\label{eqn:match-1_2D}
    \la{F}_k = \ds\sum_{j=1}^n \la{Y}_jT_k(\la{\Theta}_j)\la{Y}_l^T = 
        \la{C}^{1/2}\la{E}_1^T T_k(\la{P}_n)\la{E}_1\la{C}^{1/2} \eqfor 
	k = 0,\ldots,2n-1.
\end{equation}
Taking $k = 0$ in \eqref{eqn:match-1_2D} gives 
\begin{equation}\label{eqn:C_2D}
    \la{C} = \ds\sum_{j=1}^n \la{Y}_j\la{Y}_j^T = \la{F}_0 
    = \int_{-1}^1 \la{\eta}_0(\mu) \di{\mu}.
\end{equation}
From \eqref{eqn:C_2D}, we immediately see that $\la{C}$ is symmetric and
positive-semidefinite --- $\la{C}$ will be positive-definite if and only
if the matrix $\left[ \la{Y}_1, \ldots, \la{Y}_n\right] =
\left[ \la{\Psi}_{11},\ldots,\la{\Psi}_{mn}\right] \in
\mathbb{R}^{m\times mn}$ has rank equal to $m$. 

Analogously to the 1D case, the matrix $\la{P}_n$ has the
eigendecomposition
\begin{equation}\label{eqn:P_n_eig_2D}
    \la{P}_n \la{X} = \la{X}\la{\Theta}, \quad
    \la{X} = \left[\la{X}_1, \ldots, \la{X}_n \right]
    \in\mathbb{R}^{mn\times mn}, 
    \quad \la{X}_l^T\la{X}_j =
    \delta_{lj}\la{I}_{m\times m}
\end{equation}
(the matrices $\la{X}_j \in \mathbb{R}^{m\times m}$ are ``block
eigenvectors'' of $\la{P}_n$ corresponding to the ``block eigenvalues''
$\la{\Theta}_j$).  

We compare \eqref{eqn:match-1_2D} with
\eqref{eqn:chi_2D}--\eqref{eqn:tc_2D} to find $\la{E}_1^T\la{X}_j =
\la{C}^{-1/2}\la{Y}_j$.  Because \eqref{eqn:P_n_eig_2D} is equivalent to
$\la{\Theta}\la{X}^T = \la{X}^T \la{P}_n$, the matrix $\la{P}_n$ may be
constructed using the block-Lanczos algorithm in
Algorithm~\ref{alg:block_Lanczos} (the $mn\times m$ Lanczos ``vectors''
$\la{Q}_j$ in this algorithm are the ``columns'' of the matrix
$\la{X}^T$, i.e., $\la{X}^T = \left[ \la{Q}_1, \ldots,
\la{Q}_n \right]$) --- see, e.g., the book by Parlett
\cite{Parlett:1998:SEP}.
\begin{algorithm}[H]
\caption{Block Lanczos Algorithm for Computing $\la{\alpha}_j$,
    $\la{\beta}_j$.}
    \label{alg:block_Lanczos}
    \begin{algorithmic}
        \Require $\la{C}$, $\la{\Theta}_j$, $\la{Y}_j$ for 
	    $j = 1,\ldots,n$
        \Ensure $\la{\alpha}_j$ ($j = 1,\ldots,n$) and $\la{\beta}_j$ 
	    ($j = 1,\ldots,n-1$), i.e., the nonzero elements of 
	    $\la{P}_n$
	\State Set $\la{Q}_0 = \la{0}_{mn\times m}$ and 
	    \[
                \la{Q}_1 = 
		\begin{bmatrix}
                    \la{Y}_1^T\la{C}^{-1/2} \\
		    \vdots \\
		    \la{Y}_n^T\la{C}^{-1/2}
		\end{bmatrix}.
	    \]
        \For{$j = 1,\ldots,n$}
            \begin{enumerate}
                \item $\la{\alpha}_j = \la{Q}_j^T\la{\Theta}\la{Q}_j$;
		\item $\la{R}_j = \la{\Theta}\la{Q}_j -
			\la{Q}_{j-1}\la{\beta}_{j-1}^T -
			\la{Q}_j\la{\alpha}_j$;
        \item $\la{\beta}_j = \left(\la{R}_j^T\la{R}_j\right)^{1/2}$;
		\item $\la{Q}_{j+1} = \la{R}_j\la{\beta}_j^{-1}$.
	    \end{enumerate}    
        \EndFor
    \end{algorithmic}
\end{algorithm}


\subsection{Continuum interpretation in two dimensions}
\label{subsec:continuum_2D}

We now derive an inversion algorithm analogous to the algorithm we
constructed in \S~\ref{sec:inversion}.  The key ingredients are the
matrix extensions of $\ghat_j$ and $\g_j$; in particular, we now
consider $m \times m$ symmetric, positive-definite matrices $\Ghat_j$,
$\G_j$ for $j = 1,\ldots,n$.  These matrices may be computed via
Algorithm~\ref{alg:algorithm_1_2D} \cite{Druskin:2010:OSF,
Druskin:2010:OSF2}, which is a matrix version of
Algorithm~\ref{alg:algorithm_1}.  We conjecture that full rank of the
Gram matrix $U^*U$ is a sufficient condition for
Algorithm~\ref{alg:algorithm_1_2D} to succeed.
\begin{algorithm}[H]
    \caption{Computation of $\Ghat_j$, $\G_j$}
	\label{alg:algorithm_1_2D}
	\begin{algorithmic}
	    \Require $\la{C}$, $\la{Y}_l$, $\la{\Lambda}_l = 
	    -\xi(\la{\Theta}_l) \in\mathbb{R}^{m\times m}$ for 
	        $l = 1,\ldots,n$
	    \Ensure $\Ghat_j$, $\G_j$, $j = 1,\ldots,n$
	    \State Set $\obomega_0 = \la{0}_{2mn\times m}$, 
	        $\obmu_1 = \sqrt{0.5}\cdot
		    \left[ 
		        \la{Y}_1, \la{Y}_1, \la{Y}_2, \la{Y}_2, \ldots, 
			\la{Y}_n, \la{Y}_n \right]^T 
			\in\mathbb{R}^{2mn\times m}$, and \newline
            $\la{L} = \diag \ \left(\la{\Lambda}_1^{1/2},
                -\la{\Lambda}_1^{1/2}, \la{\Lambda}_2^{1/2},
                -\la{\Lambda}_2^{1/2}, \ldots, \la{\Lambda}_n^{1/2},
            -\la{\Lambda}_n^{1/2}\right) \in \mathbb{R}^{2mn\times
		2mn}$.
	    \For{$j = 1,\ldots,n$}
            \begin{enumerate}
                \item $\Ghat_j = \left(\obmu_j^T\obmu_j\right)^{-1}$;
                \item $\obomega_j = \obomega_{j-1} + 
		    \la{L}\obmu_j\Ghat_j$;
	        \item $\G_j = \left(\obomega_j^T\obomega_j\right)^{-1}$;
                \item $\obmu_{j+1} = \obmu_{j} - \la{L}\obomega_j\G_j$.
	    \end{enumerate}    
	    \EndFor
	\end{algorithmic}
\end{algorithm}

\begin{remark}\label{rem:heuristics}
    In what follows, we illustrate one way in which the inversion
    algorithm from \S~\ref{sec:inversion} may be extended to 2D.
    In particular, we avoid technical details and focus on providing an
    heuristic justification of our algorithm.  We have recently
    developed a more rigorous 2D inversion algorithm
    \cite{Mamonov:2015:NSI} that relies on many of the ideas discussed
    in the present paper.  

\end{remark}

In 2D, an invertible coordinate transformation to slowness (also known
as ray or traveltime) coordinates analogous to \eqref{eqn:slowness} may
not exist for most relevant cases due to the formation of caustics
\cite{Sava:2005:RWE}.  If the medium under consideration is
``approximately layered'' in the vertical direction, however, it is
plausible that an invertible transformation to ray coordinates exists.  

Henceforth we assume that rays perpendicular to the line $x = 0$ do not
intersect.  This ensures that the ray coordinate transformation $(x,y)
\mapsto (\zeta,\nu)$ exists and is invertible.  Here $\zeta$ represents
the traveltime along a ray and $\nu$ is orthogonal to $\zeta$; we also
assume the line $x = 0$ is mapped to the line $\zeta = 0$.  Thus the
curves $\nu = \const$ represent the rays orthogonal to the
line $x = 0$.  We define $\tu^i(\zeta,\nu) \equiv
u^i(x(\zeta,\nu),y(\zeta,\nu))$ (and similarly for other functions of
$x$ and $y$).  Then \eqref{eqn:cauchys_2D} transforms to
\begin{equation}\label{eqn:cauchyslow_2D}
    \tA\tu^i + \tu^i_{tt} = 0, \qquad 
        \tu|_{t=0} = \tb, \ \tu_t|_{t=0} = 0,
\end{equation}
where $\tA$ is the operator $A$ represented in ray coordinates
\cite{Sava:2005:RWE}.  In particular, in an ``approximately layered''
medium, we approximate $\tA$ along rays by
\begin{equation}\label{eqn:tA_2D}
    \tA \approx
    -\tv \frac{\partial }{\partial \zeta} \left(\frac{1}{\tv}
    \frac{\partial\tu}{\partial \zeta}\right);
\end{equation}
thus our problem essentially reduces to a 1D problem (in a layered
medium, we have $\nu = y$ and $\zeta = \int_0^x 1/v(x') \di{x'}$, as in
1D).  

As in the 1D setting, we consider the first $n$ primary and dual
snapshots, namely $\ti{U}_k$ and $\ti{W}_k$, respectively ($k =
0,\ldots,n-1$).  The orthogonalized snapshots $\oU_j$ and $\oW_j$
(computed via an algorithm analogous to Algorithm~\ref{alg:algorithm_2})
will again be localized in some sense.  Moreover, we have $\Ghat_j =
(\oU_j^*\oU_j)^{-1}$ and $\G_j = (\oW_j^*\oW_j)^{-1}$ (where $^*$ is
defined as in \S~\ref{subsec:Galerkin_Ritz} with respect to an
appropriate inner product), so $\Ghat_j$ and $\G_j$ are symmetric,
positive-semidefinite matrices.  The matrices $\Ghat_j$, $\G_j$ may be
loosely interpreted to contain information about the local wave speed as
follows.  

As in \cite{Druskin:2010:OSF, Druskin:2010:OSF2}, we use the diagonals
of $\Ghat_j^{-1}$ and $\G_j$, denoted by $\ha{\la{\g}}_j \in
\mathbb{R}^{m}$ and $\la{\g}_j \in \mathbb{R}^{m}$, respectively, as the
analogues of $\ghat_j$ and $\g_j$ from \eqref{eqn:vel_inv}.  The
reasoning behind our use of the diagonals is twofold.  First, the set of
data matrices $\la{F}_k$ ($k = 0,\ldots,2n-1$) is effectively
three-dimensional, as is the set of $\Ghat_j$, $\G_j$.  Our problem is
overdetermined because we are trying to recover an approximation of $v$
on a two-dimensional grid.  Although we use the full data to compute
$\Ghat_j$ and $\Ghat$, we reduce the dimensionality by using only the
diagonals $\ha{\la{\g}}_j$ and $\la{\g}_j$ \cite{Druskin:2010:OSF2}.
Second, recall $\Ghat_j^{-1} = (\oU_j^*\oU_j)$.  Then
$\la{e}_i^T\la{\ghat}_j = \widehat{\Gamma}^{-1}_{j,ii} =
\left(\oU_j^i\right)^*\oU_j^i$.  Since the approximate operator $\tA$ in
\eqref{eqn:tA_2D} is of the same form as in the 1D case (see
\eqref{eqn:cauchyslow}), it seems reasonable to assume the quantities
$\la{e}_i^T\ha{\la{\g}}_j$ are related to localized averages of
$\frac{1}{\tv}$.  

Our algorithm thus proceeds as follows.  We consider a background
velocity $v^0$ with $v^0(0,y) = v(0,0)$ for $y \in [-\ymax,\ymax]$.  We
choose the velocity to be simple enough so the ray coordinate
transformation is well-posed; for example, we took $v^0 = v(0,0)$ in our
numerical experiments.  For a constant background velocity, ray
coordinates are particularly simple --- in fact, $\zeta =
\frac{1}{v_0}x$ and $\nu = y$.  

The grid points at which we approximate the true velocity are
$(\zeta^0_j, \nu^0_i)$ and $(\ha{\zeta}^0_j, \ha{\nu}^0_i)$, where, for
a constant background velocity $v^0$, $\zeta^0_j$ is computed as in the
1D case and $\nu^0_i = y^i$.  The dual grid points $(\ha{\zeta}^0_j,
\ha{\nu}^0_i)$ are defined in a similar manner.  The velocity is
approximated at the grid points (in ray coordinates) by 
\begin{equation}\label{eqn:vel_inv_2D}
    \tv(\zeta^0_j, \nu^0_i) \approx \tv^0(\zeta^0_j, \nu^0_i) 
        \frac{\la{e}_i^T\ha{\la{\g}}^0_j}{\la{e}_i^T\ha{\la{\g}}_j}.
    \eqand{and}
    \tv(\ha{\zeta}^0_j, \ha{\nu}^0_i) \approx 
        \tv^0(\ha{\zeta}^0_j, \ha{\nu}^0_i) 
        \frac{\la{e}_i^T\la{\g}_j}{\la{e}_i^T\la{\g}^0_j}.
\end{equation}
We may approximately convert the grid points $(\zeta^0_j, \nu^0_i)$ and
$(\ha{\zeta}^0_j, \ha{\nu}^0_i)$ to spatial coordinates $(x^0_j, y^0_i)$
and $(\ha{x}^0_j, \ha{y}^0_i)$, respectively, by inverting the
coordinate transform using our imaged velocity (much like in the 1D
case).

In Figure~\ref{fig:2D_inversion}, we plot the results of two numerical
experiments using our 2D inversion method.  In both cases, we used a
constant background velocity.  Figure~\ref{fig:2D_inversion}(a) is the
image of a block --- the true velocity corresponding to
Figure~\ref{fig:2D_inversion}(a) is plotted in
Figure~\ref{fig:2D_inversion}(b); Figure~\ref{fig:2D_inversion}(c) is
the image of a dipping interface --- the true velocity corresponding to
Figure~\ref{fig:2D_inversion}(c) is plotted in
Figure~\ref{fig:2D_inversion}(d).  In Figures~\ref{fig:2D_inversion}(a)
and (c), the horizontal axis is in slowness coordinates, while in
Figures~\ref{fig:2D_inversion}(b) and (d) the horizontal axis is in
physical coordinates.  They show qualitatively correct inversion results, even though  our
assumption on nonintersecting rays fails for the block  model. 

The above imaging procedure was further improved  of  by some of the
authors with preliminary results (including imaging of a 2D Marmousi
cross-section) reported in \cite{Mamonov:2015:NSI}.  
\begin{figure}[ht]
    \begin{center}
    \begin{tabular}{c c}
        \includegraphics[width=0.45\textwidth]{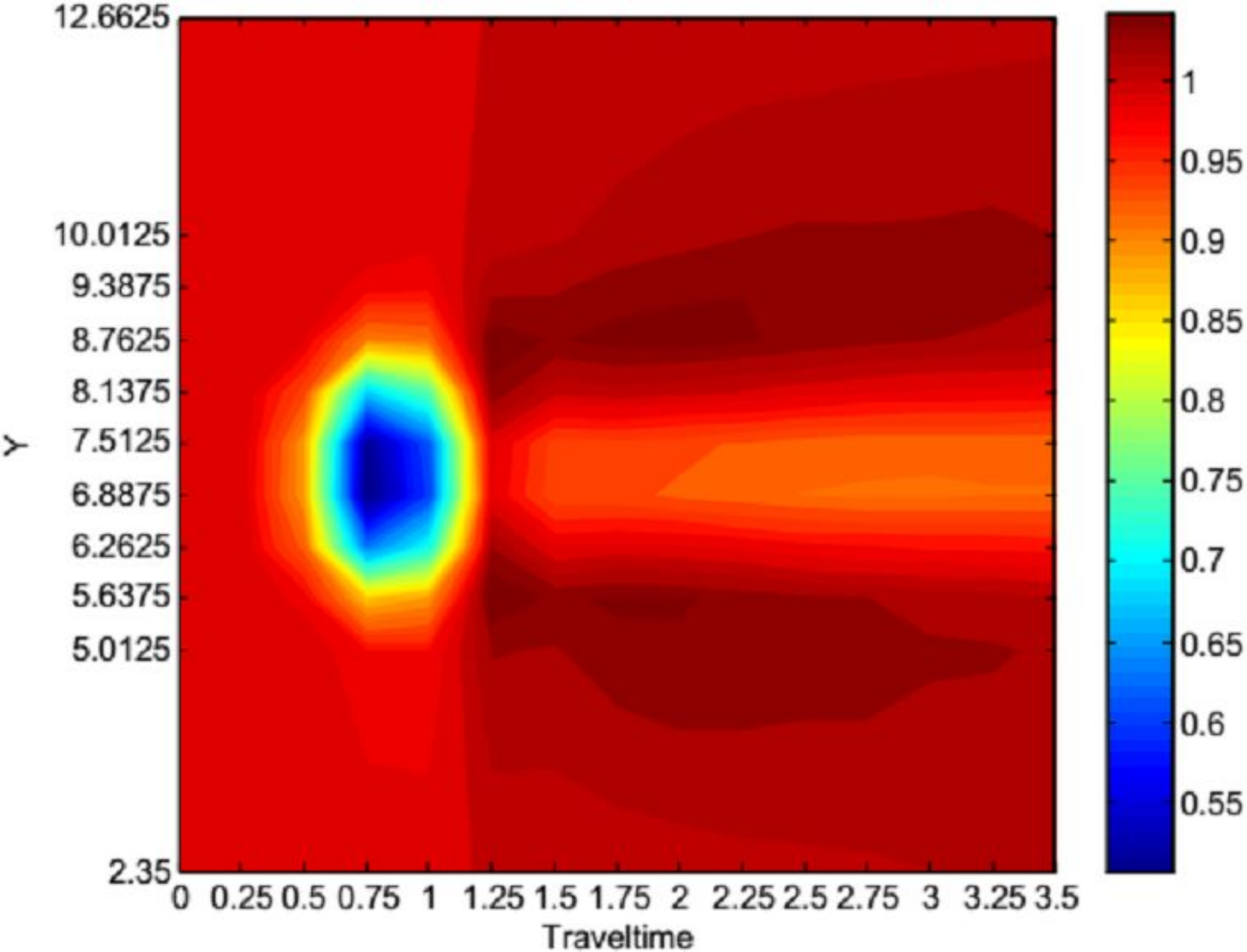} & 
	    \includegraphics[width=0.45\textwidth]{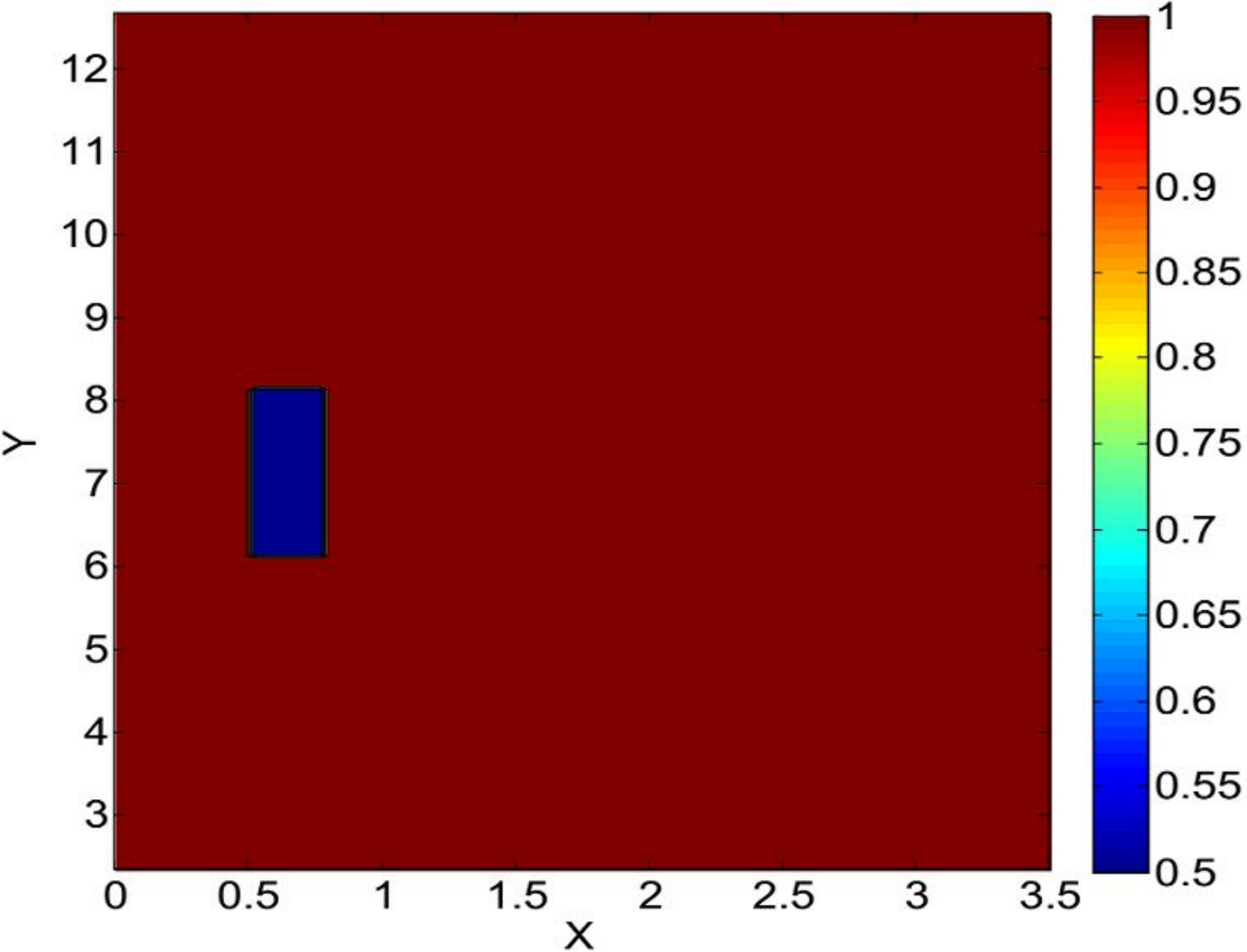}
        \\
	    (a) & (b) \\
        \includegraphics[width=0.45\textwidth]{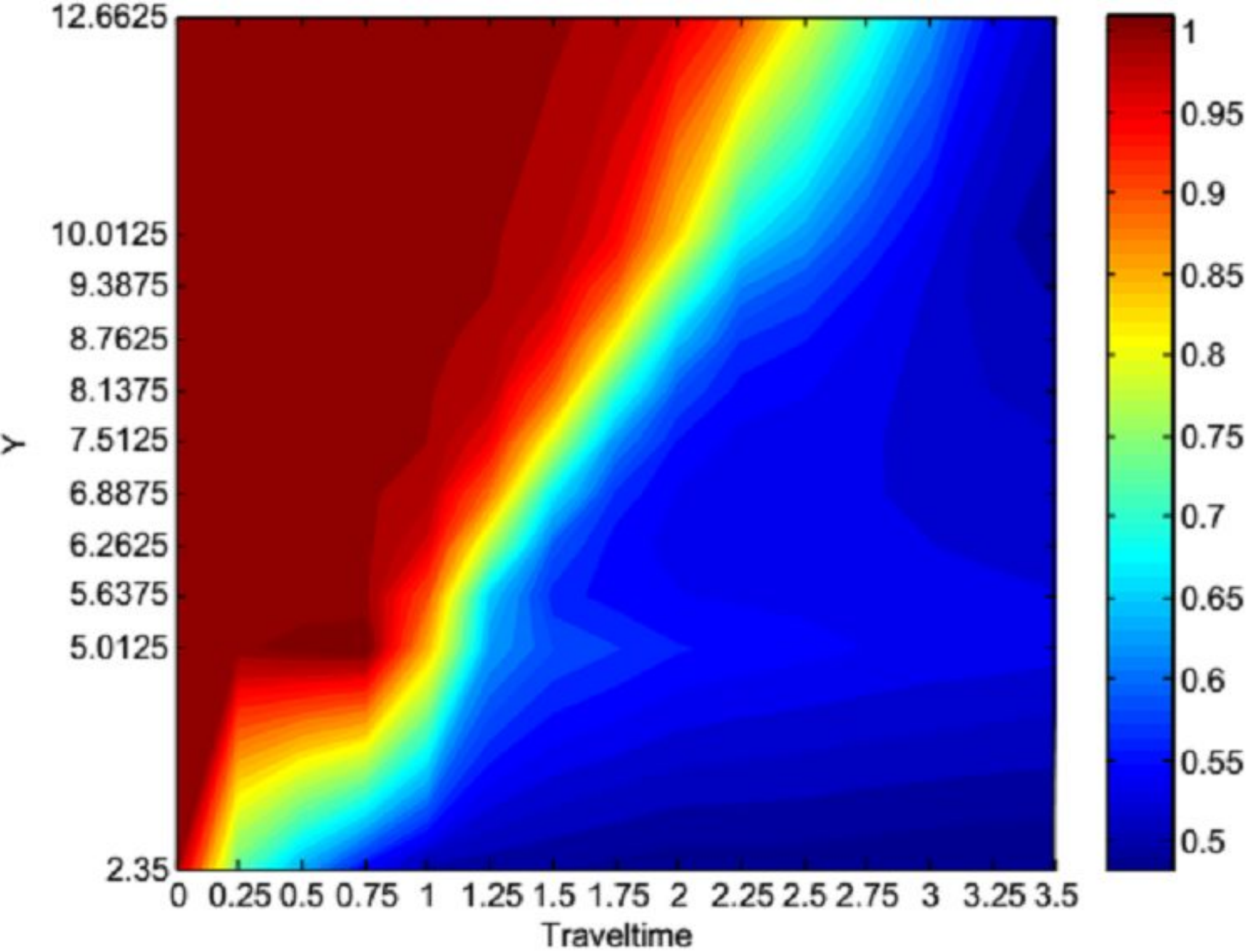}
        &
        \includegraphics[width=0.45\textwidth]{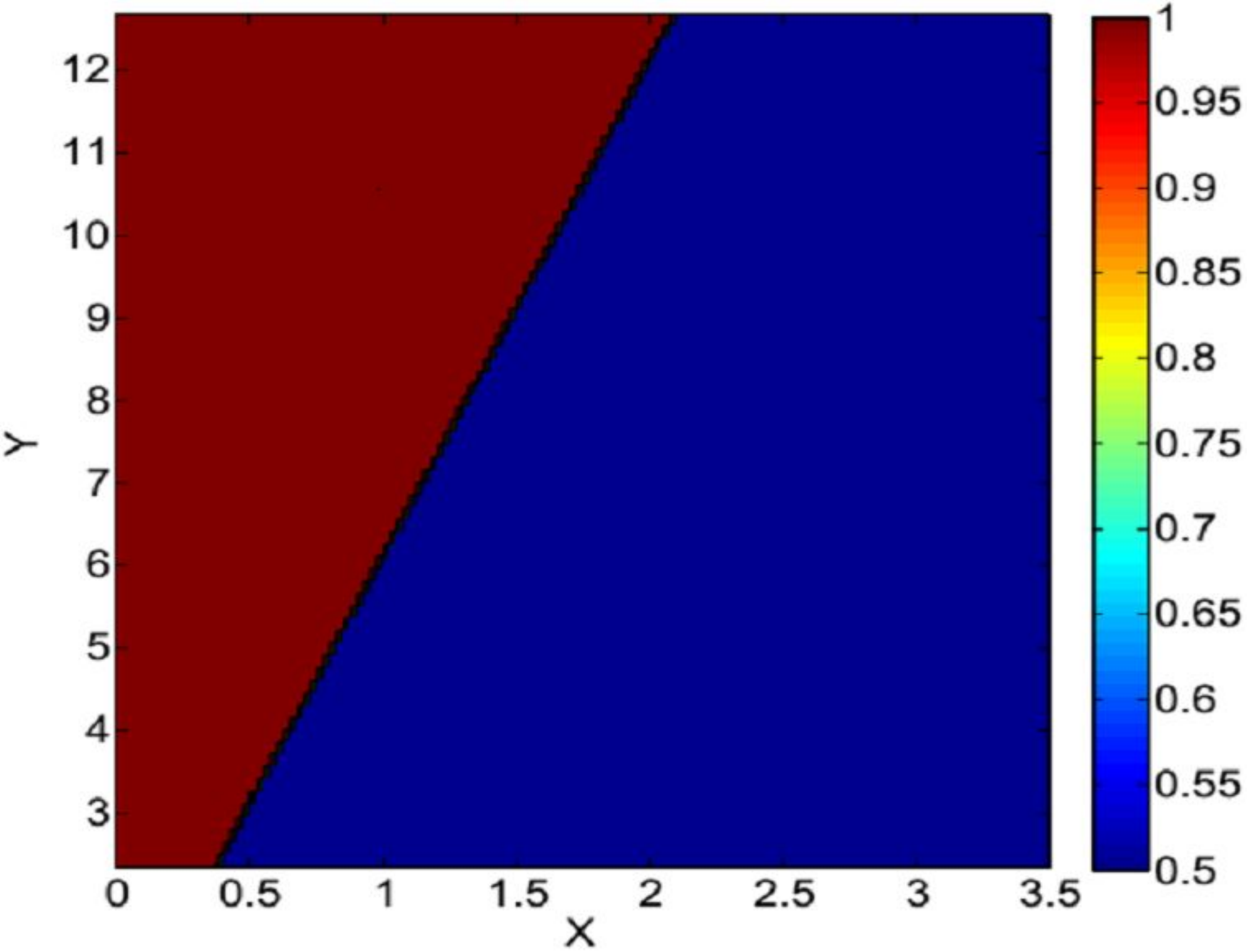}
    \end{tabular}
    \caption{In this figure, we plot the results of two numerical
    experiments using our 2D inversion method sketched in this section.
    The horizontal axis is in slowness coordinates for figures (a) and
    (c) while it is in true (physical) coordinates for figures (b) and
    (d).  We used a constant background velocity in both experiments:
    (a) image of a block inclusion; (b) true velocity corresponding to
    (a); (c) image of a dipping interface; (d) true velocity
    corresponding to (d).} \label{fig:2D_inversion}
    \end{center}
\end{figure}

\section{Conclusion}
We developed a model reduction framework for the solution of inverse
hyperbolic problems. This is a brief summary of our approach.
\begin{itemize}
    \item We start with a one-dimensional problem and
        single-input/single output (SISO) time-domain boundary
        measurements.
    \begin{itemize}
        \item We sample the data on a given temporal interval consistent
            with the Nyquist-Shannon theorem and construct the ROM
            interpolating the data at the sampling points. The ROM is
            obtained via the Chebyshev moment problem, which can be
            equivalently represented via Galerkin projection on the
            subspace of the wavefield snapshots, i.e., a Krylov subspace
            of the propagation operator.
        \item Using the Lanczos algorithm, we transform the projected
            system to a sparse form that mimics a finite-difference
            discretization of the underlying wave problem. This
            transformation is equivalent to Gram--Schmidt
            orthogonalization, and yields a localized orthogonal
            basis on the snapshot subspace.
        \item We estimate the unknown PDE coefficient via coefficients
            of the sparse system.  The coefficients of this sparse
            system are weighted averages of the true, unknown velocity,
            where the weight functions are localized (in particular,
            they are the squared orthogonalized snapshots).
        \item Numerical experiments show quantitatively good images of
            layered media, though the image quality depends on the
            consistency between the time-sampling and the pulse spectral
            content.
    \end{itemize} 
    \item We outline a generalization to the multidimensional setting
        (on a 2D example) with square multi-input/multi-output
        (MIMO) boundary data.
    \begin{itemize} 
        \item We construct the MIMO ROM data via the block-counterpart
            of the SISO algorithm. 
        \item The continuum interpretation of the MIMO ROM is done via
            geometrical optics. 
        \item Two-dimensional numerical experiments show that the
            imaging algorithm gives qualitatively correct results even
            when the geometric optics assumption does not hold.
    \end{itemize}
\end{itemize}

The key of the efficiency of the proposed approach is the weak
dependence of the orthogonalized snapshots on the media, which allows us
to use a single background Krylov basis for accurate Galerkin
projection. At the moment we only have experimental verification of that
phenomenon, and can conjecture a result similar to the asymptotic
independence of the optimal grids on variable coefficients
\cite{Borcea:2002:OFD}.  We believe that such a basis can also be found
for interpolatory model reduction in the frequency domain (via a
rational Krylov subspace), and investigation in this direction is under
way.

Another advantage of our proposed algorithm over traditional FWI is that
modeling errors are not an issue; because we use a homogeneous
background wavespeed, the background solution can be obtained
analytically.  We should mention, however, that our algorithm does not
perform well if a nonconstant background wavespeed that is not very
close to the true wavespeed is used.  Although this limits the ability
of our algorithm to be used iteratively, we believe our algorithm
performs very well either as a direct algorithm or as a nonlinear
preconditioner for generating an initial model for FWI
\cite{Borcea:2014:MRA}.  Additionally, in \S~\ref{sec:numerics}, we
discussed the stability of the 1D algorithm in the context of the choice
of the sampling stepsize $\tau$, which essentially plays the role of a
regularization parameter.

We must admit that the generalization to multidimensional problems is
still in its initial stage. The square MIMO formulation is
overdetermined; this gives rise to a multitude of different imaging
formulas, even though the equivalent state-variable ROM representation
is unique up to a change of basis. One such formula, outlined in
\cite{Mamonov:2015:NSI} (still based on the MIMO ROM construction
presented in this paper), apparently has sharper resolution than the
algorithm of \S~\ref{subsec:continuum_2D}.  We also discovered
some stability issues in the 2D case --- these will be addressed in a
forthcoming work.

Moreover, the collocated square MIMO formulation considered in this work
may not be suitable for some practically important measurement systems
in seismic exploration and other remote sensing applications. To
circumvent this deficiency, we are looking at the extension  of our
approach to non-collocated source-receiver arrays with a different
number of sources and receivers, which leads to rectangular MIMO
formulations within the Galerkin--Petrov projection framework. Another
possible extension is a back-scattering formulation used for radar
imaging, corresponding to one or a few diagonals of the square MIMO
matrix data set.


\section*{Acknowledgments}

The authors wish to thank Olga Podgornova and Fadil Santosa for helpful
discussions related to the topics presented in this paper. 


\appendix

\section{Proofs}\label{app:proofs}

In this appendix, we present some calculations and proofs we omitted in
the body of the paper.


\subsection{Derivation of \eqref{eqn:u_eta}} 
\label{subsec:derivation_of_eta}

We begin by recalling \eqref{eqn:ufun}:
\begin{equation}\label{eq:u_int_again}
    \ha{u}(x,k\tau) = 2\int_0^{\infty}\cos(k\tau s)\rho(x,s^2)s
        \ti{q}(s^2)\di{s}.
\end{equation}
We make the change of variables $y = \tau s$ in \eqref{eq:u_int_again}
to obtain
\begin{equation}\label{eq:u_int_y}
    \ha{u}(x,k\tau) = \frac{2}{\tau^2}\int_0^{\infty} \cos(k y)
    \rho(x,(y/\tau)^2) y 
    \ti{q}\left((y/\tau)^2\right)\di{y}.
\end{equation}
Henceforth we will take the principal branch of $\arccos$, namely
$\arccos : [-1,1] \mapsto [0,\pi]$. 

Next, we break the integral in \eqref{eq:u_int_y} into infinitely many
segments so we can apply an invertible change of coordinates of the form
$\mu = \cos y$ to each segment; in particular, we have
\begin{multline}\label{eqn:u_int_segments}
    \ha{u}(x,k\tau) 
    = \frac{2}{\tau^2}\ds\sum_{j=0}^{\infty}
        \int_{2j\pi}^{(2j+1)\pi} \cos(ky) 
        \rho\left(x,(y/\tau)^2\right) y 
        \ti{q}\left((y/\tau)^2\right)\di{y}
    \\
    + \frac{2}{\tau^2}\ds\sum_{j=1}^{\infty} 
        \int_{(2j-1)\pi}^{2j\pi} \cos(ky)
        \rho\left(x,(y/\tau)^2\right) y 
        \ti{q}\left((y/\tau)^2\right)\di{y}.
\end{multline}

We now make the following changes of variables in the first and
second integrals in \eqref{eqn:u_int_segments}, respectively:
\begin{subequations}\label{eqn:cov}
    \begin{equation}\label{eqn:cov_mu}
        \mu = \cos(y), \quad y = \arccos(\mu) + 2j\pi, \quad 
        dy = -\frac{1}{\sqrt{1-\mu^2}} d\mu;
    \end{equation}
    \begin{equation}\label{eqn:cov_nu}
        \nu = \cos(y), \quad y = -\arccos(\nu) + 2j\pi, \quad 
        dy = \frac{1}{\sqrt{1-\nu^2}} d\nu.
    \end{equation}
\end{subequations}
Using \eqref{eqn:cov_mu} and \eqref{eqn:cov_nu} in the first and second
integrals in \eqref{eqn:u_int_segments}, respectively, we obtain
\begin{equation*}
    \begin{aligned}
        \ha{u}(x,k\tau) 
        &= \frac{2}{\tau^2}\ds\sum_{j=0}^{\infty}
            \int_{1}^{-1} \cos(k(\arccos(\mu)+2j\pi)) 
            \rho\left(x,\left(\frac{\arccos(\mu)+2j\pi}{\tau}
    	\right)^2\right) \\
        &\qquad \cdot(\arccos(\mu)+2j\pi)
            \ti{q}\left(\left(\frac{\arccos(\mu)+2j\pi}{\tau}
    	\right)^2\right) \left(-\frac{1}{\sqrt{1-\mu^2}}\right) 
    	\di{\mu} \\
        &+  \frac{2}{\tau^2}\ds\sum_{j=1}^{\infty} 
        \int_{-1}^{1} \cos(k(-\arccos(\nu)+2j\pi))
            \rho\left(x,\left(\frac{-\arccos(\nu)+2j\pi}{\tau}
    	\right)^2\right) \\
        &\qquad \cdot(-\arccos(\nu)+2j\pi) 
            \ti{q}\left(\left(\frac{-\arccos(\nu)+2j\pi}{\tau}
    	\right)^2\right)\frac{1}{\sqrt{1-\nu^2}} \di{\nu}.
    \end{aligned}
\end{equation*}
We then use the $2\pi$-periodicity of cosine, transform $j\rightarrow
-j$ in the second sum, and use the definition of the Chebyshev
polynomials of the first kind to find 
\begin{equation*}
    \begin{aligned}
        \ha{u}(x,k\tau) 
        &= \frac{2}{\tau^2}\ds\sum_{j=0}^{\infty}
            \int_{-1}^{1} T_k(\mu) 
            \rho\left(x,\left(\frac{\arccos(\mu)+2j\pi}{\tau}
    	\right)^2\right) \\
        &\qquad \cdot(\arccos(\mu)+2j\pi)
            \ti{q}\left(\left(\frac{\arccos(\mu)+2j\pi}{\tau}
    	\right)^2\right) \frac{1}{\sqrt{1-\mu^2}} \di{\mu} \\
        &-  \frac{2}{\tau^2}\ds\sum_{j=-\infty}^{-1} 
        \int_{-1}^{1} T_k(\mu) 
            \rho\left(x,\left(\frac{\arccos(\mu)+2j\pi}{\tau}
    	\right)^2\right) \\
        &\qquad \cdot(\arccos(\mu)+2j\pi) 
            \ti{q}\left(\left(\frac{\arccos(\mu)+2j\pi}{\tau}
    	\right)^2\right)\frac{1}{\sqrt{1-\mu^2}} \di{\mu} \\
	&= \int_{-1}^1 T_k(\mu)\eta(x,\mu)\di{\mu}.
    \end{aligned}
\end{equation*}


\subsection{Support of $c^{-1}\eta_0(\mu)\di{\mu}$ for constant velocity}
\label{subsec:eta0_increase}

For $s \in [-1,1]$, let us define $N(s) \equiv \int_{-1}^s c^{-1}
\eta_0(\mu)\di{\mu}$.  Then the hypothesis of
Lemma~\ref{lem:points_of_increase} holds if $N$ has at least $n$
points of increase for $s \in [-1,1]$ (the set of all points of increase
for $N$ is also known as the support or spectrum of the measure
$c^{-1}\eta_0(\mu)\di{\mu}$ --- see, e.g., Chapter 1 of
\cite{Gautschi:2004:OPC}).  Here, we show that if the
wavespeed is constant, then $N$ has exactly
$n$ points of increase in $[-1,1]$; we provide a qualitative explanation
for the nonconstant wavespeed case at the end of the section.

Recall from \eqref{eqn:eta} and \eqref{eqn:rho} that 
\[
    \eta_0(\mu) = \eta(0,\mu) = \sum_{j=-\infty}^{\infty} r_j(\mu)
    \sum_{l=1}^{\infty} \delta\left(\frac{(\arccos(\mu) + 2\pi
    j)^2}{\tau^2}-\lambda_l\right)\frac{z_l(0)}{v(0)^2}z_l(x),
\]
where $(\lambda_l,z_l(x))$ is the $l^\textsuperscript{th}$ eigenpair of
$A$ and 
\[
    r_j(\mu) \equiv 2\mathrm{sgn}(j)
        \tq \left(\frac{(\arccos(\mu) +2j\pi)^2}{\tau^2} \right) 
        \frac{\arccos(\mu)+2j\pi}{\tau^2}  \frac{1}{\sqrt{1-\mu^2}}.
\]
Note that if $\tq$ is given by \eqref{eqn:gaussian_FT}, then $r_j > 0 $
for $\mu \in ]-1,1[$.

For simplicity, we take the wavespeed $v \equiv 1$ and $\xmax = 1$.
Recall that we take $2n$ measurements on the time interval $[0,T]$ at
the discrete times $k\tau$ for $k = 0,\ldots,2n-1$, where $(2n-1)\tau =
T$.  For the sake of illustration, let us take $\tau = 1/n$ for some
$n$; then $T = 2-\frac{1}{n}$ and the timestep $\tau$ is determined by
the number of snapshots we wish to take.

When $v \equiv 1$, the eigenfunctions of $A$ satisfy
\[
    -z_{l,xx} = \lambda_l z_l, \qquad
    z_{l,x}|_{x=0} = 0, \ z_l|_{x=1} = 0.
\]
Then the eigenvalues and (orthonormal) eigenfunctions are 
\begin{equation}\label{eqn:A_eigenfunctions}
    \lambda_l = \left[\frac{(2l-1)\pi}{2}\right]^2 \eqand{and}
    z_l(x) = \sqrt{2}\cos\left(\sqrt{\lambda_l}x\right), 
    \quad l = 1,2,\ldots.
\end{equation}
Reversing the arguments from \S~\ref{subsec:derivation_of_eta}, using
the fact that $\tau = 1/n$, and using the expressions in
\eqref{eqn:A_eigenfunctions} for the eigenfunctions of $A$ we find
\begin{equation}\label{eqn:Ns}
    \begin{aligned}
    N(s) &= 4c^{-1}\left[\sum_{j=0}^{\infty} 
            \int_{(\arccos(s) + 2\pi j)n}^{(2j+1)\pi n} 
            \tq(r^2) r \sum_{l=1}^{\infty} \delta(r^2-\lambda_l) \di{r}
            \right. \\
            &\left.\qquad +
            \sum_{j=1}^{\infty} \int_{(2j-1)\pi n}^{(2\pi j - \arccos
            (s))n}
            \tq(r^2) r \sum_{l=1}^{\infty} \delta(r^2-\lambda_l)
            \di{r}\right]. 
    \end{aligned}
\end{equation}
For a fixed value of $s \in [-1,1]$, only certain values of $\lambda_l$
will contribute to the above integrals.  In Figure~\ref{fig:eta_0_int},
we illustrate the first couple of integration intervals for the first
(second) integral in \eqref{eqn:Ns} in red (blue) for a given value of
$s$.  The square roots of the eigenvalues from
\eqref{eqn:A_eigenfunctions} are marked with crosses --- here $n = 6$.
As $s \rightarrow 1^-$, the intervals will increase in width until the
positive half of the real line is covered.
\begin{figure}[!hb]
    \centering
    \includegraphics[width=\textwidth]{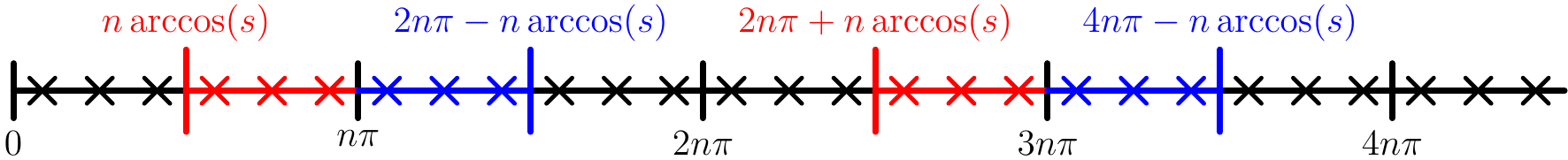}
    \caption{In this figure, we sketch the integration intervals for the
        integrals in \eqref{eqn:Ns} for a fixed value of $s \in [-1,1]$.
        The square roots of the eigenvalues of $A$, namely
        $\sqrt{\lambda_l}$, are marked with crosses.
        The red intervals in the figure correspond to the first two
        integration intervals for the first integral (i.e., for $j = 0$,
        $1$) while the blue intervals correspond to the first two
        integration intervals for the second integral (i.e., for $j =
        1$, $2$). We took $n = 6$ for this figure.}
    \label{fig:eta_0_int}
\end{figure}

From \eqref{eqn:A_eigenfunctions}, we see that $\sqrt{\lambda_l} \in
]0,n\pi[$ for $l = 1,\ldots,n$.  As $s$ increases from $-1$ to $1$ more
and more eigenvalues will be caught in the integration intervals for
the integrals in \eqref{eqn:Ns} and contribute to $N(s)$.  Since the
eigenvalues of $A$ are uniformly distributed on the positive real
line when $v$ is constant, each integration interval will contain
the same number of eigenvalues as every other integration interval
for every value of $s$, and each interval will contain at most $n$
eigenvalues.

For $i= 1,\ldots,n$, let $s_i \in [-1,1]$ be defined by
\begin{equation*}
    n\arccos(s_i) = \sqrt{\lambda_{n-(i-1)}} \Leftrightarrow s_i =
    \cos\left(\frac{\sqrt{\lambda_{n-(i-1)}}}{n}\right).
\end{equation*}
Using this and \eqref{eqn:Ns} we obtain
\begin{equation}\label{eqn:Ns_final}
    N(s) = \sum_{i=1}^n \alpha_iH(s-s_i),
\end{equation}
where $H$ is the Heaviside step function and 
\begin{multline*}
    \alpha_i = \sum_{j=0}^{\infty} \tq\left(\left(\sqrt{\lambda_{n-(i-1)}} +
    2\pi n j\right)^2\right)\left(\sqrt{\lambda_{n-(i-1)}} +
    2\pi n j\right) \\
    + \sum_{j=1}^{\infty} \tq\left(\left(\sqrt{\lambda_{n+i}} +
    2\pi n j\right)^2\right)\left(\sqrt{\lambda_{n+i}} +
    2\pi n j\right).
\end{multline*}
(For our choice of $\tq$ in \eqref{eqn:gaussian_FT}, both of the above
series converge.)  The upshot of \eqref{eqn:Ns_final} is that the $n$
points of increase of $N(s)$ are given by $s_i$.

Finally, if the wavespeed $v$ is not constant, then the eigenvalues of
$A$ will typically not be uniformly distributed along the positive real
line as in they are Figure~\ref{fig:eta_0_int} (when $v$ is constant).
In contrast, the most probable scenario is that the spectrum of $A$ is
irregularly distributed along the positive real line in which case the
summation in \eqref{eqn:Ns} gives an infinite number of points of
increase for $N(s)$.


\subsection{Proof of Lemma~\ref{lem:trans_to_slow}}
\label{subsec:trans_to_slow_proof}

Because the slowness coordinate transformation is given by
\eqref{eqn:slowness}, the chain rule implies
\[
    \der{}{u}{x} = \der{}{\tu}{\tx}\der{}{\tx}{x} 
        = \frac{1}{\tv} \der{}{\tu}{\tx}
    \eqand{and}
    \der{2}{u}{x} = \frac{1}{\tv}\der{}{}{\tx}\left(\frac{1}{\tv}
    \der{}{\tu}{\tx}\right).
\]
Using this and $\tu_{tt}(\tx,t) = u_{tt}(x,t)$ in \eqref{eqn:cauchys}
gives 
\[
    -\tv\der{}{}{\tx}\left(\frac{1}{\tv}\der{}{\tu}{\tx}\right) +
    \tu_{tt} = 0.
\]
The boundary conditions follow from the above calculations, the
identity $\tx(0) = 0$, and the definition $\tx(\xmax) = \txmax$.  

The initial condition $\tu_t|_{t=0} = 0$ holds for $\tu$ since we are
not making any coordinate transformations in time.  The derivation of
$\tb$ requires some care.  First, we note that if $u$, $w \in
L^2[0,\xmax]$, then $\tu(\tx) = u(x(\tx))$ and $\tw(\tx) = w(x(\tx)) \in
L^2[0,\txmax]$ and
\begin{equation}\label{eqn:ip_to_slowness}
    \oovs{u}{w} = \ootv{\tu}{\tw}.
\end{equation}
In terms of distributions, for functions $h$ that are (right) continuous
at $x = 0$, we have
\begin{equation}\label{eqn:delta_spatial}
    \oovs{\delta(x+0)}{h} = \frac{h(0)}{v^2(0)}.
\end{equation}
In light of \eqref{eqn:ip_to_slowness} and \eqref{eqn:delta_spatial},
the transformation of the distribution $\delta(x+0)$ to slowness
coordinates, denoted $\ti{\delta}(\tx+0)$ (since $\tx(0) = 0$), should
satisfy
\[
    \ootv{\ti{\delta}(\tx+0)}{\ti{h}} = \frac{h(0)}{v^2(0)}.
\]
We take $\ti{\delta}(\tx+0) = \frac{1}{\tv(0)}\delta(\tx+0)$; then 
\[
    \ootv{\ti{\delta}(\tx+0)}{\ti{h}} =
    \ootv{\frac{1}{\tv(0)}\delta(\tx+0)}{\ti{h}} = 
    \frac{\ti{h}(0)}{\tv^2(0)} = \frac{h(0)}{v^2(0)}
\]
because $\tx(0) = 0$.  Thus $b = v(0)\tq(A)^{1/2}\delta(x+0)$
transforms to $\tb = \tq(\tA)^{1/2}\delta(\tx+0)$.  

Finally, $\tA$ is self adjoint and positive definite with respect to
$\ootv{\cdot}{\cdot}$ thanks to \eqref{eqn:ip_to_slowness} and the facts
that $A$ is self adjoint and positive definite with respect to
$\oovs{\cdot}{\cdot}$.  


\subsection{Proof of Lemma~\ref{lem:1stc}}\label{subsec:1stc_proof}

Suppose $\tu$ and $\tw$ solve \eqref{eqn:1stc}.  We prove that $\tu$
solves \eqref{eqn:cauchyslow}; the proof that $\tw$ solves
\eqref{eqn:cauchyw} is similar.  

We differentiate the first PDE in \eqref{eqn:1stc} with respect to $t$
and the second with respect to $\tx$ and subtract the results to find
\[
    \frac{1}{\tv}\tu_{tt} - \left(\frac{1}{\tv}\tu_{\tx}\right)_{\tx}
        = \tw_{\tx t} - \tw_{t \tx} = 0.
\]
Multiplying both sides of the above identity by $\tv$ gives $\tA\tu +
\tu_{tt} = 0$, as in \eqref{eqn:cauchyslow}.  

The boundary condition $\tu|_{\tx = \txmax} = 0$ follows immediately
from \eqref{eqn:1stc}; we differentiate the boundary condition
$\tw|_{\tx = 0}$ with respect to $t$ and use the second PDE in
\eqref{eqn:1stc} to find $0 = \tw_t|_{\tx = 0} =
\left(\frac{1}{\tv}\tu_{\tx}\right)|_{\tx = 0}$, which implies
$\tu_{\tx}|_{\tx = 0} = 0$.  We follow a similar procedure for the
initial conditions; $\tu|_{t=0} = \tb$ is trivial.  We differentiate the
initial condition $\tw|_{t=0} = 0$ with respect to $\tx$ and use the
first PDE in \eqref{eqn:1stc} to find $0 = \tw_{\tx}|_{t=0} =
\left(\frac{1}{\tv}\tu_t\right)|_{t=0}$, so $\tu_t|_{t=0} = 0$.  


\subsection{Proof of Lemma~\ref{lem:snapshots_uw}}
\label{subsec:snapshots_uw_proof}

We have already essentially proved the first part of this lemma (see
\eqref{eqn:snapshots_def} and \eqref{eqn:tstepping}).

To prove the second part of the lemma, we begin by noting that the
solution to \eqref{eqn:cauchyw} is 
\begin{equation*}
    \tw(\tx,t) = \sin\left(t\sqrt{\tC}\right)
    \tC^{-1/2}\frac{1}{\tv}\der{}{\tb}{\tx}.
\end{equation*}
Then Definition~\ref{def:snapshots_pd} implies
\begin{equation}\label{eqn:wtilde_sin}
    \tw_k = \tw(\tx,(k+0.5)\tau) =
    \sin\left((k+0.5)\tau\sqrt{\tC}\right)\tC^{-1/2}\frac{1}{\tv}
    \der{}{\tb}{\tx},
\end{equation}
and, in particular,
\begin{equation*}
    \tw_0 = \tw(\tx,0.5\tau) = \sin\left(0.5\tau\sqrt{\tC}\right)
    \tC^{-1/2}\frac{1}{\tv}\der{}{\tb}{\tx}.
\end{equation*}

Thus we need to show \eqref{eqn:dual_formula} and \eqref{eqn:wtilde_sin}
are equivalent, i.e., 
\begin{multline*}
    \left[\Ts_k\left(\cos\left(\tau\sqrt{\ti{P}_C}\right)\right) + 
        \Ts_{k-1}\left(\cos\left(\tau\sqrt{\ti{P}_C}\right)\right)
        \right]\sin\left(0.5\tau\sqrt{\ti{P}_C}\right) 
        \tC^{-1/2}\frac{1}{\tv}\der{}{\tb}{\tx} \\ = 
    \sin\left((k+0.5)\tau\sqrt{\tC}\right)\tC^{-1/2}\frac{1}{\tv}
        \der{}{\tb}{\tx}.
\end{multline*}
This means we must prove
\begin{equation}\label{eqn:sin_identity}
    \left[\Ts_k(\cos x) + \Ts_{k-1}(\cos x)\right]\sin(0.5x) = 
        \sin((k+0.5)x).
\end{equation}

The well-known identities 
\begin{equation}\label{eqn:Chebyshev_identities}
    T_j^{(2)}(x) = 2\ds\sum_{\substack{k=1 \\ k \ \mathrm{odd}}}^j
    T_k(x) \quad (j \ \mathrm{odd}) \eqand{and}
    T_j^{(2)}(x) = 2\ds\sum_{\substack{k=0 \\ k \ \mathrm{even}}}^j
    T_k(x) - 1 \quad (j \ \mathrm{even}),
\end{equation}
together with $T_j(\cos(x)) = \cos(jx)$, imply \eqref{eqn:sin_identity}
is equivalent to 
\begin{equation}\label{eqn:induction_hypothesis}
    \left[2\ds\sum_{\substack{j=1 \\ j \ \mathrm{odd}}}^k  \cos(jx) + 
        2\ds\sum_{\substack{j=0 \\ j \ \mathrm{even}}}^k \cos(jx) -
        1\right]\sin(0.5x)
    = \sin((k+0.5)x).
\end{equation}
We will use induction to prove \eqref{eqn:induction_hypothesis} is an
identity.  The case $k = 0$ follows immediately.  For the induction
step, suppose \eqref{eqn:induction_hypothesis} holds; we will prove it
also holds with $k$ replaced by $k+1$.  

We have
\begin{align*}
    \sin((k+1.5)x) 
    &= \sin((k+1)x)\cos(0.5x) + \cos((k+1)x)\sin(0.5x) \\
    &= 0.5\left[\sin((k+1.5)x) + \sin((k+0.5)x)\right] 
        + \cos((k+1)x)\sin(0.5x);
\end{align*}
solving the above equation for $\sin((k+1.5)x)$ yields
\[
    \sin((k+1.5)x) = \sin((k+0.5)x) + 2\cos((k+1)x)\sin(0.5x).
\]
This and the induction hypothesis \eqref{eqn:induction_hypothesis} imply
\begin{equation*}
    \sin((k+1.5)x) 
    = \left[2\ds\sum_{\substack{j=1 \\ j \ \mathrm{odd}}}^{k+1}\cos(jx)
        + 2\ds\sum_{\substack{j=0 \\ j \ \mathrm{even}}}^{k+1}\cos(jx)
        -1\right]\sin(0.5x),
\end{equation*}
as required.

Finally, the recursion \eqref{eqn:dual_recursion} follows from
\eqref{eqn:dual_formula} because the Chebyshev polynomials satisfy
$\Ts_{k+1}(x) = 2x\Ts_k(x) - \Ts_{k-1}(x)$ and (where all Chebyshev
polynomials are evaluated at $\ti{P}_C$)
\begin{align*}
    \frac{\tw_{k+1} - 2\tw_k + \tw_{k-1}}{\tau^2} 
    &= \frac{\left\{\Ts_{k+1} + \Ts_k - 
        2\left[\Ts_{k} + \Ts_{k-1}\right] + \Ts_{k-1}+\Ts_{k-2}\right\}
        \tw_0}{\tau^2} \\
    &= \frac{\left[\Ts_{k+1} - \Ts_k - \Ts_{k-1}+\Ts_{k-2}\right]\tw_0}
        {\tau^2} \\
    &= \frac{\left[2\ti{P}_C\Ts_k - \Ts_{k-1} - \Ts_k - \Ts_{k-1} 
        - \Ts_k + 2\ti{P}_C\Ts_{k-1} \right]\tw_0}{\tau^2} \\
    &= -\frac{2}{\tau^2}\left(I - \ti{P}_C\right)
        \left[\Ts_k + \Ts_{k-1} \right]\tw_0 \\
    &= \xi\left(\ti{P}_C\right)\tw_k.
\end{align*}
The initial condition $\tw_0 + \tw_{-1} = 0$ can be derived from
\eqref{eqn:dual_formula}:
\[
    \tw_0 + \tw_{-1} = \left[\Ts_0\left(\ti{P}_C\right) 
        + \Ts_{-1}\left(\ti{P}_C\right)\right]\tw_0
        + \left[\Ts_{-1}\left(\ti{P}_C\right) 
        + \Ts_{-2}\left(\ti{P}_C\right)\right]\tw_0 = 0
\]
because $\Ts_0 = 1$, $\Ts_{-1} = 0$, and $\Ts_{-2} = -1$.  


\subsection{Proof of Lemma~\ref{lem:leapfrog}}\label{subsec:leapfrog}

In order to avoid getting too involved in technical details, we present a
proof of Lemma~\ref{lem:leapfrog} in a discrete setting.  In particular,
we discretize the differential operators involved in the proof using
finite differences.  This allows us to circumvent the technicalities
involved in specifying the domains of the differential operators in
question, although, as we will see, the discrete operators still retain
information about these domains.  Moreover, this proof highlights many
of the details of numerical simulations.  

We discretize on a staggered grid, illustrated in 
Figure~\ref{fig:grid}.  The $m+1$ ``primary'' nodes
$\{\ti{x}^j\}_{j=1}^{m+1}$ are indicated by the symbol $\circ$ and the
$m+1$ ``dual'' nodes $\{\ha{x}^j\}_{j=0}^{m}$ are indicated by the
symbol $\times$.  We take $m \gg 1$ to ensure that the continuous
operators are well approximated by the discrete operators.  In practice,
we use a uniform grid with $\ti{h}_j = h$ for $j = 1,\ldots,m$,
$\ha{h}_1 = h/2$, and $\ha{h}_j = h$ for $j = 2,\ldots,m$.  However, it
is convenient for our purposes to keep the grid steps arbitrary for now
(as long as the primary and dual grid points alternate).
\begin{figure}[!hb]
    \centering
    \includegraphics[width=\textwidth]{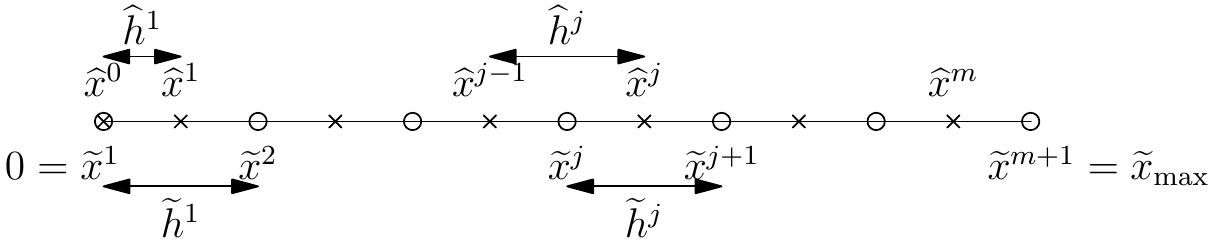}
    \caption{In this figure, we sketch the staggered grid we use to
    construct finite-difference approximations of differential
    operators.  The ``primary'' nodes $\{\ti{x}^j\}_{j=1}^{m+1}$ are
    indicated by the symbol $\circ$ and the ``dual'' nodes
    $\{\ha{x}^j\}_{j=0}^{m}$ are indicated by the symbol $\times$.}
    \label{fig:grid}
\end{figure}

Recall that the operator $\tA$ is defined by
\begin{equation}\label{eqn:tA_reminder}
    \tA\tu = -\tv\der{}{}{\tx}
    \left(\frac{1}{\tv}\der{}{\tu}{\tx}\right), \eqand{where}
    \tu_{\tx}|_{\tx = 0} = 0 \eqand{and} \tu|_{\tx = \txmax} = 0.
\end{equation}
Using centered differences, we discretize $\tu$ on the
primary nodes and $\tu_{\tx}$ on the dual nodes to obtain
\cite{Druskin:1999:GSR}
\begin{equation}\label{eqn:discrete_tA}
    \begin{aligned}
    \tA\tu(\tx^j) 
    &\approx 
    -\frac{\tv^j}{\ha{h}^{j}} 
        \left[\frac{1}{\ha{v}^j}\der{}{\tu}{\tx}(\ha{x}^{j}) - 
	\frac{1}{\ha{v}^{j-1}}\der{}{\tu}{\tx}(\ha{x}^{j-1})\right] \\ 
    &\approx 
    -\frac{\tv^j}{\ha{h}^{j}}
    \left[\left(\frac{\tu^{j+1} - \tu^j} 
        {\ha{v}^{j}\ti{h}^{j}}\right)-
	\left(\frac{\tu^j - \tu^{j-1}}{\ha{v}^{j-1}\ti{h}^{j-1}}
	\right)\right] 
	\eqfor j = 2,\ldots,m,
    \end{aligned}
\end{equation}
where $\tu^j = \tu\left(\tx^j\right)$ for $j = 1,\ldots,m+1$, $\tv^j$ is
an approximation to $\tv\left(\tx^j\right)$ for $j = 1,\ldots,m+1$, and
$\ha{v}^j$ is an approximation to $\tv\left(\ha{x}^j\right)$ for $j =
0,\ldots,m$.  For example, if $\tv$ is continuous, we may take $\tv^j =
\tv\left(\tx^j\right)$ and $\ha{v}^j = \tv\left(\ha{x}^j\right)$.  If
$\tv$ is not continuous, we may follow \cite{Borcea:2002:OFD} and take
\[
    \frac{1}{\tv^j} \equiv \frac{1}{\ha{h}^j}
        \int_{\ha{x}^{j-1}}^{\ha{x}^j}\frac{1}{\tv(\tx)} \di{\tx}
\] 
(so $\tv^j$ is the harmonic mean of $\tv$ on $\left(\ha{x}^{j-1},
\ha{x}^j\right)$) and 
\[
    \ha{v}^j \equiv \frac{1}{\ti{h}^j} \int_{\ti{x}^j}^{\ti{x}^{j+1}} 
        \tv(\tx) \di{\tx}
\] 
(so $\ha{v}^j$ is the arithmetic mean of $\tv$ on $\left(\ti{x}^j,
\ti{x}^{j+1}\right)$).

We discretize the Dirichlet boundary condition $\tu|_{\txmax} = 0$ by
setting $\tu^{m+1} = 0$.  To handle the Neumann boundary condition at
$\tx = 0$, we introduce a ``ghost node'' at $\tx^{0} = -\ti{h}^0$.
Then, for $j = 1$, \eqref{eqn:discrete_tA} is
\[
    \tA\tu(\tx^1) \approx -\frac{\tv^1}{\ha{h}^1}
    \left[\left(\frac{\tu^2 - \tu^1} 
        {\ha{v}^1\ti{h}^1}\right)-
	\left(\frac{\tu^1 - \tu^{0}}{\ha{v}^0\ti{h}^0}\right)\right].
\]
We discretize the Neumann boundary condition $\tu_{\tx}|_{\tx=0}=0$ by
setting \footnote{For smooth $\tv$ and uniform grid steps $\ti{h}^j = h$
for $j = 1,\ldots,m$, $\ha{h}^1 = h/2$, and $\ha{h}_j = h$ for $j =
2,\ldots,m$, \eqref{eqn:discrete_tA} is an $O(h^2)$ approximation of
$\tA$.  An equivalent formulation arises by taking $\ha{x}^0 =
-\ha{h}^1/2$ (instead of $\ha{x}^0 = 0$) and discretizing the
Neumann boundary condition by $\partial \tu/\partial\tx(\ha{x}^0) +
\partial\tu/\partial\tx(\ha{x}^1) \approx 0$, which, in the uniform grid
case, is an $O(h^2)$ approximation to $\partial\tu/\partial\tx(0) = 0$.}
\begin{equation}\label{eqn:discrete_Neumann}
    \frac{\tu^1 - \tu^{0}}{\ha{v}^0\ti{h}^0} = 0.
\end{equation}
In summary, we define $\btu = \left[\tu^1, \ldots, \tu^m\right]^T \in
\mathbb{R}^m$ (where we have implicitly taken $\tu^{m+1} = 0$); then
$\tA\tu(\tx^j) \approx (\btA\btu)^j$ for $j = 1,\ldots,m$, 
where we define the following matrices in $\mathbb{R}^{m\times m}$: 
\begin{equation}\label{eqn:matrices}
\begin{aligned}
    \btA = \ha{\la{R}}\ti{\la{S}}, \quad 
    \ha{\la{R}} &\equiv \ti{\la{V}}\ha{\la{\Delta}}, \quad
    \ti{\la{S}} \equiv \ha{\la{V}}^{-1}\ti{\la{\Delta}}, \quad
    \ti{\la{V}} \equiv \diag(\tv^1,\ldots,\tv^m), \quad
    \ha{\la{V}} \equiv \diag(\ha{v}^1,\ldots,\ha{v}^m), \\
    &\ti{\la{\Delta}} \equiv \diag
        (1/\ti{h}^1,\ldots,1/\ti{h}^m)\la{T}, \quad
    \ha{\la{\Delta}} \equiv \diag
        (1/\ha{h}^1,\ldots,1/\ha{h}^m)\la{T}^T, 
\end{aligned}
\end{equation}
and $\la{T}$ is the $m\times m$ Toeplitz matrix with $1$ on the main
diagonal, $-1$ on the subdiagonal, and $0$ elsewhere.
Finally, $\btA$ is self adjoint and positive definite with respect
to the inner product
\[
    \left\langle \ti{\la{f}}, \ti{\la{g}} \right\rangle_{\ha{h}/\tv}
    \equiv \sum_{j=1}^m \ti{f}^j \ti{g}^j \frac{\ha{h}^j}{\tv^j};
\]
if $\ti{\la{f}}$ and $\ti{\la{g}}$ are viewed as primary-grid
discretizations of functions $\ti{f}$ and $\ti{g}$ satisfying the
boundary conditions in \eqref{eqn:tA_reminder}, then this discrete inner
product is the midpoint-rule approximation of the inner product
$\ootv{\cdot}{\cdot}$.  

Here and throughout the remainder of this section, bold, lowercase Latin
letters adorned with $\widetilde{\phantom{\la{u}}}$ or
$\widehat{\phantom{\la{u}}}$ denote vectors in
$\mathbb{R}^m$ that correspond to discretizations of functions on the
primary grid or dual grid, respectively.  In particular, the
discretized versions of the primary and dual snapshots are denoted by
\begin{align*}
    \btu_k &\equiv \left[\tu_k^1,\ldots,\tu_k^m\right]^T \equiv
        \left[\tu_k(\tx_1),\ldots,\tu_k(\tx_m)\right]^T \\
	\intertext{and}
    \bhw_k &\equiv \left[\hw_k^1,\ldots,\hw_k^m\right]^T \equiv
        \left[\tw_k(\hx_1),\ldots,\tw_k(\hx_m)\right]^T,
\end{align*}
respectively.  Similarly, bold, uppercase Greek or Latin letters adorned
with $\widetilde{\phantom{\la{u}}}$ or $\widehat{\phantom{\la{u}}}$
denote $m\times m$ matrices that act on functions discretized on the
primary and dual grids, respectively.  For example, let us consider the
matrix $\ti{\la{S}} = \ha{\la{V}}^{-1}\ti{\la{\Delta}}$.  The matrix
$\ti{\la{\Delta}}$ acts on the $k\textsuperscript{th}$ discretized
snapshot $\btu_k$ to produce the vector $\ti{\la{\Delta}}\btu_k$, which
is an approximation of $\frac{\partial \tu_k}{\partial \tx}$ on the
\emph{dual grid} (because, as discussed above, we discretize
$\frac{\partial\tu}{\partial\tx}$ on the dual grid).  Since vector
$\ti{\la{\Delta}}\btu_k$ is a discretization of a function on the dual
grid, it can be acted on by the matrix $\ha{\la{V}}^{-1}$.  In summary,
matrices with $\ti{\phantom{u}}$ (respectively, $\ha{\phantom{u}}$) act
on vectors with $\ti{\phantom{u}}$ (respectively, $\ha{\phantom{u}}$);
this notation allows us to retain
information about the domains of the continuous differential operators
in the discrete setting.  \footnote{Since all of the vectors we consider
are in $\mathbb{R}^m$ and all of the matrices are in
$\mathbb{R}^{m\times m}$, we are allowed to intermix notations in
matrix-vector multiplication, e.g., $\ha{\la{\Delta}}\btu_k$ is well
defined in a linear-algebraic sense; however, we are viewing the
matrices and vectors as discretizations of differential operators and
functions, respectively, on certain grids, so it is important to
distinguish between those defined on the primary grid versus those
defined on the dual grid.}  

We now focus on the discretization of the dual operator $\tC$: 
\begin{equation}\label{eqn:tC_reminder}
    \tC\tw =
    -\frac{1}{\tv}\der{}{}{\tx}\left(\tv\der{}{\tw}{\tx}\right),
    \eqand{where} \tw|_{\tx = 0} = 0 \eqand{and} \tw_{\tx}|_{\tx =
    \txmax} = 0.
\end{equation}
For $j = 0,\ldots,m$, we denote $\hw^j \equiv \tw(\hx^j)$.  Analogously
to what we did before, we discretize $\tw$ on the dual nodes and
$\tw_{\tx}$ on the primary nodes to arrive at
\begin{equation}\label{eqn:discrete_tC}
    \begin{aligned}
    \tC\tw(\hx^j) 
    &\approx -\frac{1}{\ha{v}^j\ti{h}^j}
        \left[\ti{v}^{j+1}\der{}{\tw}{\tx}(\tx^{j+1}) 
	- \ti{v}^j\der{}{\tw}{\tx}(\tx^j)\right] \\
    &\approx -\frac{1}{\ha{v}^j\ti{h}^j} 
    \left[\ti{v}^{j+1}\left(\frac{\hw^{j+1} - \hw^j}{\ha{h}^{j+1}}
    \right) - \ti{v}^j\left(\frac{\hw^j - \hw^{j-1}}{\ha{h}^j}\right)
	\right] 
    \eqfor j = 1,\ldots,m-1.
    \end{aligned}
\end{equation}
The Dirichlet boundary condition at $\tx = 0$ is discretized by $\hw_0 =
\tw(\hx^0) = \tw(0) = 0$, while the Neumann boundary condition is
discretized by introducing a ghost node $\hx^{m+1} = \hx^m +
\ha{h}^{m+1}$ and taking 
\[
    \ha{v}^{m+1}\left(\frac{\tw^{m+1} -\tw^m}{\ti{h}^{m+1}}\right) = 0.
\]
Then $\tC\tw(\hx^j) \approx (\bhC\bhw)^j$ for $j = 1,\ldots,m$, where 
\begin{equation}\label{eqn:bhC}
    \bhC \equiv \ti{\la{S}}\ha{\la{R}}. 
\end{equation}
Note $\bhC$ is self adjoint and positive definite with respect to the
inner product
\[
    \left\langle \ha{\la{f}}, \ha{\la{g}}\right \rangle_{\ti{h}\ha{v}}
    \equiv \sum_{j=1}^m \ha{f}^j\ha{g}^j\ti{h}^j\ha{v}^j;
\]
if $\ha{\la{f}}$ and $\ha{\la{g}}$ are viewed as dual-grid
discretizations of functions $\ha{f}$ and $\ha{g}$ satisfying the
boundary conditions in \eqref{eqn:tC_reminder}, then this discrete
inner product is the midpoint-rule approximation of the inner product
$\iptv{\cdot}{\cdot}$.  

From \eqref{eqn:matrices} and \eqref{eqn:bhC}, we find $\btA$ and $\bhC$
are similar; in particular
\begin{equation}\label{eqn:AC_similar}
    \btA = \ti{\la{S}}^{-1}\bhC\ti{\la{S}} \eqand{and}
    \btA = \ha{\la{R}}\bhC\ha{\la{R}}^{-1}.  
\end{equation}
(This is the only place in the proof where our notation does not work
perfectly --- in particular, $\ti{\la{S}}^{-1}$ acts on dual-grid
vectors while $\ha{\la{R}}^{-1}$ acts on primary-grid vectors.)
From this we obtain the following identities, which prove useful in
forthcoming calculations:
\begin{equation}\label{eqn:AtoC}
    \begin{aligned}
        \ti{\la{S}}\btA^{-1/2}\sin\left(0.5\tau\sqrt{\btA}\right) &=
            \sin\left(0.5\tau\sqrt{\bhC}\right)\bhC^{-1/2}\ti{\la{S}}
            \myspace
        \sin\left(0.5\tau\sqrt{\btA}\right)\btA^{-1/2}\ha{\la{R}} &= 
            \ha{\la{R}}\bhC^{-1/2}\sin\left(0.5\tau\sqrt{\bhC}\right).
    \end{aligned}
\end{equation}
We will prove the first of these identities --- the second identity
can be proved analogously.  We have
\begin{align*}
    \ti{\la{S}}\btA^{-1/2}\sin\left(0.5\tau\sqrt{\btA}\right)
    &= \ti{\la{S}}\sum_{j=0}^{\infty} \frac{(\tau/2)^{2j+1}}{(2j+1)!}
        \btA^j\\
    &= \ti{\la{S}}\sum_{j=0}^{\infty} \frac{(\tau/2)^{2j+1}}{(2j+1)!}
    \ti{\la{S}}^{-1}\bhC^j\ti{\la{S}} \\
    &= \sum_{j=0}^{\infty} \frac{(\tau/2)^{2j+1}}{(2j+1)!}
    \left(\bhC^{1/2}\right)^{2j+1}\bhC^{-1/2}\ti{\la{S}} \\
    &= \sin\left(0.5\tau\sqrt{\bhC}\right)\bhC^{-1/2}\ti{\la{S}}.
\end{align*}

Next, we define the matrices
\begin{equation}\label{eqn:Lambdas}
    \ti{\la{\Lambda}}_{\tau} \equiv \frac{2}{\tau}\ti{\la{\Delta}}
        \btA^{-1/2}\sin\left(0.5\tau\sqrt{\btA}\right) 
    \eqand{and}
    \ha{\la{\Lambda}}_{\tau}^T \equiv
        \frac{2}{\tau}\ti{\la{V}}^{-1}
	\sin\left(0.5\tau\sqrt{\btA}\right)\btA^{-1/2}\ti{\la{V}}
	\ha{\la{\Delta}};
\end{equation}
$\ti{\la{\Lambda}}_{\tau}$ and $\ha{\la{\Lambda}}_{\tau}^T$ are discrete
approximations of $\Lt$ and $\Lt^T$, respectively.  We consider the
following discrete approximation to \eqref{eqn:tstepping1st}:
\begin{equation}\label{eqn:tstepping1st_discrete}
    \begin{cases}
        \dfrac{\bhw_k - \bhw_{k-1}}{\tau} = 
	    \ha{\la{V}}^{-1}\ti{\la{\Lambda}}_{\tau}\btu_k 
	    &\text{for } k = 0,\ldots,2n-1, \myspace
	\dfrac{\btu_{k+1} - \btu_{k}}{\tau} = 
	    -\ti{\la{V}}\ha{\la{\Lambda}}_{\tau}^T\bhw_k 
	    &\text{for } k = 0,\ldots,2n-2, \myspace
	\btu_0 = \ti{\la{b}}, \ \bhw_0 + \bhw_{-1} = 0.
    \end{cases}
\end{equation}
Applying $-\ti{\la{V}}\ha{\la{\Lambda}}_{\tau}^T$ to the first equation
in \eqref{eqn:tstepping1st_discrete} and simplifying the result via the
second equation in \eqref{eqn:tstepping1st_discrete} gives
\begin{equation}\label{eqn:almost_primary}
    \frac{\btu_{k+1} - 2\btu_k + \btu_{k-1}}{\tau^2} 
    = -\ti{\la{V}}\ha{\la{\Lambda}}_{\tau}^T\ha{\la{V}}^{-1}
    \ti{\la{\Lambda}}_{\tau}\btu_k \eqand{for} k = 0,\ldots,2n-2.
\end{equation}
The initial conditions for this iteration are $\btu_0 = \ti{\la{b}}$ and
$\btu_1 = \btu_{-1}$ since, by the second equation in
\eqref{eqn:tstepping1st_discrete} (applied for $k = 0$ and $k = -1$),
\[
    \frac{\btu_1 - \btu_{-1}}{\tau} = \frac{\btu_1 - \btu_0}{\tau} +
    \frac{\btu_0 - \btu_{-1}}{\tau} =
    -\ti{\la{V}}\ha{\la{\Lambda}}_{\tau}^T\left(\bhw_0 +
    \bhw_{-1}\right) = 0.
\]
The operator on the right-hand side of \eqref{eqn:almost_primary}
satisfies
\begin{align*}
   -\tv\Lt\frac{1}{\tv}\Lt^T 
   &\approx  -\ti{\la{V}}\ha{\la{\Lambda}}_{\tau}^T\ha{\la{V}}^{-1}
   \ti{\la{\Lambda}}_{\tau} \\
   &= -\frac{4}{\tau^2}
   \sin\left(0.5\tau\sqrt{\btA}\right)\btA^{-1/2}
   \underbrace{\ha{\la{R}}\ti{\la{S}}}_{=\btA} 
   \btA^{-1/2}\sin\left(0.5\tau\sqrt{\btA}\right) \\
   &= -\frac{2}{\tau^2}\left[\la{I} -
   \cos\left(\tau\sqrt{\btA}\right)\right] \\
   &= \xi\left(\ti{\la{P}}\right),
\end{align*}
where $\ti{\la{P}} \equiv \cos\left(\tau\sqrt{\btA}\right)$.  This, in
combination with \eqref{eqn:almost_primary}, implies $\btu_k$ satisfies
the recursion
\[
    \frac{\btu_{k+1} - 2\btu_k + \btu_{k-1}}{\tau^2} 
    = \xi\left(\ti{\la{P}}\right)\btu_k 
    \eqfor k = 0,\ldots,2n-2, \qquad \btu_0 =
    \ti{\la{b}}, \ \btu_1 = \btu_{-1},
\]
which is a discrete approximation of \eqref{eqn:primary_recursion}.
Note in the continuum limit we have $-\tv\Lt\frac{1}{\tv}\Lt^T =
\xi\left(\ti{P}\right)$.

We now apply the operator $\ha{\la{V}}^{-1}\ti{\la{\Lambda}}_{\tau}$ to
the second equation in \eqref{eqn:tstepping1st_discrete} and simplify
using the first equation in \eqref{eqn:tstepping1st_discrete} to find
\begin{equation}\label{eqn:almost_dual}
     \frac{\bhw_{k+1} - 2\bhw_k + \bhw_{k-1}}{\tau^2} 
     = -\ha{\la{V}}^{-1}\ti{\la{\Lambda}}_{\tau}\ti{\la{V}}
    \ha{\la{\Lambda}}_{\tau}^T\bhw_k \eqfor k = 0,\ldots,2n-2.
\end{equation}
The initial conditions for this recursion are $\bhw_0 + \bhw_1 = 0$ and
(taking $k = 0$ in the first equation in
\eqref{eqn:tstepping1st_discrete} and using \eqref{eqn:AtoC})
\[
    \bhw_0 = \frac{\tau}{2}\ha{\la{V}}^{-1}\ti{\la{\Lambda}}_{\tau}
    \ti{\la{b}} = \ti{\la{S}}\btA^{-1/2}
    \sin\left(0.5\tau\sqrt{\btA}\right)\ti{\la{b}}
    = \sin\left(0.5\tau\sqrt{\bhC}\right)\bhC^{-1/2}\ti{\la{S}}
    \ti{\la{b}}.
\]
This is a discrete approximation to $\tw_0 =
\sin\left(0.5\tau\sqrt{\tC}\right)\tC^{-1/2}\dfrac{1}{\tv}
\der{}{\tb}{\tx}$.  Moreover, by \eqref{eqn:AtoC} we have
\begin{align*}
    -\frac{1}{\tv}\Lt\tv\Lt^T 
    &\approx -\ha{\la{V}}^{-1}\ti{\la{\Lambda}}_{\tau}\ti{\la{V}}
    \ha{\la{\Lambda}}_{\tau}^T \\
    &= -\frac{4}{\tau^2}\ti{\la{S}}\btA^{-1/2}
    \sin\left(0.5\tau\sqrt{\btA}\right)
    \sin\left(0.5\tau\sqrt{\btA}\right)\btA^{-1/2}\ha{\la{R}} \\
    &= -\frac{4}{\tau^2}\sin\left(0.5\tau\sqrt{\bhC}\right)
    \bhC^{-1/2}\underbrace{\ti{\la{S}}\ha{\la{R}}}_{=\bhC}\bhC^{-1/2}
    \sin\left(0.5\tau\sqrt{\bhC}\right) \\
    &= \xi\left(\ti{\la{P}}_C\right),
\end{align*}
where $\ha{\la{P}}_C \equiv \cos\left(\tau\sqrt{\bhC}\right)$.
Then \eqref{eqn:almost_dual} implies $\hw_k$ satisfies the recursion
\[
    \begin{aligned}
        &\frac{\bhw_{k+1} - 2\bhw_k + \bhw_{k-1}}{\tau^2} 
            = \xi\left(\ha{\la{P}}_C\right)\bhw_k 
	    \eqfor k = 0,\ldots,2n-2, \\
        &\bhw_0 + \bhw_{-1} = 0, \ 
            \bhw_{0} =\sin\left(0.5\tau\sqrt{\bhC}\right)\bhC^{-1/2}
            \ti{\la{S}}\ti{\la{b}},
    \end{aligned}
\]
which is a discrete approximation of \eqref{eqn:dual_recursion}.  Again,
in the continuum limit, we have $-\frac{1}{\tv}\Lt\tv\Lt^T =
\xi\left(\ti{P}_C\right)$.  

Finally, we must prove that $\Lt^T$ is indeed the adjoint of $\Lt$ with
respect to the inner product $\Lttx{\cdot}{\cdot}$.  Let $\ti{f}, \ha{g}
\in L^2[0,\txmax]$ such that $\Lt\ti{f}$, $\Lt^T\ha{g} \in
L^2[0,\txmax]$ with $\ti{f}$ satisfying the boundary conditions in
\eqref{eqn:tA_reminder} and $\ha{g}$ satisfying the boundary conditions
in \eqref{eqn:tC_reminder}.  Also, let $\ti{\la{f}} \equiv
\left[\ti{f}(\tx^1),\ldots,\ti{f}(\tx^m)\right]^T$ and $\ha{\la{g}}
\equiv \left[\ha{g}(\hx^1),\ldots,\ha{g}(\hx^m)\right]^T$.  We define
the inner products 
\[
    \left\langle \ha{\la{f}}, \ha{\la{g}}\right\rangle_{\ti{h}} 
    \equiv \ds\sum_{j=1}^m \ha{f}^j \ha{g}^j \ti{h}^j \eqand{and}
    \left\langle \ti{\la{f}}, \ti{\la{g}}\right\rangle_{\ha{h}}
    \equiv \ds\sum_{j=1}^m \ti{f}^j \ti{g}^j \ha{h}^j.
\]
Then, using \eqref{eqn:matrices}, \eqref{eqn:Lambdas}, and the fact that
functions of $\btA$ are self adjoint with respect to $\left\langle
\cdot, \cdot \right\rangle_{\ha{h}/\tv}$, we obtain
\begin{align*}
    \Lttx{\Lt\ti{f}}{\ha{g}} 
    &\approx \left\langle \ti{\la{\Lambda}}_{\tau}\ti{\la{f}},
    \ha{\la{g}} \right\rangle_{\ti{h}} \\
    &=\left\langle 
    \frac{2}{\tau}\ti{\la{\Delta}}\btA^{-1/2}
    \sin\left(0.5\tau\sqrt{\btA}\right)\ti{\la{f}}, 
    \ha{\la{g}}\right\rangle_{\ti{h}} \\
    &=\frac{2}{\tau}\left\langle \la{T}\btA^{-1/2}
    \sin\left(0.5\tau\sqrt{\btA}\right)\ti{\la{f}},
    \ha{\la{g}}\right\rangle_{l^2(\mathbb{R}^m)} \\
    &=\frac{2}{\tau}\left\langle \btA^{-1/2}
    \sin\left(0.5\tau\sqrt{\btA}\right)\ti{\la{f}},
    \la{T}^T\ha{\la{g}}\right\rangle_{l^2(\mathbb{R}^m)} \\
    &= \frac{2}{\tau}\left\langle \btA^{-1/2}
    \sin\left(0.5\tau\sqrt{\btA}\right)\ti{\la{f}},
    \ti{\la{V}}\diag(1/\ha{h}^1,\ldots,1/\ha{h}^m)
    \la{T}^T\ha{\la{g}}\right\rangle_{\ha{h}/\tv} \\
    &= \frac{2}{\tau}\left\langle \btA^{-1/2}
    \sin\left(0.5\tau\sqrt{\btA}\right)\ti{\la{f}},
    \ha{\la{R}}\ha{\la{g}}\right\rangle_{\ha{h}/\tv} \\
    &= \frac{2}{\tau}\left\langle \ti{\la{f}},
    \sin\left(0.5\tau\sqrt{\btA}\right)\btA^{-1/2}\ha{\la{R}}\ha{\la{g}}
    \right\rangle_{\ha{h}/\tv} \\
    &= \left\langle \ti{\la{f}}, \frac{2}{\tau}\ti{\la{V}}^{-1}
    \sin\left(0.5\tau\sqrt{\btA}\right)\btA^{-1/2}\ha{\la{R}}\ha{\la{g}}
    \right\rangle_{\ha{h}} \\
    &= \left\langle \ti{\la{f}}, \ha{\la{\Lambda}}_{\tau}^T\ha{\la{g}}
    \right\rangle_{\ha{h}} \\
    &\approx \Lttx{\ti{f}}{\Lt^T\ha{g}}.
\end{align*}


\subsection{Proof of Proposition~\ref{prop:alg_2}}
\label{subsec:proof_prop_alg_2}

First, we use Algorithm~\ref{alg:algorithm_2} to show that $\ou_1$ and
$\ou_2$ are orthogonal.  We have
\begin{align*}
    \ootv{\ou_2}{\ou_1}
    &= \ootv{\ou_1 - \g_j\tv\Lt^T\ow_1}{\ou_1} \\
    &= \ootv{\ou_1}{\ou_1} - \g_j\Lttx{\Lt^T\ow_1}{\ou_1} \\
    &= \ghat_1^{-1} - \g_j\Lttx{\ow_1}{\Lt\ou_1} \\
    &= \ghat_1^{-1} - \g_j\iptv{\ow_1}{\frac{1}{\tv}\Lt\ou_1} \\
    &= \ghat_1^{-1} - \g_j\ghat_1^{-1}\iptv{\ow_1}{\ow_1} \\
    &= \ghat_1^{-1} - \ghat_1^{-1} \\
    &= 0.
\end{align*}
Similarly, $\iptv{\ow_2}{\ow_1} = 0$.  

Now, suppose for induction that, via Algorithm~\ref{alg:algorithm_2},
we have constructed $\ou_1, \ldots, \ou_j$ such that
$\ootv{\ou_j}{\ou_k} = 0$ for $k = 1,\ldots,j-1$ and
$\ow_1,\ldots,\ow_j$ such that $\iptv{\ow_j}{\ow_k} = 0$ for $k =
1,\ldots,j-1$.  Our goal is to show $\ootv{\ou_{j+1}}{\ou_k} = 0$ for $k
= 1,\ldots,j$.  Proceeding as in the previous paragraph, we find
\begin{align*}
    \ootv{\ou_{j+1}}{\ou_k}
    &= \ootv{\ou_j - \g_j\tv\Lt^T\ow_j}{\ou_k} \\
    &= \ootv{\ou_j}{\ou_k} - \g_j\iptv{\ow_j}{\frac{1}{\tv}\Lt\ou_k} \\
    &= \ootv{\ou_j}{\ou_k} - \g_j\ghat_k^{-1}
        \iptv{\ow_j}{\ow_k - \ow_{k-1}}.
\end{align*}
By the induction hypothesis, the last expression above is zero for $k =
1,\ldots,j-1$, while for $k = j$ it is equal to
\begin{equation*}
    \ghat_j^{-1} - \g_j\ghat_j^{-1}\iptv{\ow_j}{\ow_j} = \ghat_j^{-1} -
        \ghat_j^{-1} = 0.
\end{equation*}
A similar argument shows $\iptv{\ow_{j+1}}{\ow_k} = 0$ for $k =
1,\ldots,j$.  

Finally, the equalities
\[
        \mathrm{span}\left\{\ou_1,\ldots,\ou_n\right\} =
            \ti{\mathcal{K}}_n^u\left(\tu_0,\ti{P}\right)
	\eqand{and}
        \mathrm{span}\left\{\ow_1,\ldots,\ow_n\right\} =
            \ti{\mathcal{K}}_n^w\left(\tw_0,\ti{P}_C\right)
\]
are corollaries of Lemmas~\ref{lem:u_Lanczos} and \ref{lem:w_Lanczos},
respectively, in combination with the fact that the Lanczos algorithm
generates an orthonormal basis for the Krylov subspace
$\mathcal{K}_n(b,B)$, where $b$ is the starting vector and $B$ is the
operator in question \cite{Parlett:1998:SEP}.


\subsection{Proofs of Lemmas~\ref{lem:u_Lanczos} and
\ref{lem:w_Lanczos}} \label{subsec:proof_u_Lanczos}

From Algorithm~\ref{alg:algorithm_2} we have
\begin{align*}
    \ou_{j+1} 
    &= \ou_j - \g_j\tv\Lt^T\ow_j \\
    &= \ou_j - \g_j\tv\Lt^T\left(\ow_{j-1} + \ghat_j\frac{1}{\tv}\Lt
        \ou_j\right) \\
    &= \ou_j - \g_j\tv\Lt^T\ow_{j-1} - \g_j\ghat_j\tv\Lt^T
        \frac{1}{\tv}\Lt\ou_j \\
    &= \ou_j + \g_j\g_{j-1}^{-1}\left(\ou_j - \ou_{j-1}\right) +
	\g_j\ghat_j\xi\left(\ti{P}\right)\ou_j, \numberthis
	\label{eqn:u_Lanczos_int}
\end{align*}
where the last equality follows from \eqref{eqn:xi_factored}.

We define $\vartheta_j = \ou_j/\normootv{\ou_j} = \ghat_j^{1/2}\ou_j$.
Then \eqref{eqn:u_Lanczos_int} becomes
\begin{equation*}
    \ghat_{j+1}^{-1/2}\vartheta_{j+1} 
    = \ghat_j^{-1/2}\left(1 + \g_j\g_{j-1}^{-1}\right)\vartheta_j
        -\g_j\g_{j-1}^{-1}\ghat_{j-1}^{-1/2}\vartheta_{j-1} +
	\g_j\ghat_j^{1/2}\xi\left(\ti{P}\right)\vartheta_j.
\end{equation*}
This can be rearranged as
\begin{equation*}
    \xi\left(\ti{P}\right)\vartheta_j = b_j^u\vartheta_{j+1} +
    a_j^u\vartheta_j + b_{j-1}^u\vartheta_{j-1},
\end{equation*}
where $a_j^u$ and $b_j^u$ are defined as in \eqref{eqn:ab_u}.  Because
the functions $\vartheta_j$ ($j = 1,\ldots,n$) form an orthonormal set
by Proposition~\ref{prop:alg_2}, this is exactly the Lanczos three-term
recurrence relation \cite{Parlett:1998:SEP}.  

Lemma~\ref{lem:w_Lanczos} may be proved similarly.  


\subsection{Proof of Lemma~\ref{lem:poly_alg_2}}
\label{subsec:proof_poly_alg_2}

We use induction to prove this lemma for the primary orthogonalized
snapshots, $\ou_j$.  For the base case, we define $q_1^u(x) \equiv 1$;
then $\ou_1 = q_1^u\left(\xi\left(\ti{P}\right)\right)\ou_1$, $q_1^u$ is
a polynomial of degree $0$, and $q_1^u(0) = 1$.   

Next, let $j \ge 2$.  Suppose for induction that $\ou_k =
q_k^u\left(\xi\left(\ti{P}\right)\right)\ou_1$ for $k = 1,\ldots,j$,
where $q_k^u$ is a polynomial of degree $k-1$ such that $q_k^u(0) = 1$.
Then Algorithm~\ref{alg:algorithm_2} and the induction hypothesis give 
(see \eqref{eqn:u_Lanczos_int})
\begin{equation*}
    \ou_{j+1} 
    = \left(1+\g_j\g_{j-1}^{-1}\right)\ou_j -
        \g_j\g_{j-1}^{-1}\ou_{j-1} + 
	\g_j\ghat_j\xi\left(\ti{P}\right)\ou_j
    = q_{j+1}^u\left(\xi\left(\ti{P}\right)\right)\ou_1,
\end{equation*}
where
\[
    q_{j+1}^u(x) \equiv \left(1+\g_j\g_{j-1}^{-1}\right)q_j^u(x) -
        \g_j\g_{j-1}^{-1}q_{j-1}^u(x) + \g_j\ghat_jxq_j^u(x)
\]
is a polynomial of degree $j$ (since $\g_j$, $\ghat_j \ne 0$).
Moreover, by the induction hypothesis we have
\[
    q_{j+1}^u(0) = \left(1+\g_j\g_{j-1}^{-1}\right)-\g_j\g_{j-1}^{-1}=1.
\]

The proof for the dual orthogonalized snapshots is similar.  


\subsection{Proof of Remark~\ref{rem:polynomials}}
\label{subsec:proof_of_rem_polynomials}

For simplicity, we will work in spatial coordinates instead of in
slowness coordinates for this proof.  We define $p_j^{\xi} \equiv
\ghat_j^{1/2}\ghat_1^{-1/2}q_j^u$; then, thanks to
Lemma~\ref{lem:poly_alg_2}, we have $\vartheta_j =
p_j^{\xi}\left(\xi\left(\ti{P}\right)\right)\vartheta_1$.  

Lemma~\ref{lem:u_Lanczos} and the statement of
Remark~\ref{rem:polynomials} imply that $\vartheta_j$ and $q_j^{\xi}$
satisfy the following recursions, respectively (here $p_j^{\xi} \equiv
p_j^{\xi}\left(\xi\left(P\right)\right)$ and $q_j^{\xi} \equiv
q_j^{\xi}(x)$):
\begin{equation}\label{eqn:recursion_comparison}
    \begin{alignedat}{2}
        &\text{Set } \vartheta_0 = 0 \text{ and } \vartheta_1 
            = c^{-1/2}\ou_1 = p_1^{\xi}\vartheta_1. \qquad
    	&&\text{Set } q_0^{\xi} = 0 \text{ and } q^{\xi}_1 = 1. \\
        &\textbf{for } j = 1,\ldots,n \textbf{ do } \qquad
            &&\textbf{for } j = 1,\ldots,n \textbf{ do } \\
        &\quad 1.\ a_j^u = \oovs{p_j^{\xi}\vartheta_1}
            {\xi\left(P\right) p_j^{\xi}\vartheta_1}; \qquad
    	&&\quad 1.\ \alpha_i^{u} = \ipxit{q_i^{\xi}}{x q_i^{\xi}}; \\
        &\quad 2.\ r = \left[\left(\xi\left(P\right) 
            - a_j^uI\right)p_j^{\xi}
            -b_{j-1}^up_{j-1}^{\xi}\right]\vartheta_1; \qquad
	    &&\quad 2.\ r=\left[(x-\alpha_i^{u})q_i^{\xi} 
	    -\beta_{i-1}^{u}q_{i-1}^{\xi}\right]q_1; \\
        &\quad 3.\ b_j^u = \sqrt{\oovs{r}{r}}; \qquad
            &&\quad 3.\ \beta_i^{u} = \sqrt{\ipxit{r}{r}}; \\
        &\quad 4.\ \vartheta_{j+1} = \dfrac{r}{b_j^u} = 
            p_{j+1}^{\xi}\vartheta_1. \qquad
    	&&\quad 4.\ q_{i+1}^{\xi}=\dfrac{r}{\beta_i^{u}}. \\
        &\textbf{end for} \qquad
        &&\textbf{end for}
    \end{alignedat}
\end{equation}
Because $p_1^{\xi} \equiv 1$ and $q_1^{\xi} \equiv 1$, the above
recursions imply $p_j^{\xi} = q_j^{\xi}$ if $a_j^u = \alpha_j^u$ ($j =
1,\ldots,n$) and $b_j^u = \beta_j^u$ ($j = 1,\ldots,n-1$).  Before
proving this, we note the above recursions imply $p_j^{\xi}$ and
$q_j^{\xi}$ are polynomials of degree $j-1$.  

Mimicking the derivation of \eqref{eqn:eta} in
\S~\ref{subsec:derivation_of_eta}, we find
\begin{equation}\label{eqn:GQ_again}
    \oovs{f(\xi(P))\vartheta_1}{g(\xi(P))\vartheta_1} = 
	\int_{-1}^1 (f\circ\xi)(\mu) (g\circ\xi)(\mu)
	\frac{\eta_0(\mu)}{c} \di{\mu}.
\end{equation}
If $f(\xi)$ and $g(\xi)$ are both polynomials of degree less than or
equal to $n-1$, then $(f\circ\xi)(\mu)$ and $(g\circ\xi)(\mu)$ are both
polynomials of degree less than or equal to $n-1$ with respect to the
independent variable $\mu$ (since $\xi(\mu) = -\frac{2}{\tau^2}(1-\mu)$
is linear in $\mu$); thus $[(f\circ\xi)(g\circ\xi)](\mu)$ is a
polynomial of degree less than or equal to $2n-2$, so the Gaussian
quadrature from \S~\ref{subsec:fdr} computes the integral in
\eqref{eqn:GQ_again} exactly.  In particular, this implies that the
inner products in the recursion on the left-hand side of
\eqref{eqn:recursion_comparison} may be replaced by the Gaussian
quadrature rule, i.e., 
\[
    a_j^u = \oovs{p_j^{\xi}(\xi)\vartheta_1}
            {\xi p_j^{\xi}(\xi)\vartheta_1}
	   = \ipxit{p_j^{\xi}}{\xi p_j^{\xi}}
	   = \frac{1}{c}\sum_{j=1}^n y_j^2 p_j^{\xi}(\xi(\theta_j))
	   \xi(\theta_j) p_j^{\xi}(\xi(\theta_j))
\]
and similarly for $b_j^u$.  Because both recursions in
\eqref{eqn:recursion_comparison} have the same initialization, a
standard induction argument shows $a_j^{u} = \alpha_j^u$ for $j =
1,\ldots,n$ and $b_j^u = \beta_j^u$ for $j = 1,\ldots,n-1$.  As stated
above, this implies $p_j^{\xi} = q_j^{\xi}$ for $j = 1,\ldots,n$.


\subsection{Proof of Proposition~\ref{prop:GS}}
\label{subsec:proof_of_prop_GS}
We will prove the proposition for the dual snapshots; the proof for
the primary snapshots is similar.  

The proof is by induction.  If $j = 1$, then, according to
\eqref{eqn:w_GS}, $\owgs_1 = \tw_0$; on the other hand, thanks to
Algorithm~\ref{alg:algorithm_2} and Lemma~\ref{lem:leapfrog}, 
$\ow_1 = \ghat_1\frac{1}{\tv}\Lt\tu_0 = \frac{2\ghat_1}{\tau}\tw_0 =
d_1^w\owgs_1$.

Next, suppose for induction that $\ow_i = d_i^w\owgs_i$ for 
$i = 1,\ldots,j-1$ and define
\[
    s_i \equiv \frac{\owgs_i}{\normiptv{\owgs_i}}.
\]
Then $\ow_j$ and $\owgs_j$ are in $\mathrm{span}
\{s_1,\ldots,s_{j-1},\tw_{j-1}\}$, so 
\begin{equation}\label{eqn:owgs_int}
    \ow_j - \owgs_j = \ds\sum_{i=1}^{j-1} \rho_i s_i +
    \rho_j\tw_{j-1}
\end{equation}
for some coefficients $\rho_i$.  We take the inner product of both
sides of the above equation with $s_k$ for $k = 1,\ldots,j-1$ and
use the fact that $\iptv{\ow_j}{s_i} = \iptv{\owgs_j}{s_i} = 0$ for
$i = 1,\ldots,j-1$ to find 
\[
    0 = \iptv{\ow_j-\owgs_j}{s_k} = 
    \rho_k + \rho_j\iptv{\tw_{j-1}}{s_k}.
\]
Substituting this into \eqref{eqn:owgs_int} gives
\begin{equation}\label{eqn:Lanczos2GS}
    \ow_j = \owgs_j - \ds\sum_{i=1}^{j-1} \rho_j 
    \iptv{\tw_{j-1}}{s_i}s_i + \rho_j \tw_{j-1} = 
        (1+\rho_j)\owgs_j.
\end{equation}

Next, by \eqref{eqn:w_GS}, Lemma~\ref{lem:snapshots_uw},
\eqref{eqn:Lanczos2GS}, and Lemma~\ref{lem:poly_alg_2}, we have
\begin{equation}\label{eqn:owgs_int_2}
    \owgs_j = Q_j^w\left(\xi\left(\ti{P}_C\right)\right)\ow_1 
    = (1+\rho_j)^{-1}\ow_j 
    = (1+\rho_j)^{-1}q_j^w\left(\xi\left(\ti{P}_C\right)\right)
    \ow_1,
\end{equation}
where
\begin{equation}\label{eqn:Q_j}
    Q_j^w\left(\xi\left(\ti{P}_C\right)\right) \equiv
\frac{\tau}{2\ghat_1} \left[\Ts_{j-1}\left(\ti{P}_C\right) +
\Ts_{j-2}\left(\ti{P}_C\right)\right] - \sum_{i=1}^{j-1}
c_{ij}q_i^w\left(\xi\left(\ti{P}_C\right)\right)
\end{equation}
and, by the induction hypothesis,
\[
    c_{ij} \equiv \iptv{\tw_{j-1}}{\dfrac{\ow_i}{\normiptv{\ow_i}}}
\dfrac{1}{\normiptv{\ow_i}}.
\]
Recall $\xi\left(\ti{P}_C\right) = 0$ if and only if $\ti{P}_C = I$.
Then \eqref{eqn:Q_j}, standard results about Chebyshev polynomials,
Lemma~\ref{lem:poly_alg_2}, and \eqref{eqn:owgs_int_2} imply
\[
    Q_j^w(0) = \dfrac{\tau}{2\ghat_1}(2j-1) -
        \ds\sum_{i=1}^{j-1}c_{ij} \left(\dfrac{1}{\ghat_1}
        \ds\sum_{k=1}^i\ghat_i\right)
        = (1+\rho_j)^{-1}q_j^w(0)
        = (1+\rho_j)^{-1}\left(\dfrac{1}{\ghat_1}
        \ds\sum_{i=1}^j\ghat_i\right).
\]
The conclusion of the proposition follows by taking $d_j^w = 1+\rho_j$.  


\subsection{Proof of Lemma~\ref{lem:data_lemma}}
\label{subsec:lem2_proof}

To show that $u_k = UT_k(\la{H})\la{e}_1$, it suffices to demonstrate
\begin{equation}\label{eqn:Tke1}
    T_k(\la{H})\la{e}_1 = \la{e}_{k+1} \eqfor k = 0,\ldots,n-1,
\end{equation}
because $u_k = U\la{e}_{k+1}$.

We prove \eqref{eqn:Tke1} by induction.  Since we use the Chebyshev
three-term recursion formula
\begin{equation}\label{eqn:Chebyshev_recurrence}
    T_{k+1}(\la{H}) = 2\la{H}T_k(\la{H}) - T_{k-1}(\la{H}),
\end{equation}
the base of induction consists of the two cases $k = 0$, $1$.  

The case $k = 0$ is trivial:
\begin{equation*}
    T_0(\la{H})\la{e}_1 = \la{I}\la{e}_1 = \la{e}_1.
\end{equation*}

For the case $k = 1$ we observe from \eqref{eqn:snapshots_def} that $u_1
= \cos(\tau\sqrt{A})u_0 = Pu_0$, so 
\begin{equation*}
    \begin{aligned}
        T_1(\la{H})\la{e}_1 = \la{H}\la{e}_1 
        &= (U^*U)^{-1}U^*PU\la{e}_1 = (U^*U)^{-1}U^*Pu_0 \\
	    &= (U^*U)^{-1}U^*u_1 = (U^*U)^{-1}U^*U\la{e}_2 = \la{e}_2.
    \end{aligned}
\end{equation*}

For the induction step we use the trigonometric identity 
\begin{equation}\label{eqn:trig_identity}
    \begin{aligned}
        Pu_k &= \cos\left(\tau\sqrt{A}\right)
	    \cos\left(k\tau\sqrt{A}\right)u_0\\ 
	     &= \frac{1}{2}\left[\cos\left((k+1)\tau\sqrt{A}\right) +
                \cos\left((k-1)\tau\sqrt{A}\right)\right]u_0 \\
	     &= \frac{1}{2}\left(u_{k+1}+u_{k-1}\right),
    \end{aligned}
\end{equation}
where the first and last equalities follow from
\eqref{eqn:snapshots_def}.  Then the induction hypotheses are
$T_k(\la{H})\la{e}_1 = \la{e}_{k+1}$ and $T_{k-1}(\la{H})\la{e}_1 =
\la{e}_{k}$, which in conjunction with
\eqref{eqn:Chebyshev_recurrence}--\eqref{eqn:trig_identity} imply, for
$k = 0,\ldots,n-2$, that
\begin{equation*}
    \begin{aligned}
    T_{k+1}(\la{H})\la{e}_1 &= 2\la{H}T_k(\la{H})\la{e}_1 - 
        T_{k-1}(\la{H})\la{e}_1 \\
	&= 2\la{H}\la{e}_{k+1} - \la{e}_{k} \\
	&= 2(U^*U)^{-1}U^*PU\la{e}_{k+1}-\la{e}_k \\
	&= 2(U^*U)^{-1}U^*Pu_k - \la{e}_k \\
	&= (U^*U)^{-1}U^*(u_{k+1}+u_{k-1}) - \la{e}_k \\
	&= (U^*U)^{-1}U^*(U\la{e}_{k+2}+U\la{e}_k) - \la{e}_k \\
	&= (\la{e}_{k+2}+\la{e}_k) - \la{e}_k = \la{e}_{k+2}.
    \end{aligned}
\end{equation*}

For $k = 0,\ldots, n-1$, the formula for $f_k$ is an immediate
consequence of \eqref{eqn:f_k}, the fact that $u_0 = U\la{e}_1$, and the
first part of this lemma:
\begin{equation*}
    f_k = u_0^*u_k 
        = (U\la{e}_1)^*UT_k(\la{H})\la{e}_1 
        = \la{e}_1^T(U^*U)T_k(\la{H})\la{e}_1.
\end{equation*}

The proof that \eqref{eqn:match} holds for $k = n,\ldots,2n-1$ is more
subtle.  First, we define the operator 
\[
    \ha{H} \equiv U(U^*U)^{-1}U^*P,
\]
so $U\la{H} = \ha{H}U$.  In fact, if $g$ is a polynomial, we have
$Ug(\la{H}) = g(\ha{H})U$.  Moreover, the operator $\ha{H}$ is
self adjoint with respect to the inner product $\llangle \cdot, \cdot
\rrangle$.  

Next, we note that
\begin{equation}\label{eqn:Chebyshev_relationship}
    T_{n+j}(x) = T_{j+1}^{(2)}(x)T_{n-1}(x) -
        T_j^{(2)}(x)T_{n-2}(x)
\end{equation}
for all $j \ge 0$, where $T_j^{(2)}$ is the $j\textsuperscript{th}$
Chebyshev polynomial of the second kind (the identity
\eqref{eqn:Chebyshev_relationship} can be proved by induction on $j$).  
Then 
\begin{align*}
    \la{e}_1^TU^*UT_{n+j}(\la{H})\la{e}_1 &= 
        \left\llangle U\la{e}_1, UT_{n+j}(\la{H})\la{e}_1\right\rrangle
	\\
    &= \left\llangle U\la{e}_1, T_{n+j}(\ha{H})U\la{e}_1\right\rrangle\\
    &= \left\llangle U\la{e}_1, \left[T_{j+1}^{(2)}(\ha{H})
        T_{n-1}(\ha{H}) -
        T_j^{(2)}(\ha{H})T_{n-2}(\ha{H})\right]U\la{e}_1\right\rrangle\\
    &= \left\llangle T_{j+1}^{(2)}(\ha{H})U\la{e}_1,
        T_{n-1}(\ha{H})U\la{e}_1\right\rrangle -
	\left\llangle T_j^{(2)}(\ha{H})U\la{e}_1, 
	T_{n-2}(\ha{H})U\la{e}_1\right\rrangle\\
    &= \left\llangle UT_{j+1}^{(2)}(\la{H})\la{e}_1,
	UT_{n-1}(\la{H})\la{e}_1\right\rrangle -
	\left\llangle UT_j^{(2)}(\la{H})\la{e}_1, 
	UT_{n-2}(\la{H})\la{e}_1\right\rrangle.
	\numberthis \label{eqn:another_intermediate}
\end{align*}
Using the identities \eqref{eqn:Chebyshev_identities}
and the fact that $T_k(\la{H})\la{e}_1 = \la{e}_{k+1}$ for $k =
0,\ldots,n-1$, we find 
\[
    UT_j^{(2)}(\la{H})\la{e}_1 = T_j^{(2)}(P)U\la{e}_1
\]
for $j = 0,\ldots,n-1$.  
Using this, the fact that $P$ is self adjoint with respect to
$\left\llangle \cdot, \cdot \right\rrangle$, and
\eqref{eqn:Chebyshev_relationship} in
\eqref{eqn:another_intermediate} gives
\begin{align*}
    \la{e}_1^TU^*UT_{n+j}(\la{H})\la{e}_1 &=
    \left\llangle T_{j+1}^{(2)}(P)U\la{e}_1, U\la{e}_n \right\rrangle -
	\left\llangle T_j^{(2)}(P)U\la{e}_1, U\la{e}_{n-1}\right\rrangle
	\\
    &= \left\llangle U\la{e}_1, T_{j+1}^{(2)}(P)u_{n-1}\right\rrangle -
	\left\llangle U\la{e}_1, T_j^{(2)}(P)u_{n-2}\right\rrangle \\
    &= \left\llangle u_0, \left[T_{j+1}^{(2)}(P)T_{n-1}(P) - 
	T_j^{(2)}(P)T_{n-2}(P)\right]u_0\right\rrangle
	\\
    &= u_0^*T_{n+j}(P)u_0 \\
    &= u_0^*u_{n+j} \\
    &= f_{n+j}
\end{align*}
for $j = 0,\ldots,n-1$.  


\subsection{Proof of
Lemma~\ref{lem:specH}}\label{subsec:lem_specH_proof}

Since $\eta_0$ satisfies the hypothesis of
Lemma~\ref{lem:points_of_increase}, $U$ is of full rank; thus $U^*U \in
\mathbb{R}^{n\times n}$ is a symmetric, positive-definite matrix.  Let
$\la{x}$, $\la{z} \in \mathbb{R}^n$.  Then, since $U^*PU \in
\mathbb{R}^{n\times n}$ is symmetric, we have
\begin{align*}
    \usuh{\la{H}\la{x}}{\la{z}} 
    &= \ltrn{(U^*U)^{1/2}\la{H}\la{x}}{(U^*U)^{1/2}\la{z}}\\
    &= \ltrn{(U^*U)^{-1/2}(U^*PU)\la{x}}{(U^*U)^{1/2}\la{z}}\\
    &= \ltrn{(U^*PU)\la{x}}{\la{z}} \\
    &= \ltrn{\la{x}}{(U^*PU)\la{z}} \\
    &= \ltrn{(U^*U)^{1/2}\la{x}}{(U^*U)^{-1/2}(U^*PU)\la{z}} \\
    &= \usuh{\la{x}}{(U^*U)^{-1}(U^*PU)\la{z}} \\
    &= \usuh{\la{x}}{\la{H}\la{z}};
\end{align*}
thus $\la{H}$ is self adjoint with respect to $\usuh{\cdot}{\cdot}$.  

Next, we symmetrize $\la{H}$ by defining 
\begin{equation}\label{eqn:H_tilde}
    \ti{\la{H}} \equiv (U^*U)^{1/2}\la{H}(U^*U)^{-1/2} 
        = (U^*U)^{-1/2}(U^*PU)(U^*U)^{-1/2} = \ti{\la{H}}^T.
\end{equation}
Because $\ti{\la{H}}$ is symmetric, it can be orthogonally diagonalized
as 
\begin{equation}\label{eqn:eigen_H_tilde}
    \ti{\la{H}} = \ti{\la{\Phi}}\ti{\la{\Theta}}\ti{\la{\Phi}}^T, 
    \eqand{where}
    \ti{\la{\Phi}}^T\ti{\la{\Phi}} = \la{I}_{n\times n}
\end{equation}
and $\la{\Theta}$ is a diagonal matrix of the eigenvalues of $\la{H}$
(which are the same as those of $\ti{\la{H}}$ since $\la{H}$ and
$\ti{\la{H}}$ are similar).  If we define $\la{\Phi} \equiv
(U^*U)^{-1/2}\ti{\la{\Phi}}$, then \eqref{eqn:H_tilde} and
\eqref{eqn:eigen_H_tilde} imply
\begin{equation*}
    \la{H} = \la{\Phi}\la{\Theta}\la{\Phi}^T(U^*U),
    \eqand{where}
    \la{\Phi}^T(U^*U)\la{\Phi} = \la{I}_{n\times n}.
\end{equation*}


\subsection{Proof of Lemma~\ref{lem:TplusH}}\label{subsec:lem3_proof}

In order to compute $U^*PU$ and $U^*U$ we will need the inner products
of the snapshots.  Using \eqref{eqn:snapshots_def}, the fact that $A$
(and functions of $A$) are self adjoint with respect to
$\llrr{\cdot}{\cdot}$, and the fact that functions of $A$ commute, we
find, for $j$, $k = 0,\ldots,n-1$, that
\begin{equation}\label{eqn:snapshot_inner_products}
    \begin{aligned}
        \llangle u_j, u_k \rrangle 
	&= \llrr{v(0)\cos\left(j\tau\sqrt{A}\right)
	    \tq(A)^{1/2}\delta} 
	    {v(0)\cos\left(k\tau\sqrt{A}\right)\tq(A)^{1/2}\delta}
	    \\
	&= \left\langle \delta, \cos\left(j\tau\sqrt{A}\right)
	    \cos\left(k\tau\sqrt{A}\right)\tq(A)\delta\right\rangle.
    \end{aligned}
\end{equation}
Applying the trigonometric identity
\begin{equation*}
    \cos\left(j\tau\sqrt{A}\right)\cos\left(k\tau\sqrt{A}\right) = 
       \frac{1}{2}\left[\cos\left((j+k)\tau\sqrt{A}\right) +
        \cos\left((j-k)\tau\sqrt{A}\right)\right]
\end{equation*}
to \eqref{eqn:snapshot_inner_products} we obtain
\begin{equation*}
    \begin{aligned}
        \llangle u_j, u_k \rrangle 
	&= \frac{1}{2}\left[\left\langle \delta,
	    \cos\left((j+k)\tau\sqrt{A}\right)\tq(A)\delta\right\rangle+
	    \left\langle \delta, \cos\left((j-k)\tau\sqrt{A}\right)
	    \tq(A)\delta \right\rangle\right] \\
	&= \frac{1}{2}\left(\left\langle \delta, u_{j+k}\right\rangle + 
	    \left\langle \delta, u_{j-k}\right\rangle\right) \\
	&= \frac{1}{2}\left[u_{j+k}(0) + u_{j-k}(0)\right],
    \end{aligned}
\end{equation*}
where the snapshots with negative indices are defined
using the evenness of cosine, i.e., we take $u_l(x) \equiv 
u_{-l}(x)$ for $l < 0$.  Thus
\begin{equation}\label{eqn:snapshot_inner_products_result}
    \llangle u_j, u_k \rrangle
        = \frac{1}{2}(f_{j+k} + f_{j-k}).
\end{equation}

Let us consider $U^*PU$ first.  Applying the formula in
\eqref{eqn:trig_identity} to $PU$, we get
\begin{equation}\label{eqn:UsPU_intermediate}
    U^*PU = \frac{1}{2}U^*\left(\left[u_{-1},u_0\ldots,u_{n-2}\right] 
        + \left[u_1,u_2\ldots,u_{n}\right]\right).
\end{equation}

Using the inner product formula
\eqref{eqn:snapshot_inner_products_result}, the first product on the
right-hand side of \eqref{eqn:UsPU_intermediate} becomes
\begin{align*}
    &U^*[u_{-1},u_0\ldots,u_{n-2}] = \\
    &\frac{1}{2}
    \begin{bmatrix}
        f_1+f_{-1} & f_0+f_0 & f_{-1}+f_1 & f_{-2}+f_2 & \hdots &
            f_{-n+2}+f_{n-2} \\
        f_2+f_0 & f_1+f_1 & f_0+f_2 & f_{-1}+f_3 & \hdots & 
            f_{-n+3} + f_{n-1} \\
        f_3+f_1 & f_2+f_2 & f_1+f_3 & f_0+f_4 & \hdots & 
            f_{-n+4}+f_n \\
        \vdots & \vdots & \vdots & \vdots & \ddots & \vdots \\
        f_n+f_{n-2} & f_{n-1}+f_{n-1} & f_{n-2}+f_n &f_{n-3}+f_{n+1} 
            & \hdots & f_1+f_{2n-3} \\
    \end{bmatrix}.\numberthis\label{eqn:first_product}
\end{align*}
Similarly, for the second product in \eqref{eqn:UsPU_intermediate} we
have 
\begin{align*}
    &U^*[u_1,u_2\ldots,u_n] = \\
    &\frac{1}{2}
    \begin{bmatrix}
        f_{-1}+f_1 & f_{-2}+f_2 & f_{-3}+f_3 & f_{-4}+f_4 & \hdots &
            f_{-n}+f_n \\
        f_0+f_2 & f_{-1}+f_3 & f_{-2}+f_4 & f_{-3}+f_5 & \hdots & 
            f_{-n+1} + f_{n+1} \\
        f_1+f_3 & f_0+f_4 & f_{-1}+f_5 & f_{-2}+f_6 & \hdots & 
	    f_{-n+2}+f_{n+2} \\
        \vdots & \vdots & \vdots & \vdots & \ddots & \vdots \\
        f_{n-2}+f_n & f_{n-3}+f_{n+1} & f_{n-4}+f_{n+2} & 
	    f_{n-5}+f_{n+3} & \hdots & f_{-1}+f_{2n-1} \\
    \end{bmatrix}.\numberthis\label{eqn:second_product}
\end{align*}
The same inner product formula applied to $U^*U$ yields
\begin{align*}
    &U^*U = \\
    &\frac{1}{2}
    \begin{bmatrix}
        f_0+f_0 & f_{-1}+f_1 & f_{-2}+f_2 & f_{-3}+f_3 & \hdots &
            f_{-n+1}+f_{n-1} \\
        f_1+f_1 & f_0+f_2 & f_{-1}+f_3 & f_{-2}+f_4 & \hdots & 
            f_{-n+2}+f_n \\
        f_2+f_2 & f_1+f_3 & f_0+f_4 & f_{-1}+f_5 & \hdots & 
	    f_{-n+3}+f_{n+1} \\
        \vdots & \vdots & \vdots & \vdots & \ddots & \vdots \\
        f_{n-1}+f_{n-1} & f_{n-2}+f_n & f_{n-3}+f_{n+1} & 
	    f_{n-4}+f_{n+2} & \hdots & f_0+f_{2n-2} \\
    \end{bmatrix}.\numberthis\label{eqn:third_product}
\end{align*}

Finally, using the evenness of the cosine (i.e., $f_l = f_{-l}$ for $l <
0$), we observe that each of
\eqref{eqn:first_product}--\eqref{eqn:third_product} can be expressed as
a sum of a Toeplitz matrix and a Hankel matrix:
\begin{equation*}
    \begin{aligned}
        &U^*\left[u_{-1},u_0,\ldots,u_{n-2}\right] 
	    =\frac{1}{2}\left(\la{T}^+ + \la{H}^-\right),\\
        &U^*\left[u_1,u_2,\ldots,u_n\right] 
	    = \frac{1}{2}\left(\la{T}^- +\la{H}^+\right),\\
	&U^*U = \frac{1}{2}\left(\la{T}^0+ \la{H}^0\right).
    \end{aligned}
\end{equation*}


\subsection{Derivation of Algorithm~\ref{alg:algorithm_1}}
\label{subsec:algorithm_1_derivation}

Let $\left(-\lambda_l, \br_l\right)$, $l =
1,\ldots,n$, be an eigenpair of $\bM$, i.e.,
\begin{equation}\label{eqn:M_eig}
    \bM\br_l + \lambda_l\br_l = 0.
\end{equation}
Since $\bM$ is similar to $\xi\left(\la{P}_n\right)$,
Lemma~\ref{lem:points_of_increase} implies $-\lambda_l = \xi(\theta_l)
\in \left[-\frac{2}{\tau^2},0\right]$; 
Lemma~\ref{lem:points_of_increase} also implies the eigenvalues
$\lambda_l$ are distinct.

We introduce the auxiliary variables
\begin{equation}\label{eqn:s}
    s_{l,j} \equiv \dfrac{r_{l,j+1}-r_{l,j}}
        {\sqrt{\lambda_l}\g_j} 
    \eqand{and}
    \ha{s}_{l,j} \equiv \dfrac{r_{l,j+1}-r_{l,j}}
        {-\sqrt{\lambda_l}\g_j}
    \eqfor l = 1,\ldots,n, \quad j = 1,\ldots,n.
\end{equation}
Let 
\begin{equation}\label{eqn:bg}
    \bg_l \equiv \delta\left[r_{l,1}, s_{l,1}, \ldots, r_{l,n}, 
        s_{l,n}\right]^T 
    \eqand{and} 
    \hbg_l \equiv \delta\left[r_{l,1}, \ha{s}_{l,1}, \ldots, r_{l,n}, 
        \ha{s}_{l,n}\right]^T, 
\end{equation}	
where $\delta$ is a constant we will determine later.  We also define
\[
    \la{O} \equiv 
    \begin{bmatrix}
        0    & 1 &        &        &        &   \\
	1    & 0 & -1     &        &        &   \\
	& -1 & 0 & 1      &        &        &   \\
	&    & 1 & 0      & -1     &        &   \\
	&    &   & \ddots & \ddots & \ddots &   \\
	&    &   &        & -1     & 0      & 1 \\
	&    &   &        &        & 1      & 0 
    \end{bmatrix}
    = \la{O}^T
    \in \mathbb{R}^{2n\times 2n}
    \eqand{and}
    \la{T} \equiv \la{O}\la{\Gamma}^{-1},
\]
where $\la{\Gamma}$ is defined in \eqref{eqn:OGamma}.  Then, in
combination with \eqref{eqn:s}, \eqref{eqn:M_eig} may be written in
first-order form as
\begin{equation}\label{eqn:crux_alg_1}
    \la{L}\la{Q} = \la{Q}\la{T},
\end{equation}
where 
\begin{equation}\label{eqn:L_bold}
    \la{L} \equiv \diag\left(\sqrt{\lambda_1}, -\sqrt{\lambda_1},
    \sqrt{\lambda_2}, -\sqrt{\lambda_2}, \ldots, \sqrt{\lambda_n},
    -\sqrt{\lambda_n}\right),
\end{equation}
and 
\begin{equation}\label{eqn:Q_def}
    \la{Q} \equiv 
    \begin{bmatrix}
        \td & \bg_1^T  & \td \\
	\td & \hbg_1^T & \td \\
	    & \vdots   & \td \\
	\td & \bg_n^T  & \td \\
	\td & \hbg_n^T & \td 
    \end{bmatrix}
    \equiv
    \begin{bmatrix}
           |    &      |     &        &    |    &     |      \\
	\obmu_1 & \obomega_1 & \cdots & \obmu_n & \obomega_n \\
           |    &      |     &        &    |    &     |      \\
    \end{bmatrix} 
    \in \mathbb{R}^{2n\times 2n}.
\end{equation}
Note that \eqref{eqn:crux_alg_1} is an eigendecomposition of $\la{T}^T =
\la{\Gamma}^{-1}\la{O}$, i.e., $\la{T}^T \la{Q}^T = \la{Q}^T \la{L}$.
This may be written in a different basis as 
\begin{equation}\label{eqn:symmetric_eig}
    \ti{\la{T}}^T\la{\Gamma}^{1/2}\la{Q}^T 
    = \la{\Gamma}^{1/2}\la{Q}^T\la{L}, 
    \eqand{where} 
    \ti{\la{T}}^T \equiv \la{\Gamma}^{1/2}\la{T}^T\la{\Gamma}^{-1/2} 
    = \la{\Gamma}^{-1/2}\la{O}\la{\Gamma}^{-1/2}.
\end{equation}
Since $\ti{\la{T}}^T$ is symmetric and we are assuming all eigenvectors
of symmetric matrices are normalized with Euclidean norm $1$, we have
\[
    \la{I}_{2n\times 2n} 
    = \left(\la{\Gamma}^{1/2}\la{Q}^T\right)^T\la{\Gamma}^{1/2}\la{Q}^T
    = \la{Q}\la{\Gamma}\la{Q}^T;
\]
this implies 
\begin{equation}\label{eqn:normalization}
    \la{Q}^T\la{Q} = \la{\Gamma}^{-1}.
\end{equation}
The algorithm is essentially given by \eqref{eqn:crux_alg_1} and
\eqref{eqn:normalization}; all that remains is for us to initialize the
algorithm appropriately, i.e., we need to compute
\begin{equation}\label{eqn:obmu_1}
    \obmu_1 = \delta[r_{1,1}, r_{1,1}, r_{2,1}, r_{2,1}, \ldots, 
        r_{n,1}, r_{n,1}]^T. 
\end{equation}

We begin by determining the constant $\delta$ from \eqref{eqn:bg}.
First, by \eqref{eqn:Q_def}--\eqref{eqn:symmetric_eig},
\begin{equation}\label{eqn:delta}
    1 = \left\langle\la{\Gamma}^{1/2}\bg_l, \la{\Gamma}^{1/2}\bg_l
        \right\rangle_{l^2\left(\mathbb{R}^{2n}\right)}
    = \delta^2\left(\ltrn{\bDhat^{1/2}\br_l}{\bDhat^{1/2}\br_l} +
	\ltrn{\la{D}^{1/2}\la{s}_l}{\la{D}^{1/2}\la{s}_l}\right)
    = 2\delta^2,
\end{equation}
where $\la{D} \equiv \diag(\g_1,\ldots,\g_n)$.  The last inequality
above holds because $\bDhat^{1/2}\br_l$ is an eigenvector of the
symmetric matrix $\bMt$; similarly, by eliminating $\tmu_{j,k}$ from the
recursion \eqref{eqn:Galerkin_recursion}, it can be shown that
$\left(-\lambda_l, \la{D}^{1/2}\la{s}_l\right)$ is an eigenpair of the
symmetric matrix $\ti{\la{N}} \equiv \la{R}\la{R}^T$, where $-\bMt =
\la{R}^T\la{R}$ with $\la{R}$ upper triangular is the Cholesky
decomposition of $-\bMt$.  (The previous analysis also holds with
$\bg_l$ replaced by $\hbg_l$.)

Next, $\bDhat^{1/2}\br_l$ and $\la{X}_l$ are normalized eigenvectors of
$\bMt = \xi\left(\la{P}_n\right)$ by \eqref{eqn:M_eig} and
\eqref{eqn:P_n_eig}, respectively.  Thus \eqref{eqn:X_j_1}, the fact
that the eigenvalues $-\lambda_l$ are distinct (by
Lemma~\ref{lem:points_of_increase}), and \eqref{eqn:ghat_1} imply
\begin{equation}\label{eqn:r_m1}
    y_l^2/c = \left(\la{e}_1^T\la{X}_l\right)^2
    = \left(\la{e}_1^T\bDhat^{1/2}\br_l\right)^2
    = \ghat_1r_{l,1}^2 \Rightarrow r_{l,1} = y_l.
\end{equation}
Then \eqref{eqn:crux_alg_1}--\eqref{eqn:Q_def} and 
\eqref{eqn:normalization}--\eqref{eqn:r_m1} give us the algorithm for
computing $\ghat_j$ and $\g_j$, which we summarize in
Algorithm~\ref{alg:algorithm_1}.  Algorithm~\ref{alg:algorithm_1} is
isomorphic to Algorithm~\ref{alg:algorithm_2}; this is due to the close
relationship between \eqref{eqn:crux_alg_1} and \eqref{eqn:crux_alg_2}.


\bibliography{KGL_SIAM_updated_arxiv}
\bibliographystyle{siam}

\end{document}